\documentclass[12pt,reqno]{amsart}

\usepackage[usenames, dvipsnames]{color}

\usepackage{epsfig,amsmath,amsthm}
\usepackage{nccmath}
\makeatletter
\renewcommand{\paragraph}{%
  \@startsection{paragraph}{4}{\z@}%
  {1.2ex \@plus .2ex}%
  {-1em}%
  {\normalfont\normalsize\bfseries}}
\makeatother
\usepackage[foot]{amsaddr}
\usepackage{graphicx}
\usepackage{amssymb}
\usepackage{enumerate}
\usepackage{comment}

\usepackage{paralist}
\usepackage[left=2cm,right=2cm,top=2cm,bottom=2cm]{geometry}
\usepackage{fancyhdr}
\usepackage{hyperref}
\hypersetup{
	colorlinks=true,
	linkcolor=blue,
	citecolor=red,
	urlcolor=blue,
	pdfborder={0 0 0}
}
\DeclareMathSizes{10}{10}{6}{4}

\usepackage{cleveref}
\crefname{theorem}{Theorem}{Theorems}
\crefname{thm}{Theorem}{Theorems}
\crefname{lemma}{Lemma}{Lemmas}
\crefname{lem}{Lemma}{Lemmas}
\crefname{remark}{Remark}{Remarks}
\crefname{rmk}{Remark}{Remarks}
\crefname{prop}{Proposition}{Propositions}
\crefname{notation}{Notation}{Notations}
\crefname{claim}{Claim}{Claims}
\crefname{defn}{Definition}{Definitions}
\crefname{definition}{Definition}{Definitions}
\crefname{cor}{Corollary}{Corollaries}
\crefname{example}{Example}{Examples}
\crefname{section}{Section}{Sections}
\crefname{figure}{Figure}{Figures}
\crefname{assumption}{Assumption}{Assumptions}

\newtheorem{theorem}{Theorem}[section]

\newtheorem{claim}[theorem]{Claim}
\newtheorem{lemma}[theorem]{Lemma}
\newtheorem{cor}[theorem]{Corollary}
\newtheorem{conj}[theorem]{Conjecture}
\newtheorem{prop}[theorem]{Proposition}
\newtheorem{defn}[theorem]{Definition}

\newtheorem{remark}[theorem]{Remark}

\numberwithin{equation}{section}

\theoremstyle{definition}

\DeclareMathOperator{\CR}{CR}

\DeclareMathOperator{\CLE}{CLE}

\DeclareMathOperator{\indx}{ind}

\newcommand{\R}{\mathbb{R}}
\newcommand{\C}{\mathbb{C}}
\newcommand{\N}{\mathbb{N}}
\newcommand{\Z}{\mathbb{Z}}

\newcommand{\E}{\mathbb{E}}
\newcommand{\HH}{\mathbb{H}}

\newcommand{\eps}{\varepsilon}

\newcommand{\D}{\mathbb{D}}

\newcommand{\Q}{\mathbb{Q}}
\renewcommand{\P}{\mathbb{P}}

\newcommand{\1}{\mathbf{1}}
\renewcommand{\Re}{\mathrm{Re}}
\renewcommand{\Im}{\mathrm{Im}}

\newcommand{\Juk}{\mathcal{J}}
\newcommand{\Ftw}{\mathcal{F}}
\newcommand{\cbls}{\mathbf{c}_\D^{\rm{loop}}}
\newcommand{\cblsg}{\mathbf{c}_{\D,g_{\rm cone}}^{\rm{loop}}}
\newcommand{\fconf}{\hat{\mathbf{c}}}
\newcommand{\Ctw}{\frac{\pi^{3/4}}{2^{7/4}}}
\newcommand{\CSLE}{\mathbf{C}_{\rm{SLE}4}^{\rm loop}}

\def\cV{\mathcal{V}}

\def\cL{\mathcal{L}}

\def\cD{\mathcal{D}}
\def\cC{\mathcal{C}}

\newcommand{\Modd}{\mathcal{M}}
\newcommand{\ann}{\mathcal{A}}
\newcommand{\strip}{\mathcal{S}}
\newcommand{\ED}{\operatorname{ED}}
\newcommand{\Int}{\operatorname{Int}}
\newcommand{\Ext}{\operatorname{Ext}}

\begin{document}

\title{Renormalised two-point functions of CLE$_4$ gaskets}

\author{Juhan Aru$^\ast$}
\email{juhan.aru@epfl.ch}

\address{$^\ast$ EPFL, Institute of Mathematics, 1015 Lausanne, Switzerland}

\author{Titus Lupu$^\dagger$}
\email{titus.lupu@sorbonne-universite.fr}

\address{$^\dagger$ Sorbonne Université, Université Paris Cité, CNRS, Laboratoire de Probabilités, Statistique et Modélisation, LPSM, F-75005 Paris, France}

\thanks{J.A. is supported by Eccellenza grant 194648 of the Swiss National Science Foundation and he is a member of NCCR Swissmap. }

\begin{abstract}
We first consider nested CLE$_4$ in a simply-connected domain and compute some exact renormalised probabilities: the probability that two points belong to the same CLE$_4$ gasket and the probability that two points belong to the outermost CLE$_4$ gasket. The resulting conformally covariant formulas have a non-trivial modular dependence, expressed explicitly in terms of Jacobi theta functions.

We are guided by the conformal field theory of the Ashkin-Teller (AT) model, but our proofs are purely probabilistic using Brownian loop soups, gauge-twisted Gaussian free fields (GFFs) and the level line geometry of the 2D continuum GFF.
    
More generally, we also calculate renormalised probabilities that two points belong to CLE$_4$ gaskets sampled in alternation with certain two-valued sets of the Gaussian free field. These quantities correspond to the two-point function of the conjectured scaling limit of the AT single spins on the critical line. At the decoupling point, our results recover the Ising model correlations; they also suggest a CLE$_4$-based continuum FK representation of the AT single-spin fields along a whole segment of the critical line.
\end{abstract}
\maketitle
\tableofcontents
\section{Introduction and main results}

In the 1980s, conformal field theory (CFT) emerged as a powerful way to explain the universality of two-dimensional scaling limits of critical statistical physics models through algebraic and analytic structures \cite{BPZ}. In the 2000s, Schramm--Loewner evolutions (SLE$_\kappa$) and their loop versions, the Conformal Loop Ensembles CLE$_\kappa$, provided a geometric and probabilistic approach to the same universality phenomena \cite{SchrammLERW, LawlerSchrammWerner2003ConformalRestr,WernerstFlour, CardySLE}.

These two perspectives are expected to be deeply related: CLE$_\kappa$ should encode geometric aspects of CFTs with central charge $c = 1 - 6(\sqrt{\kappa}/2-2/\sqrt{\kappa})^2$ and also represent in some way the algebraic structures. See, for instance, the reviews \cite{CardySLE, PeltolaReview}. Despite substantial recent progress on the SLE/CLE--CFT correspondence, see e.g. \cite{JesperandCoo,BaverezJego24MKS,ang2024integrabilityconformalloopensemble, imaginaryliouville}, precise mathematical connections between CFT structures and CLE observables remain not very well understood. 

In this paper we connect a certain continuous family of \(c=1\) CFTs to CLE$_4$, by translating free field reprsentations of Ashkin--Teller (AT) model from \cite{ginsparg1988curiosities, saleur1988correlation, zamolodchikovAT} into exact probabilistic language of the level line geometry of the Gaussian free field  \cite{SchSh,SchSh2, ASW, ALS1}, its various relations to Brownian loop soups of critical intensity \cite{SheffieldWerner2012CLE, ALS2, ALS4} and crucially also the recent work on gauge-twisted Gaussian free fields \cite{KasselLevy16CovSym, LupuTopoAIHP}. As a result, we obtain explicit formulas for CLE$_4$ gasket correlations and explain how CLE$_4$ appears for a whole continuum range of $c=1$ CFTs on the AT critical line. Hence, we showcase how following ideas from the CFT and translating them to probabilistic language can provide new insights into well-studied probabilistic objects.  We plan to further zoom in on this link between the level line geometry and $c=1$ CFTs in subsequent works. \\

At large, our paper has three main results. First, we compute what is called the two-point twist correlation function in a simply-connected domain and connect it to certain renormalised probabilities of certain CLE$_4$ gasket events. This is the probabilistic counterpart of the \(\mathbb Z_2\)-orbifold twist field at \(c=1\) \cite{dixon1987orbifold}, and the answer is an explicit theta-function formula. Second, building on this we use the idea of compactification of the free field to obtain exact two-point functions for nested CLE$_4$ gaskets and outermost CLE$_4$ gaskets - these correspond to limiting probabilities of 4-Potts FK representation. To our knowledge, these are the first rigorous computations for correlation functions of CLE$_\kappa$ in the case of a non-trivial modular parameter (i.e. that would require conformal bootstrap in the CFT language). For comparison, the three point formulas on the sphere, where no modular parameter appears, were recently obtained in \cite{ang2024integrabilityconformalloopensemble} for all $\kappa \in (0, 8]$ using Liouville CFT. See also a recent physics paper \cite{downing2pt} for a  recent two-point calculation for loop ensembles in terms of conformal blocks \footnote{We have been communicated by Xin Sun that indeed our concrete expression matches their expression that is given in terms of conformal blocks}. Third, we generalise our technique to provide formulas along a segment of the critical line for the Ashkin--Teller two-point spin correlation limits and show how CLE$_4$ gaskets sampled at different heights naturally appear along a whole segment of the AT critical line - see also \cite{AlcadeHeeneyLis2026} for an alternative random current based way of seeing this surprising appearance.  \\

Let us now zoom in further on conformal loop ensembles and the mathematical statements of our results. The Conformal Loop Ensembles CLE$_{\kappa}$, $\kappa\in (8/3,8]$ form a one-parameter family of random collections of non-overlapping loops in a simply connected planar domain that are conformally invariant in law. They were introduced in 
\cite{SheffieldCLE09,SheffieldWerner2012CLE} as the conjectural scaling limits of the interfaces for various statistical physics models and these conjectures have been proved in the realm of latticfe models so far for percolation, Ising and FK-Ising ($\kappa = 6, 3, 16/3$) respectively. In our case of interest, $\kappa=4$, it is known that the CLE$_{4}$ appears in the scaling limit of critical random walk loop soup outer boundaries or the outer boundaries of outermost sign components of a metric graph GFF \cite{BrugCamiaLis2014RWLoopsCLE,LupuConvCLE} and thereby has a natural coupling to the 2D continuum Gaussian free field (GFF)
\cite{MS,WaWu,ASW}. Furthermore, the CLE$_4$ is conjectured to be scaling limit of loops 
in the double-dimer model 
\cite{Kenyon14DoubleDimers,Dubedat19DoubleDimers,
BasokChelkak18taufunc}
and in the loop $O(n)$ model with parameters 
$n=2, x=1/\sqrt{2}$
\cite{KagerNienhuis04,PeledSpinka17On}. As explained below it conjecturally also describes the 4-Potts model and its FK representation, and also appears in the FK representation of a single-spin in the Ashkin-Teller model on a whole segment of the critical line.

 To describe our results on the correlation functions of CLE$_4$ more precisely, recall that CLE comes in nested and non-nested (also called simple) versions - the former often corresponding to outer boundaries of outermost clusters in the corresponding statistical physics model, and the latter corresponding to all sign clusters themselves \cite{SheffieldCLE09}. The nested CLE$_4$ can be constructed iteratively from a simple CLE$_4$ by sampling first a simple CLE$_4$, then sampling independent CLE$_4$-s inside each of the loops, and then iterating this procedure independently inside each of the loops that have appeared. 
Conversely, given nested CLE$_4$, the simple version is obtained by restricting to outermost loops and their closure.

We express the different two-point correlation functions via the Jacobi theta functions
$\theta_2$ and $\theta_3$:
$$
\theta_2(q) = 
\sum_{n\in\Z} q^{(n+1/2)^2},
\qquad
\theta_3(q) = 
\sum_{n\in\Z} q^{n^2},
$$
and the usual Dirichlet Green's function $G_D(z_1,z_2)$ in the domain $D$ with the normalization $G_D(z_1, z_2) \sim \frac{-1}{2\pi}\log|z_1-z_2|$ near the diagonal.

\begin{theorem}[Two-point correlations of nested CLE$_4$]
\label{thm:main1}
Let $D$ be a simply-connected domain and $z_1 \neq z_2 \in D$. Let $E^{\rm nest}_\eps(z_1,z_2)$ denote the event that there is a nested CLE$_4$ gasket that comes $\eps_1$-close to both $z_1$ and $\eps_2$-close to $z_2$. Then we have
$$\lim_{\max(\eps_1, \eps_2) \to 0}\eps_1^{-1/8}\eps_2^{-1/8}\P(E^{\rm nest}_\eps(z_1,z_2)) = C_0\CR(z_1,D)^{-1/8}\CR(z_2,D)^{-1/8}\exp\left(\frac{\pi}{2}G_D(z_1,z_2)\right),$$
where $C_0$ is an absolute constant \footnote{Both here and in the next statements these absolute constants are explicit up to a term depending on SLE$_4$ loop measure (that we think does not have an explicit expression). This term would disappeaer when using conformal radii instead of actual radii, or when normalising the two-point function by product of one-point functions. We prefer to keep this dependence and remark here that some duality arguments are nicer using the Euclidean normalisation, although the conformal radii one would also follow from the same techniques.}.
Further, this two-point function also equals,
up to an absolute constant,
$$\Big\langle:\exp\Big(i\sqrt\frac{\pi}{2}\Phi_D(z_1)\Big):~
:\exp\Big(-i\sqrt\frac{\pi}{2}\Phi_D(z_2)\Big):
\Big\rangle .$$
Here the GFF $\Phi_D$ in $D$ is normalized to have covariance $G_D$, and we consider the usual imaginary multiplicative chaos.
\end{theorem}
Further, by restricting nested CLE$_4$ to only odd or only even gaskets, one obtains the correlations of 
$:\sin\big(\sqrt\frac{\pi}{2}\Phi_D\big):$ or 
$:\cos\big(\sqrt\frac{\pi}{2}\Phi_D\big):$, 
respectively. 
The explicit formulas for those can also be calculated and are given in Proposition \ref{prop:calc2ptsum}. 

More surprisingly, we also manage to obtain the two-point function for the outermost CLE$_4$ gasket.
\begin{theorem}[Two-point correlations of the simple CLE$_4$]
\label{thm:main-simple}
Let $D$ be a simply-connected domain and $z_1 \neq z_2 \in D$. Let $E^{\rm simp}_{\eps_1,\eps_2}(z_1,z_2)$ denote the event that the  CLE$_4$ gasket comes $\eps_1$-close to $z_1$ and $\eps_2$-close to $z_2$. Then we have
$$
\lim_{\max(\eps_1,\eps_2)\to 0}
\eps_1^{-1/8}\eps_2^{-1/8}
\P(E^{\rm simp}_{\eps_1,\eps_2}(z_1,z_2))
=
C_1\CR(z_1,D)^{-1/8}\CR(z_2,D)^{-1/8}
\sqrt{\frac{\theta_3(q^{1/4})}{\theta_2(q^{1/4})}},
$$
where $C_1$ is an absolute constant, and the nome $q$ is uniquely determined by
$$
\exp\!\left(\pi G_D(z_1,z_2)\right)
=
\frac{\theta_3(q^{1/2})}{\theta_2(q^{1/2})}.
$$
\end{theorem}
In fact there are formulas even for the renormalised probability that $z_1$, $z_2$ are on a $m$-th CLE$_4$-gasket, when counted from the boundary. For $m > 1$ this is again a combination of several geometric events, as there are countably many $m$-th level gaskets. 
See Theorem \ref{thm:kthgasket}.\\

\noindent There are several ways to interpret these results in terms of conjectured scaling limits:
\begin{itemize}
    \item The full correlations, i.e. the renormalised probabilities of being on the same CLE$_4$ gasket in Theorem \ref{thm:main1}, 
    should correspond to the renormalised spin-spin correlations of the XY-model at the Kosterlitz-Thouless point \cite{kadanoff1979multicritical}.
    \item Restricting to odd CLE$_4$ gaskets should give the scaling limit of the FK representation of the 4-Potts model with wired boundary conditions.
    \item Restricting to even CLE$_4$ gaskets should give the FK representation of the 4-Potts model with free boundary conditions. 
\end{itemize}

In the realm of relevant results for scaling limits we also obtain a geometric representation for the scaling limit of two-point spin-spin correlation in the critical Ising model again using CLE$_4$. In order to describe these results we recall the notation $A_{-a,b}$ for two-valued local sets of the Gaussian free field, introduced in \cite{ASW}. These can be seen as random gaskets, whose law agrees with the law of the CLE$_4$ gasket in the case $a = b = 2\lambda$, where $2\lambda$ is the so-called height gap, equal to $\sqrt{\pi/2}$ in our normalisation.

\begin{theorem}
\label{Thm TVS Ising}
Let $D$ be a simply connected domain and $z_1 \neq z_2 \in D$. Let $E^{\rm Ising}_{\eps_1, \eps_2}$ 
denote the event that one of the CLE$_4$ gaskets obtained from the following iteration comes $\eps_1$-close to $z_1$ and $\eps_2$-close to $z_2$.
\begin{enumerate}
    \item We sample a CLE$_4$ gasket in $D$ to obtain $A_1$; 
\item Inside each connected component of $D \setminus A_1$, we sample a random gasket with the law of the two-valued set $A_{-2\sqrt{2}\lambda + 2\lambda, 2\lambda}$ to obtain $A_2$;
\item we then again sample CLE$_4$ inside each connected component of $A_2$ and iterate this way: on odd steps we sample CLE$_4$ gaskets, on even $A_{-2\sqrt{2}\lambda + 2\lambda, 2\lambda}$ gaskets.
\end{enumerate}
Then the Ising spin two-point function in the unit disk is given by 
$$\lim_{\max(\eps_1,\eps_2) \to 0} (\eps_1\eps_2)^{-1/8}
\P(E^{\rm Ising}_{\eps_1, \eps_2}) = C_2\CR(z_1,D)^{-1/8}\CR(z_2,D)^{-1/8}\frac{\theta_3(q^{1/4})}{\theta_2(q^{1/4})},$$
where $C_2$ is an absolute constant, and the nome $q$ is fixed by $\exp(\pi G_D(z_1,z_2)) = \frac{\theta_3(q^{1/2})}{\theta_2(q^{1/2})}$.
\end{theorem}
Similar formulas can be obtained along the whole critical line of the Ashkin-Teller model - see Theorem \ref{thm:twopointAT}. They provide explicit calculations for the conjectured scaling limit of two-point spin-spin correlations, but also suggest CLE$_4$-based continuum FK representations for all these limiting fields, including a CLE$_4$-based FK representation of the continuum limit of the Ising spin field - see Conjecture \ref{conj:ATFK}. During the final phases of our work we learned that this conjecture has been independently put forward also in \cite{AlcadeHeeneyLis2026} using considerations coming from random current representations for the AT model - see \cite[Section 1.3]{AlcadeHeeneyLis2026}, where they provide a more detailed picture including both AT spins and both free and wired boundary conditions. Moreover, in that work, the authors actually prove the conjecture at a special point of the AT critical line corresponding to the XOR-Ising model, see \cite[Theorem 1.2, Theorem 1.6]{AlcadeHeeneyLis2026}. In particular, they thus manage to provide a CLE$_4$ based geometric representation of the continuum limit of the Ising spin model. Combining the two perspectives should open the road for further progress, and even more so when combined with the other interesting recent work on geometric representations of the XOR-spins \cite{TomasAvelio}. \\

A starting point for our calculations is a Brownian loop soup computation of what we call the twist correlation functions - see Theorem \ref{thm:twopointtwist}. 
This corresponds to the disk two-point version of the 
$\Z_2$-orbifold twist correlations at infinite compactification radius in \cite{dixon1987orbifold}. The one-point function is exactly determined using spectral zeta functions. These twist correlation functions alre correspond to the two point function in the unit disk of certain Brownian winding fields defined in \cite{BrugCamiaLis2018winding, CamiaGandolfiKleban16}, though we will not go in further detail on that in the current article. 

The next step is to show that such correlations are related to renormalised probabilities of being in a certain (randomly chosen) iterations of nested CLE$_4$ gaskets using \cite{LupuTopoAIHP} and localization arguments. To obtain renormalised probabilities of nested or simple CLE$_4$, we thereafter follow the idea that relevant statistical physics models correspond to the compactified GFF \cite{saleur1988correlation} and formulate this idea mathematically in terms of sums over GFFs with boundary conditions. 

We believe that similar techniques also allow one to calculate the renormalised probability that four distinct points on the sphere are on the same CLE$_4$ gasket - in fact the calculation of the twist fields of \cite{dixon1987orbifold} remains the same and only the step in adding back the compactification radius needs extra care because of two real degrees of freedom. The observations made in the paper also raise interesting questions on the level line geometry of the compactified and twisted free fields. These are questions we plan to touch upon in future work. \\

The rest of the paper is organised as follows. In Section \ref{sec:prelim} we first gather the main objects and notations and then prove a series of preliminary ingredients, which are usually small modifications of already known results. In Section \ref{Sec Twist loop soup} we define and study what we call the twist field correlations using the Brownian loop soup. We prove equivalences between different cut-offs and BLS cluster events, and do an explicit calculation in the case of two points - see Theorem \ref{thm:twopointtwist}. In Section \ref{Sec Twist CLE4} we connect these twist fields to further geometric renormalised probabilities involving Brownian loop soup clusters, excursion sets of the free field and finally CLE$_4$. In Section \ref{sec:together} we combine twist field correlations over different boundary conditions to obtain our main theorems.

\section{Background and auxiliary results}\label{sec:prelim}

Given the length of the article, we have opted to omit writing out definitions and theorems that already exist verbatim in the literature. 
Thus, all the subsections other than \ref{Subsec notations} 
will contain a new technical result, usually minor results that could nevertheless sometimes be of more general interest too. 
The section \ref{Subsec notations} contains a list of the main objects used in the article with relevant notations and references.

\subsection{List of standard used objects with notations and references}
\label{Subsec notations}

\paragraph{Domains.}
We denote by $\D$ the unit disk
$\D = \{z\in\C\vert ~\vert z\vert <1\}$,
and by $\D(z,\eps) = z+\eps\D$, the open disk of radius 
$\eps$ around $z$. 
We use the notation $D$ for general simply-connected domains, 
and domains with holes are denoted by
$$D_{\eps_1, \dots, \eps_n}(z_1, \dots, z_n) := 
D \setminus \bigcup_{j=1}^n\overline{\D(z_j, \eps_j)}.$$
\paragraph{Loops.}
We denote Jordan loops by $\Gamma$. 
For Jordan loops 
$\Gamma$ in $\C$, not hitting $0$ and surrounding $0$, 
we denote
$$
R^{-}(\Gamma) = d(0,\Gamma),
\qquad
R^{+}(\Gamma) = \max\{\vert z\vert~\vert z\in \Gamma\}.
$$
We further consider just the range of the loops and an equivalence relation on such Jordan loops,
where we identify $\Gamma$ with all its rescales
$\lambda \Gamma$ for $\lambda>0$.
We denote by $[\Gamma]$ the equivalence class of
$\Gamma$ and refer to $[\Gamma]$ as the \textit{shape}
of $\Gamma$.

The ratio $
\dfrac{R^{+}(\Gamma)}{R^{-}(\Gamma)}$
is scale-invariant,
and thus, it depends only on the shape $[\Gamma]$.
We will denote this ratio by $\rho^{\pm}([\Gamma])$.

Further, we denote by $\Int(\Gamma)$
the open simply connected domain surrounded by $\Gamma$.
We let $\Ext(\Gamma)$ denote the unbounded
connected component of $\C\setminus\Gamma$,
that is to say the set of points not surrounded by
$\Gamma$.
Consider $G_{\Ext(\Gamma)}$
the Green's function on
$\Ext(\Gamma)$.
For $z\in \Ext(\Gamma)$,
the quantity $G_{\Ext(\Gamma)}(z,\infty)$
(i.e. the Green's function at points $z$ and $\infty$),
is well defined, finite and strictly positive.
For instance, one can apply an inversion to map
$\infty$ to a finite point, and then use the conformal invariance of Green's functions.

\paragraph{Brownian loop soups.}

We denote by 
$\cL_{D}$ the Brownian loop soup in 
$D$
with intensity
$\frac{1}{2}
\mu^{\rm loop}_{D}$.
This measure on loops is given by
$$
\mu^{\rm loop}_{D}(\cdot) = 
\int_{D}\int_{0}^{+\infty}
\P^{z,z}_{D,t}(\cdot)
p_D(t,z,z) \dfrac{dt}{t}\, d^2 z,
$$
where
$\P^{z,z}_{D,t}$
is the Brownian bridge probability measure
from $z$ to $z$ in time $t$ with conditioning to stay 
inside $D$,
and $p_D(t,z,z)$ is the heat kernel,
associated to $\frac{1}{2}\Delta$,
with $0$ boundary condition on $\partial D$.
For $a\geq 0$, denote by
$\Xi^{a}_{D}$
the Poisson point process of boundary excursions in
$D$ with intensity
$
a
\mu^{\rm exc}_{D}.
$
The excursion measure $\mu^{\rm exc}_{D}$
is given by
$$
\mu^{\rm exc}_{D}(\cdot) =
\iint_{\partial D\times \partial D}
\mu^{x,y,\sharp}_D(\cdot)
H_D(dx,dy),
$$
where $H_D(dx,dy)$ is the boundary Poisson kernel of the domain $D$
(see e.g. \cite[Section 5.2]{LawC}), and
for $x\neq y\in \partial D$,
$\mu^{x,y,\sharp}_D$ is the Brownian excursion probability measure from $x$ to $y$ in $D$.
In general $x$ and $y$ are prime ends of $D$,
rather than points in $\C$
(see \cite[Section 2.4]{Pommerenke}).
But if $D$ is simply connected and $\partial D$ is given
by a Jordan loop, then the prime ends are in a one-to-one correspondence with the points of $\partial D$
\cite[Theorem 2.15]{Pommerenke}.
We will always take 
$\Xi^{a}_{D}$
to be independent from
$\cL_{D}$. We refer the reader to \cite{LW2004BMLoopSoup,LawC,ALS2} for standard results on the construction and properties of the Brownian loop soups and Poisson point processes of excursions. 

The corresponding objects on the metric graph \cite{Lupu2016Iso}
are used only in the preliminaries and their notations are decorated by $\widetilde{\quad}$-s. The metric graph loop soup at the critical intensity is denoted by 
$\widetilde \cL_{\widetilde{D}}$, and 
$\widetilde \Xi_{\widetilde{D}}^a$ denotes the PPP of Brownian boundary to boundary excursions of intensity equal to $a \geq 0$.

The Brownian loop soup $\cL_{D}$ and the excursions 
$\Xi^{a}_{D}$,
respectively $\widetilde \cL_{\widetilde{D}}$
and $\widetilde \Xi_{\widetilde{D}}^a$,
appear in the random walk representations
of the Gaussian free field
(a.k.a. ``isomorphism theorems'')
in continuum, respectively on metric graphs.
There is a rich literature on this topic.
Here we refer in particular to
\cite{LeJan2011Loops,Sznitman2012LectureIso,Lupu2016Iso,ALS2,ALS4}.
More precisely, $\cL_{D}$, 
respectively $\widetilde \cL_{\widetilde{D}}$,
corresponds to a GFF with $0$ boundary condition,
and $\cL_{D}\cup\Xi^{a}_{D}$,
respectively 
$\widetilde \cL_{\widetilde{D}}\cup
\widetilde \Xi_{\widetilde{D}}^a$,
corresponds to a GFF with boundary condition 
$v=\sqrt{2a}$ (that is $a=v^2/2$).
This will be evoked farther.

\paragraph{(Nested) CLE$_4$.}

Conformal loop ensembles CLE$_\kappa$ were introduced in \cite{SheffieldCLE09}, and their relation to Brownian loop soups described in \cite{SheffieldWerner2012CLE}. We will be working with the case $\kappa = 4$. 
In the language of Conformal Field Theory,
this corresponds to the central charge $c=1$.

A simple CLE$_4$ in a domain $D$ can be equivalently described either as 
\begin{itemize}
    \item a countable collection of disjoint and non-surrounding each other simple Jordan loops, also not touching the boundary, 
    \item or using the closed union of all the loops, called the CLE$_4$ gasket, that has empty interior almost surely.
\end{itemize}

Nested CLE$_4$, denoted CLE$_4^{\rm nest}$, is obtained by iteratively sampling an independent CLE$_4$ in the interior of each loop - see e.g. \cite{KemppainenWerner16CLE}. 
Thus, around each point $z$ we see 
\begin{itemize}
    \item a sequence of nested CLE$_4$ loops in 
    CLE$_4^{\rm nest}$, 
    \item or equivalently, a sequence of nested CLE$_4$ gaskets, whose inner and outer boundary towards $z$ are two consecutive nested CLE$_4$ loops.  
\end{itemize}

\paragraph{SLE$_4$ loop.}
We also consider the SLE$_4$ loop measure on simple loops separating $0$ from $\infty$ in the complex plane, first put forward in \cite{KemppainenWerner16CLE}. 
This infinite $\sigma$-finite measure can be normalized such that, in the full-plane setting, it agrees with the counting measure of loops in nested CLE$_4$ and separating $0$ and $\infty$.

Equivalently, one may construct it by first constructing a rooted SLE loop measure via whole-plane SLE and unrooting - this was done in \cite{Zhan21SLEloop}.
This construction also spans a wider range of $\kappa$. These constructions agree up to a multiplicative constant \cite{ang2024sleloopmeasureliouville}, and our normalization fixes this constant by matching the nested CLE$_4$ counting measure.

We will mainly consider a slice of this measure called the shape measure: this can be seen as the loop measure induced by the infinitesimal conditioning 
$\CR(0, \Int(\Gamma)) = 1$. 
We denote this shape measure by 
$\mu_{\text{SLE}_4}^{\rm loop}$.

We refer to \cite{KemppainenWerner16CLE,Zhan21SLEloop, ang2024sleloopmeasureliouville} for these constructions and their equivalence. Up to normalization, this is also known as the Malliavin--Kontsevich--Suhov measure for $\kappa=4$ (central charge $c=1$), see e.g. \cite{AiraultMalliavin01Unitarizing,KontsevichSuhov07MKS,Zhan21SLEloop,BaverezJego24MKS,CaiGao}.

\paragraph{Gaussian free fields.}
We denote by 
$\Phi_{D}^{(v)}$
the continuum GFF on 
$D$
with boundary condition 
$v$ on $\partial D$. We normalize it to have the covariance equal to the zero boundary Green's function with the singularity 
$G_D(z, w) \sim \frac{1}{2\pi}\log \frac{1}{|w- z|}$ 
near the diagonal. 
This fixes the height gap to $2\lambda = \sqrt{\pi/2}$. 
In the case $v=0$, we will simply write $\Phi_{D}$.

We denote $\tilde \phi_{\widetilde D}^{(v)}$ the metric graph GFF on the metric graph $\widetilde D$ with boundary conditions $v$. These boundary conditions will be constant in each of the boundary components. 
Again, $\tilde \phi_{\widetilde D}$
will correspond to the case $v=0$.

We refer the reader to \cite{PowWer,BerPow,SGFF} for standard results on the GFF,
and to \cite{Lupu2016Iso} for the particular metric graph setting.

\paragraph{Two-valued sets and first passage sets.}

We use the notation $A_{-a,b}$ to denote the two-valued local sets of the Gaussian free field, introduced in \cite{ASW} and further studied in \cite{AS2}. The case $a = b = 2\lambda$ corresponds to the case of CLE$_4$ in the sense that the two-valued local set $A_{-2\lambda,2\lambda}$ has the same law as 
the CLE$_4$ gasket. 

Similarly to the CLE$_4$, the set $A_{-a,b}$ can be seen either as a closed set, or a collection of (not necessarily disjoint, but non-crossing) loops given by the connected components of the boundary of 
$D \setminus A_{-a,b}$. 
Each of these loops $\ell$ comes with a label 
$s_\ell \in \{-a,b\}$ that is determined by the underlying free field.

The case $a+b = 2\lambda$ corresponds to a countable union of SLE$_4(\rho_1, -2-\rho_1)$ level lines.

The case $a \to \infty$ or $b \to \infty$ is called the First passage sets, denoted either $A_b$ or $A_{-a}$ respectively. They were introduced in \cite{ALS1} and further studied in \cite{ALS2}, to which we refer for further results.

\paragraph{(Sign) excursion sets of the Gaussian free field}

In the metric graph setting it is natural to decompose 
$\tilde \phi_{\widetilde D}^{(v)}$ into sign excursions: connected components of 
$\{x\in \widetilde D \vert 
\tilde \phi_{\widetilde D}^{(v)}(x) \neq 0\}$. 
We call these sets the excursion sets of the metric graph GFF.

In \cite{ALS4} the scaling limit of such a decomposition was proved, thereby giving sense to an excursion decomposition of the continuum Gaussian free field. Also a direct continuum description and characterisation was given. The corresponding limiting sets are again called the (sign) excursion sets of the GFF. 

The sign excursion decomposition has also a very beautiful description in terms of Brownian loop soup and boundary excursion clusters.

Namely, the excursion clusters of 
$\Phi_{D}^{(v)}$
are distributed exactly as the clusters of
$\cL_{D}
\cup
\Xi^{v^2/2}_{D}$. See Theorem 2 in \cite{ALS4}.

\paragraph{Imaginary chaos, $\cos$ and $\sin$ of the GFF}

Let $\Phi_D$ be the zero-boundary GFF in $D$, normalized so that
\[
G_D(z,w)\sim \frac{1}{2\pi}\log\frac1{|z-w|}
\qquad\text{as }w\to z.
\]
For $\beta\in(0,\sqrt2)$, we define the imaginary multiplicative chaos
\[
:e^{i\beta\sqrt{2\pi}\Phi_D}:
\]
as the limit of the circle-average approximations
\[
:e^{i\beta\sqrt{2\pi}\Phi_D(z)}:
=
\lim_{\eps\to 0}
\eps^{-\beta^2/2}
e^{i\beta\sqrt{2\pi}\,\Phi_{D,\eps}(z)},
\]
interpreted in the usual distributional sense.
With this normalization,
\[
\langle
:e^{i\beta\sqrt{2\pi}\Phi_D(z)}:
\rangle
=
\CR(z,D)^{-\beta^2/2},
\]
where the notation $\langle \cdot \rangle$ denotes the limits of expectations of circle-average approximations.
Similarly we denote
$$\Big\langle:\exp\Big(i\beta\sqrt{2\pi}\Phi_D(z_1)\Big):~
:\exp\Big(-i\beta\sqrt{2\pi}\Phi_D(z_2)\Big):
\Big\rangle $$
to be the limit of 
$$\E\left(\eps^{-\beta^2/2}
e^{i\beta\sqrt{2\pi}\,\Phi_{D,\eps}(z_1)}\eps^{-\beta^2/2}
e^{i\beta\sqrt{2\pi}\,\Phi_{D,\eps}(z_2)}\right).$$
We further define
\[
:\cos(\beta\sqrt{2\pi}\Phi_D):
=
\frac12\left(:e^{i\beta\sqrt{2\pi}\Phi_D}:+:e^{-i\beta\sqrt{2\pi}\Phi_D}:\right),
\]
and
\[
:\sin(\beta\sqrt{2\pi}\Phi_D):
=
\frac{1}{2i}\left(:e^{i\beta\sqrt{2\pi}\Phi_D}:-:e^{-i\beta\sqrt{2\pi}\Phi_D}:\right).
\]
We refer to \cite{junnila2020imaginary} for the construction and basic properties of the imaginary chaos.

\paragraph{Jacobi theta functions and elliptic integrals}

We will make use of the complete elliptic integral
$$\;\;(\star)\;\;K(k) = \int_{0}^{1}
        \frac{dx}{\sqrt{(1-k^{2}x^2)(1-x^2)}},\;K'(k):=K(\sqrt{1-k^{2}}).$$
We also let $\theta_2(q), \theta_3(q), \theta_4(q)$ be the Jacobi theta functions with nome $q$ given by 
$$\theta_2(q) := \sum_{n \in \Z} q^{(n+1/2)^2}; \qquad\theta_3(q) := \sum_{n \in \Z} q^{n^2}; \qquad \theta_4(q) := \sum_{n \in \Z}(-1)^n q^{n^2}.$$

The relations between $q, k, K, K'$ are as follows.
$$q = \exp(-\pi K'(k)/K(k)),$$
$$k = \theta_2^2(q)/\theta_3^2(q),$$
$$K' = -(\log q) \theta_3(q)^2/2,$$ 
where the latter two are formulas 2.1.7 and 2.2.3 in \cite{lawden2013elliptic}.

We also use the complementary nome $\hat q$, which satisfies $\log q \log \hat q = \pi^2$. We refer to \cite{lawden2013elliptic} Chapters 1 and 2 for background and many fascinating identities.

In our set-up the parameter $k >0$ will be given by $r^2$, where $r$ is determined as follows: there is a unique conformal map from $(D, z_1, z_2)$ to $(\D, -r, r)$. By the conformal invariance of the Green's function, one has $G_D(z_1, z_2) = \frac{1}{2\pi}\log\frac{1+r^2}{2r}$, 
and hence $r$ is equivalently determined by a conformal invariant $G_D(z_1, z_2)$ which uniquely defines $q$. 
See Lemma \ref{lem:parametr}.
Importantly in our set-up we then have the relation
\begin{equation}\label{eq:rtheta}
r = \theta_2(q)/\theta_3(q).
\end{equation}

\subsection{A stability result on conformal maps}

In this section we will present a stability result on the behaviour of univalent maps. This is most likely classical, but as we did not manage to locate it. 
We provide also a proof.

Consider open simply connected domains
$U\subset\C$, such that $0\in U$ and
$\CR(0,U)=1$.
We will denote by $f_U$ the unique conformal mapping from
$U$ to $\D$ fixing the origin $f_U(0)=0$
and with the derivative $f'_{U}(0)=1$.
We also recall that by Koebe quarter Theorem
\cite{Ahlfors2010ConfInv},
the domain $U$ necessarily contains the disk
$\frac{1}{4}\D$. The next proposition gives universal bounds on the behaviour of the map $f_U$ near the origin. 

\begin{prop}
\label{Prop de Branges}
There are constants $R_{\rm univ}\in (0,1/4)$
and $C_{\rm univ}>0$,
such that for every domain $U$ as above,
for every $z\in\overline{\D(0,R_{\rm univ})}$,
$$
\vert f_U(z) - z \vert 
\leq C_{\rm univ} \vert z\vert^{2},
\qquad
\vert f_U^{-1}(z) - z \vert 
\leq C_{\rm univ} \vert z\vert^{2}.
$$
\end{prop}
\begin{proof}
The main point is that $R_{\rm univ}$  
and $C_{\rm univ}$ do not depend on $U$.

The inverse mapping $f_U^{-1}$ is an
univalent function on $\D$.
It has a series expansion
$$
f_U^{-1}(z) = z + 
\sum_{n\geq 2} a_n z^{n} .
$$
For our purpose, it suffices to have universal bounds on the
coefficients $a_n$.
The best possible bounds are provided by
de Branges theorem (Bieberbach conjecture):
$\vert a_n\vert\leq n$ \cite[Chapter 17]{ConwayComplexVar2}.
Much weaker and easier to obtain bounds
would also be perfectly sufficient,
for instance by combining the Cauchy integral expression of
$a_n$ and the Koebe growth estimate for $f_U^{-1}$ 
\cite[Theorem 5-3]{Ahlfors2010ConfInv}.
However, we will use de Branges theorem in what follows,
since it is at our disposal.
Set
$$
C = \sum_{n\geq 2} n \dfrac{1}{2^{n-2}}
>1.
$$
Then, for every domain $U$ as above and for every $z\in\overline{\D(0,1/2)}$,
$$
\vert f_U^{-1}(z) - z \vert 
\leq C \vert z\vert^{2}.
$$
Further, set
$$
R' = 
\dfrac{1}{2 C}
<
\dfrac{1}{2}.
$$
Then, for every $U$ appropriate domain and $z\in\overline{\D(0,R')}$,
$$
\vert f_U^{-1}(z)\vert
\geq \vert z\vert - C \vert z\vert^{2}
\geq \dfrac{1}{2} \vert z\vert.
$$
This implies that
$$
\overline{\D(0,R'/2)}
\subset
f_U^{-1}(\overline{\D(0,R')})
\subset
f_U^{-1}(\overline{\D(0,1/2)}),
$$
whatever the domain $U$ of appropriate type.

Let $z\in \overline{\D(0,R'/2)}$.
Then $f_U(z)\in \overline{\D(0,R')}$, and
$$
\vert f_U(z) -z\vert \leq C \vert f_U(z)\vert^2.
$$
By iterating the inequality once more, we get
$$
\vert f_U(z) -z\vert \leq 
C(\vert z\vert + C \vert f_U(z)\vert^2)^2
\leq
2C \vert z\vert^2 + 2 C^3 \vert f_U(z)\vert^4 .
$$
By continuing iterating in this way,
we get that for every $N\geq 1$,
$$
\vert f_U(z) -z\vert \leq
\Big(
\sum_{k=1}^{N-1}
2^{2^{k+1} -(k+2)}
C^{2^k -1}\vert z\vert^{2^k}
\Big)
+ 2^{2^N - (N+1)}C^{2^N -1}\vert f_U(z)\vert^{2^N}.
$$
Since $\vert f_U(z)\vert\leq R'$, we get
$$
2^{2^N - (N+1)}C^{2^N -1}\vert f_U(z)\vert^{2^N}
\leq
\dfrac{2^{2^N - (N+1)}C^{2^N -1}}{2^{2^N} C^{2^N}}
=
\dfrac{1}{2^{N+1} C},
$$
which converges to $0$ as $N\to +\infty$.
Thus,
$$
\vert f_U(z) -z\vert \leq
\sum_{k\geq 1}
2^{2^{k+1} -(k+2)}
C^{2^k -1}\vert z\vert^{2^k}
\leq
\vert z\vert^2
\sum_{k\geq 1}
2^{2^{k+1} -(k+2)}
C^{2^k -1}(R'/2)^{2^k -2}
=
4C
\vert z\vert^2 
.
$$

So if one takes 
$R_{\rm univ} = R'/2$ and
$C_{\rm univ} = 4C$,
then the pair $(R_{\rm univ},C_{\rm univ})$
satisfies the desired property.
\end{proof}
\subsection{Extremal distance and its intrinsic approximation}

In this section we work in the unit disk $\D$, 
and ask the rather natural question: how can we approximate the extremal distance between a small loop surrounding $0$ and the boundary of the disk using something that is intrinsic only to the loop, i.e. does not see the boundary.

Consider a Jordan loop $\Gamma$ contained in $\D$
that avoids the point $0$ and surrounds $0$.
Denote by $\ann_\Gamma$
the annular domain delimited by
$\Gamma$ from the inside and by $\partial \D$
from the outside.
Let $\ED(\Gamma,\partial\D)$
be the extremal distance,
or extremal length,
between $\Gamma$ and $\partial \D$; 
a good detailed reference is 
\cite[Section 4]{Ahlfors2010ConfInv}

Extremal distance is also simply the effective electrical resistance
between $\Gamma$ and $\partial\D$.
It is a conformal invariant that parametrizes the modulus of annular domains.
It can be expressed in terms of
the boundary Poisson kernel in the domain $\ann_\Gamma$.
Let $H_{\ann_\Gamma}(dz,dw)$
be this boundary Poisson kernel
(see \cite[Section 5.2]{LawC} for this notion),
which is a measure on
$(\partial\ann_\Gamma)^{2}=(\Gamma\cup\partial\D)^{2}$.
For $w = e^{i\theta}\in\partial\D$,
$$
H_{\ann_\Gamma}(dz,dw)
=
H_{\ann_\Gamma}(dz,e^{i\theta}) d\theta,
$$
that is to say there is a density with respect to the
arc-length measure on $\partial\D$.
Further,
\begin{equation}
\label{Eq ED H}
\ED(\Gamma,\partial\D)
=
\bigg(
\int_{\Gamma}
\int_{0}^{2\pi}
H_{\ann_\Gamma}(dz,e^{i\theta}) d\theta
\bigg)^{-1}.
\end{equation}
This follows from
\cite[Theorem 4-5]{Ahlfors2010ConfInv}.
This also corresponds to the resistance being the inverse
of the conductance,
since $H_{\ann_\Gamma}$ is the conductance kernel. \\

As mentioned, we want to approximate this extremal distance between the loop $\Gamma$ 
and $\partial \D$
(when the loop gets smaller and smaller) by a quantity that depends purely on the loop and no longer takes into account that the loop is inside the unit disk. 

We recall that $\Ext(\Gamma)$ denotes exterior of $\Gamma$
and $G_{\Ext(\Gamma)}(z,w)$ the zero boundary Green's function in the domain $\Ext(\Gamma)$.
The intrinsic quantity we are interested in is described as follows:
$$
\int_{0}^{2\pi}
G_{\Ext(\Gamma)}(R^{+}(\Gamma)e^{i\alpha},\infty)
d\alpha
.
$$
It is scale-invariant,
i.e. it does not change if one replaces
$\Gamma$ by $\lambda\Gamma$, for a $\lambda>0$.
This comes from the scale invariance of Green's functions.
So it is a function of the shape 
$[\Gamma]$ (equivalence class under scaling).
So we use the notation
\begin{equation}\label{Eq intro Y Gamma}
Y([\Gamma]) = 
\int_{0}^{2\pi}
G_{\Ext(\Gamma)}(R^{+}(\Gamma)e^{i\alpha},\infty)
d\alpha
.
\end{equation}
If we denote further 
$$
R^{+}(\Gamma) = 
\max_{z\in\Gamma} \vert z\vert,
$$
then our intrinsic approximation result reads as follows.
\begin{prop}
\label{Prop expansion ED}
We have that
$$
\Big\vert
\ED(\Gamma,\partial\D)
-
\dfrac{1}{2\pi}
(\log R^{+}(\Gamma)^{-1})
-\dfrac{1}{2\pi}Y([\Gamma])
\Big\vert
\leq
\dfrac{1}{\pi}
\dfrac{R^{+}(\Gamma)}{1-R^{+}(\Gamma)}.
$$
\end{prop}
To prove this proposition, we will also introduce $\nu^{\infty}_{\Gamma}$
the harmonic probability measure on
$\Gamma$ seen from $\infty$. First we decompose this measure on the boundary of the unit disk.
\begin{lemma}
We have the following equality between measures
supported on $\Gamma$:
$$
\nu^{\infty}_{\Gamma}(dz)
=
\1_{z\in\Gamma}
\int_{0}^{2\pi}
G_{\Ext(\Gamma)}(e^{i\theta},\infty)H_{\ann_\Gamma}(dz,e^{i\theta})
d\theta
.
$$
In particular,
\begin{equation}
\label{Eq 1 int G Ext Gamma H}
\int_{0}^{2\pi}
\int_{\Gamma}
G_{\Ext(\Gamma)}(e^{i\theta},\infty)H_{\ann_\Gamma}(dz,e^{i\theta})d\theta
=1.
\end{equation}
\end{lemma}
\begin{proof}
This follows by decomposing a Brownian
excursion from 
$z\in \Gamma$ to $\infty$
in the domain $\Ext(\Gamma)$
at the last visit time of $\partial\D$.
\end{proof}
Further, if we denote by $P_{\C\setminus(R^{+}(\Gamma)\overline{\D})}$
the Poisson kernel of the domain
$\C\setminus(R^{+}(\Gamma)\overline{\D})$, 
we can decompose as follows.
\begin{lemma}
For every $\theta\in[0,2\pi)$,
\begin{eqnarray}
\label{Eq G Ext Gamma decomp}
G_{\Ext(\Gamma)}(e^{i\theta},\infty)
&=&
\dfrac{1}{2\pi}
\log R^{+}(\Gamma)^{-1}
\\
\nonumber
&&+
\int_{0}^{2\pi}
G_{\Ext(\Gamma)}(R^{+}(\Gamma)e^{i\alpha},\infty)
P_{\C\setminus(R^{+}(\Gamma)\overline{\D})}
(e^{i\theta},R^{+}(\Gamma)e^{i\alpha})
(R^{+}(\Gamma)d\alpha).
\end{eqnarray}
In particular,
\begin{multline}
\label{Eq rel G Ext Gamma ED}
1
=
\dfrac{1}{2\pi}
(\log R^{+}(\Gamma)^{-1})
\ED(\Gamma,\partial\D)^{-1}
\\
+
\int_{0}^{2\pi}
d\theta
\int_{0}^{2\pi}
(R^{+}(\Gamma)d\alpha)
\int_{\Gamma}
G_{\Ext(\Gamma)}(R^{+}(\Gamma)e^{i\alpha},\infty)
P_{\C\setminus(R^{+}(\Gamma)\overline{\D})}
(e^{i\theta},R^{+}(\Gamma)e^{i\alpha})
H_{\ann_\Gamma}(dz,e^{i\theta}).
\end{multline}
\end{lemma}
\begin{proof}
Consider a Brownian excursion
from $e^{i\theta}$ to $\infty$
in the domain $\Ext(\Gamma)$.
There are two cases:
either this excursion avoids $R^{+}(\Gamma)\overline{\D}$,
or it hits
$R^{+}(\Gamma)\partial \D$.
The first case gives rise to the partition function
$$
G_{\C\setminus(R^{+}(\Gamma)\overline{\D})}
(e^{i\theta},\infty)
=
\dfrac{1}{2\pi}
\log R^{+}(\Gamma)^{-1} .
$$
In the second case we decompose the excursion at the position
$R^{+}(\Gamma)e^{i\alpha}$
of first hitting of $R^{+}(\Gamma)\overline{\D}$,
which gives rise to
$$
\int_{0}^{2\pi}
G_{\Ext(\Gamma)}(R^{+}(\Gamma)e^{i\alpha},\infty)
P_{\C\setminus(R^{+}(\Gamma)\overline{\D})}
(e^{i\theta},R^{+}(\Gamma)e^{i\alpha})
(R^{+}(\Gamma)d\alpha).
$$
This is where \eqref{Eq G Ext Gamma decomp}
comes from.
To obtain \eqref{Eq rel G Ext Gamma ED},
one integrates the identity \eqref{Eq G Ext Gamma decomp}
against
$H_{\ann_\Gamma}(dz,e^{i\theta})d\theta$,
and uses the identities \eqref{Eq 1 int G Ext Gamma H}
and \eqref{Eq ED H}.
\end{proof}

We are now ready to prove the proposition.
\begin{proof}[Proof of Proposition \ref{Prop expansion ED}.]
We start from the identity \eqref{Eq rel G Ext Gamma ED} and conclude that
$$
\ED(\Gamma,\partial\D) - \dfrac{1}{2\pi}\log R^{+}(\Gamma)^{-1}$$
equals
$$ \ED(\Gamma,\partial\D)\int_{0}^{2\pi}
d\theta
\int_{0}^{2\pi}
(R^{+}(\Gamma)d\alpha)
\int_{\Gamma}
G_{\Ext(\Gamma)}(R^{+}(\Gamma)e^{i\alpha},\infty)
P_{\C\setminus(R^{+}(\Gamma)\overline{\D})}
(e^{i\theta},R^{+}(\Gamma)e^{i\alpha})
H_{\ann_\Gamma}(dz,e^{i\theta}).$$
It now remains to plug the expression \eqref{Eq ED H}, the definition of $Y([\Gamma])$ and the following pointwise approximation.
\begin{claim}
\label{Eq Poisson kernel uniform}
For every $\theta,\alpha\in [0,2\pi)$,
$$
\Big\vert
R^{+}(\Gamma)P_{\C\setminus(R^{+}(\Gamma)\overline{\D})}
(e^{i\theta},R^{+}(\Gamma)e^{i\alpha})
-
\dfrac{1}{2\pi}
\Big\vert
\leq
\dfrac{1}{\pi}
\dfrac{R^{+}(\Gamma)}{1-R^{+}(\Gamma)}.
$$
\end{claim}
\begin{proof}[Proof of Claim]
The Poisson kernel 
$P_{\C\setminus(R^{+}(\Gamma)\overline{\D})}$
is explicit:
$$
P_{\C\setminus(R^{+}(\Gamma)\overline{\D})}
(z,R^{+}(\Gamma)e^{i\alpha}) = 
\dfrac{1}{2\pi R^{+}(\Gamma)}
\Re\bigg(
\dfrac{1+z^{-1}R^{+}(\Gamma)e^{i\alpha}}
{1-z^{-1}R^{+}(\Gamma)e^{i\alpha}}
\bigg)
.
$$
Thus,
$$
R^{+}(\Gamma)P_{\C\setminus(R^{+}(\Gamma)\overline{\D})}
(e^{i\theta},R^{+}(\Gamma)e^{i\alpha})
-
\dfrac{1}{2\pi}
=
\dfrac{R^{+}(\Gamma)}{\pi}
\Re\bigg(
\dfrac{e^{i(\alpha-\theta)}}
{1-R^{+}(\Gamma)e^{i(\alpha-\theta)}}
\bigg).
$$
This implies the desired bound. 
\end{proof}
\end{proof}

Denote by $\CR(0,\C\setminus \Gamma)=\CR(0,\Int(\Gamma))$
the conformal radius seen from $0$ of $\Int(\Gamma)$.
The following result says that the quantity 
$Y([\Gamma])$ is universally bounded.
\begin{prop}
\label{Prop bounds Y}    
We have that
\begin{equation}
\label{Eq bound Y Gamma} 
Y([\Gamma])\leq 2 \log 2 .
\end{equation}
Moreover, 
\begin{equation}
\label{Eq bound Y Gamma R+ CR} 
Y([\Gamma])\leq \log \dfrac{R^{+}(\Gamma)}{\CR(0,\C\setminus \Gamma)}.
\end{equation}
\end{prop}
\begin{proof}
Let $\lambda\in (0,1]$.
We have that $Y([\lambda\Gamma]) = Y([\Gamma]).$
Thus,
$$
Y([\Gamma])
=
\lim_{\lambda\to 0}
(
2\pi\ED(\lambda\Gamma,\partial\D)
-
\log R^{+}(\lambda\Gamma)^{-1}
).
$$
So \eqref{Eq bound Y Gamma} follows from quite standard comparison inequalities (see e.g. \cite[Corollary 2.4, Proposition 2.5]{ALS3}):
\begin{equation}
\label{Eq bounds ED}
\ED(\Gamma,\partial\D)
\leq
\dfrac{1}{2\pi}
\log \CR(0,\C\setminus \Gamma)^{-1},
~~
\dfrac{1}{2\pi}
\log R^{+}(\Gamma)^{-1}
\leq
\ED(\Gamma,\partial\D)
\leq
\dfrac{1}{2\pi}
\log (4 R^{+}(\Gamma)^{-1}).
\end{equation}
Further,
$$
\dfrac{R^{+}(\lambda\Gamma)}{\CR(0,\C\setminus \lambda\Gamma)}
=
\dfrac{R^{+}(\Gamma)}{\CR(0,\C\setminus \Gamma)} .
$$
Thus,
$$
\log \dfrac{R^{+}(\Gamma)}{\CR(0,\C\setminus \Gamma)}
-
Y([\Gamma])
=
\lim_{\lambda\to 0}
(\log \CR(0,\C\setminus \lambda\Gamma)^{-1}
-
2\pi \ED(\lambda\Gamma,\partial\D)).
$$
Then we apply the first inequality in
\eqref{Eq bounds ED} to get \eqref{Eq bound Y Gamma R+ CR}.
\end{proof}

\subsection{Properties of the shape distribution of the SLE$_4$ loop}
\label{Subsec stat SLE4}



We will extract an exact law related to random shapes distributed according to 
$\mu_{\text{SLE}_4}^{\rm loop}$.

\begin{prop}
\label{Prop law shape}
Denote by $[\Gamma^{\ast}]$ a random shape induced by $\mu_{\text{SLE}_4}^{\rm loop}$ and normalized to be a probability measure.
Denote by $\rho^{+}([\Gamma^{\ast}])$ the ratio
$$
\rho^{+}([\Gamma^{\ast}])
=
\dfrac{R^{+}(\Gamma^{\ast})}{\CR(0,\C\setminus\Gamma^{\ast})},
$$
which depends only on the shape $[\Gamma^{\ast}]$.
Let $(\mathcal{R}_t)_{t\geq 0}$
denote a Bessel 3 process starting from $0$
and by $T_{\pi}^{\rm BES3}$ the first hitting time of level $\pi$
by $(\mathcal{R}_t)_{t\geq 0}$.
Then the random variable 
\begin{equation}
\label{Eq log rho plus Y}
\log \rho^{+}([\Gamma^{\ast}])
-
Y([\Gamma^{\ast}])
\end{equation}
has the same distribution as $T_{\pi}^{\rm BES3}$.
\end{prop}

To prove this proposition, we will build a link between the SLE$_4$ loop whose shape is sampled from $\mu_{\text{SLE}_4}^{\rm loop}$ and the CLE$_4$ loop in $\D$, for which \cite{ALS3} gives such a result.

In this direction, 
\begin{itemize}
    \item Consider $\check{\Gamma}_1$ a random
SLE$_4$ loop surrounding $0$,
with the shape $[\check{\Gamma}_1]$ distributed
according to the invariant measure 
$\mu_{\text{SLE}_4}^{\rm loop}$ \cite{KemppainenWerner16CLE},
and with $\CR(0,\C\setminus\check{\Gamma}_1)=1$.
\item We now sample a CLE$_4$ inside the interior surrounded by $\check{\Gamma}_1$,
and we take $\check{\Gamma}_2$ to be the loop
that surrounds $0$.
By construction, the shape $[\check{\Gamma}_2]$
is again distributed according to 
$\mu_{\text{SLE}_4}^{\rm loop}$.
\item Finally, we map conformally the interior of
$\check{\Gamma}_1$ to the unit disk $\D$
via $f_{\check{\Gamma}_1}$ and we set
$\check{\Gamma}=f_{\check{\Gamma}_1}(\check{\Gamma}_2)$.
Then $\check{\Gamma}$ is distributed as the
$\CLE_4$ loop in $\D$ surrounding $0$.
\end{itemize}

\begin{proof}
Note that \eqref{Eq bound Y Gamma R+ CR} ensures that
\eqref{Eq log rho plus Y} is non-negative.
By \cite[Theorem 1]{ALS3}, and Williams path decomposition,
the random variable
\begin{equation}
\label{Eq log CR - ED}
\log\CR(0,\C\setminus \check{\Gamma})^{-1}
-
2\pi\ED(\check{\Gamma},\partial\D)
\end{equation}
is distributed as $T_{\pi}^{\rm BES3}$,
and is independent from $\ED(\check{\Gamma},\partial\D)$.
Further, by conformal invariance of the extremal distance and the choice $\CR(0, \C \setminus \check{\Gamma}_1)$ 
the random variable
\begin{equation}
\label{Eq log CR - ED 2}
\log\CR(0,\Int(\check{\Gamma}_1)\setminus \check{\Gamma}_2)^{-1}
-
2\pi\ED(\check{\Gamma}_2,\check{\Gamma}_1)
\end{equation}
has exactly the same law, i.e. is distributed as 
$T_{\pi}^{\rm BES3}$,
and is again independent from $\ED(\check{\Gamma}_1,\check{\Gamma}_2)$.

Hence, the conditional distribution of \eqref{Eq log CR - ED}
is the same if we condition on the value of 
$\ED(\check{\Gamma}_2,\check{\Gamma}_1)$.
But in the limit when $\ED(\check{\Gamma}_2,\check{\Gamma}_1)$
tends to infinity, by Proposition \ref{Prop expansion ED} the expression
\eqref{Eq log CR - ED} is close to
$$
\log\dfrac{R^{+}(\check{\Gamma})}
{\CR(0,\C\setminus \check{\Gamma})}
- Y([\check{\Gamma}]),
$$
up to an error $o(1)$. But this equals \eqref{Eq log rho plus Y} by definition and we conclude.
\end{proof}

\begin{cor}\label{cor:expmoments}
Let $[\Gamma^{\ast}]$ be a random shape distributed according to $\mu_{\text{SLE}_4}^{\rm loop}$.
Denote by $\rho^{\pm}([\Gamma^{\ast}])$ the ratio
$$
\rho^{\pm}([\Gamma^{\ast}])
=
\dfrac{R^{+}(\Gamma^{\ast})}{d(0,\Gamma^{\ast})},
$$
which depends only on the shape $[\Gamma^{\ast}]$.
Then
$$
\E\Big[
\rho^{\pm}([\Gamma^{\ast}])^{1/8}
e^{-Y([\Gamma^{\ast}])/8}
\Big]
<+\infty.
$$
\end{cor}
\begin{proof}
We have that
$$
\E\Big[
\rho^{\pm}([\Gamma^{\ast}])^{1/8}
e^{-Y([\Gamma^{\ast}])/8}
\Big]
\leq
4^{1/8}
\E\Big[
\rho^{+}([\Gamma^{\ast}])^{1/8}
e^{-Y([\Gamma^{\ast}])/8}
\Big].
$$
According to Proposition \ref{Prop law shape},
$$
\E\Big[
\rho^{+}([\Gamma^{\ast}])^{1/8}
e^{-Y([\Gamma^{\ast}])/8}
\Big]
=
\E
\Big[
e^{\frac{1}{8}T_{\pi}^{\rm BES3}}
\Big].
$$
But, as $t\to +\infty$,
$$
\P(T_{\pi}^{\rm BES3}>t) = e^{-\frac{1}{2}t + o(t)};
$$
see \cite[Section 5.5]{ALS3}.
Thus, the above exponential moment is finite.
\end{proof}

\subsection{Geometry of the continuum GFF using iterated TVS} 
Here we will recap some relevant results on the level line geometry of the continuum free field with constant boundary conditions, and also their relation to Brownian loop soups. 

Recall that each two-valued set (TVS) $A_{-a,b}$ of the GFF with zero boundary condition $\Phi_D$ can be seen either as a fractal gasket or as a countable collection of loops 
$\ell$, 
that form the boundaries of connected components of $D \setminus A_{-a,b}$. These are SLE$_4$-type loops, and to each of them a label $s_\ell$ with values in $\{-a, b\}$ is attached that is determined by $\Phi_D$, 
and such that we can write
$$\Phi_D = \sum_{\ell \in A_{-a,b}}\Phi_{\Int(\ell)}+
s_\ell \1_{\Int(\ell)},$$
and further, conditionally on $A_{-a,b}$, 
the fields $\Phi_{\Int(\ell)}$ are given by a collection of independent zero boundary GFFs in the domains $\Int(\ell)$. 

It is shown in \cite{AS2} that only in the case of CLE$_4$, i.e. of $A_{-a,a+4\lambda}$ for $a\in\R$, 
the labels $s_\ell$ are conditionally independent given the set, 
and further only in the case $A_{-2\lambda, 2\lambda}$ they are given by i.i.d symmetric coin tosses.

Here we describe the relation between the nested CLE$_4$ and $\Phi_D$, and in particular explain how random walks appear from these i.i.d. labels and how the sign excursion sets can be encoded using the nested CLE$_4$ and the labels. These are all known results, just somewhat reformulated.

\begin{prop}\label{prop:CLEcoupling1}
Consider a zero boundary GFF $\Phi_D$ on $D$ and the nested CLE$_4$ determined by exploring consecutively two-valued sets $A_{-2\lambda, 2\lambda}$. Then, for all boundary conditions $v \in 2\lambda \Z$, we have the following properties in coupling with $\Phi_D + v$: 
\begin{enumerate}
\item The labels $s_\ell \in \{-2\lambda, 2\lambda\}$ attached to each loop $\ell \in \text{CLE}_4^{n}$ are i.i.d. symmetric conditionally on the loops, 
and in terms of distributions,
$$\sum_{\Gamma \in \text{CLE}^{\rm nest}_4}(v1_{\Gamma \text{ outermost}}+s_\Gamma)
\1_{\Int(\Gamma)} = \Phi_D +v.$$
\item 
We assign the height $v$ to the boundary of the domain. Thereby, each gasket in CLE$_4^{\rm nest}$ can be uniquely attached a height corresponding to the sum of labels of loops $s_\ell$ in CLE$_4^{\rm nest}$ strictly surrounding this gasket, including the boundary of the domain. If this sum equals $h$, then we say that the gasket is at height $h$. Equivalently we can attach a height to each loop in CLE$_4^{\rm nest}$ by considering the sum of labels of loops surrounding and including this loop and the boundary. This way the height of the gasket corresponds to the height of its outermost loop.   
\item In this coupling, conditionally on the loops, for each $z \in D$, the heights $h$ of the gaskets in 
CLE$_4^{\rm nest}$ surrounding $z$ form a simple symmetric random walk $S^z(m)$ with step-size $2\lambda$ and starting point $v$. 
\item For $v \neq 0$, consider $A_0$, the first passage set of $\Phi_D + v$ with boundary values $0$. Then $A_0$ is almost surely equal to taking the connected component of the boundary of the closed union of all gaskets of non-zero height. In particular, for each $z \in D \cap \Q^2$, the boundary of the connected component of 
$D \setminus A_0$ containing $z$ is constructed as follows: we consider the nested loop sequence surrounding $z$ until the first time that a loop of height $0$ appears and that is the boundary. 
\item Further, the sign excursion sets of $\Phi_D + v$ have the following description using nested CLE$_4$: we consider the closed union of all the gaskets in CLE$_4^{\rm nest}$ of height different than zero. The connected components have then the law of sign excursion sets.
\item Finally, the dual of the excursion sets, i.e. the complement of the excursion sets in $D$, is given by the disjoint union of all the gaskets of height $0$ in the nested CLE$_4^{\rm nest}$. 
\end{enumerate}
\end{prop}

\begin{proof}
Point A is the basic property and e.g. is explained in \cite{ASW}. Points B, C follow from the construction of nested CLE and point A.

Points D, E are discussed in \cite[Remark 17]{ALS4}. 
Finally, point F follows from point E. 
\end{proof}

In fact, it will be also important for us to generalize this proposition to arbitrary constant boundary conditions $v \in \R$, and to different iterations of two-valued sets. These are relevant for level set descriptions of the compactified free field, but we will not go in further detail on that here. The proposition is basically the same, the previous case corresponding to the limit $g \to 1$ in a specific sense.

\begin{prop}\label{prop:CLEcoupling2}
Consider a zero boundary GFF $\Phi_D$ on $D$, $g> 1$, 
and the following iteration of two-valued sets:
\begin{itemize}
    \item We start by exploring the outermost CLE$_4$ (i.e. $A_{-2\lambda, 2\lambda}$); 
\item next, inside each loop of CLE$_4$ we explore a two-valued set $A_{- 2\lambda, 2\sqrt{g}\lambda - 2\lambda}$ or $A_{-2\sqrt{g}\lambda+  2\lambda, + 2\lambda}$ depending on whether the label of the CLE$_4$ loop towards $z$ was $2\lambda$ or $-2\lambda$ respectively;
\item we then again explore CLE$_4$ in the interior of each new loop;
\item we continue this way making sure that on every even step we are at a height in $2\sqrt{g}\lambda\Z$, and on every odd step at a height in 
$2\sqrt{g}\lambda \Z \pm 2\lambda$.
\end{itemize}
Then, for any  $v \in 2\sqrt{g}\lambda \Z$, we have the following properties in coupling with $\Phi_D + v$: 
\begin{enumerate}
\item The labels $s_\ell \in \{-2\lambda, 2\lambda\}$ attached to each of $\ell$ appearing in an odd iteration, i.e. in an exploration of CLE$_4$, are still i.i.d. symmetric conditionally on all the loops. The labels appearing in the even iterations have the following conditional laws, given all the loops and the label $s_\ell$ of the last surrounding CLE$_4$ loop $\ell$: 
\begin{itemize}
\item they are equal to $(\sqrt{g}-1)s_\ell$ with probability $1/\sqrt{g}$;
    \item and equal to $-s_\ell$ with probability $1-1/\sqrt{g}$.
\end{itemize}
In terms of distributions we still have 
$$\sum_{\Gamma \in \text{CLE}^{\rm nest}_4}(v1_{\Gamma \text{ outermost}}+s_\Gamma)
\1_{\Int(\Gamma)} = \Phi_D +v.$$
\item 
We assign the height $v$ to the boundary of the domain. Thereby, each gasket in the iteration can still be uniquely attached a height corresponding to the sum of labels of loops $s_\ell$ strictly surrounding this gasket and the height of the boundary. Again, if this sum equals $h$, then we say that the gasket is at height $h$. Equivalently we can attach a height to each loop in the iteration by considering the sum of labels surrounding and including this loop and the height of the boundary. This way the height of the gasket corresponds to the height of its outermost loop.   
\item In this coupling, conditionally on the loops, 
for each $z \in D$, the heights $h$ of the CLE$_4$ gaskets in the iteration (i.e. the gaskets explored at each odd step) surrounding $z$ form a lazy random walk $S^z(m)$ with independent steps taking values in $\{0, -2\sqrt{g}\lambda, 2\sqrt{g}\lambda\}$, with probabilities $1-\frac{1}{\sqrt{g}}, \frac{1}{2\sqrt{g}},\frac{1}{2\sqrt{g}}$ respectively, and starting point $v$. 
\item For $v \neq 0$, consider $A_0$, the first passage set of $\Phi_D + v$ with boundary values $0$. Then $A_0$ is almost surely equal to taking the connected component of the boundary of the closed union of all gaskets in the iteration of non-zero height. In particular, for each 
$z \in D \cap \Q^2$, the boundary of the connected component of $D \setminus A_0$ containing $z$ is constructed as follows: we consider the nested loop sequence surrounding $z$ until the first time that a loop of height $0$ appears and that is the boundary. 
\item Further, the sign excursion sets of $\Phi_D + v$ have the following description using the nested loop ensemble: we consider the closed union of all the gaskets of height different than zero. The connected components have then the law of sign excursion sets.
\item Finally, the dual of the excursion sets, i.e. the complement of the excursion sets in $D$, is given by the disjoint union of the all the CLE$_4$ gaskets of height $0$ in this nesting. 
\end{enumerate}
\end{prop}

\begin{proof}
The part on CLE$_4$ comes from previous proposition and the probabilities for different labels from the local set decomposition and zero expectation, see e.g. \cite{AS2}. Further, Points B, C follow from A and the construction.

So it remains to argue D and E, as point F is again just a consequence of E. Further point E follows from point D by iteration and the fact that conditionally on the loop of the outer boundary, each excursion set is again given just by an FPS \cite{ALS2}. And point D itself is a direct consequence of the uniqueness of first passage sets, \cite[Theorem 4.3]{ALS1}.
\end{proof}

\subsection{Convergence of Brownian loop soup clusters}

The topological equivalence \cite{LupuTopoAIHP} that allows us to compute geometric observables of GFF sign excursion clusters via Brownian loop soups is valid on the level of metric graphs. We need to be able to push these results to the continuum set-up. 

The set-up is as in \cite[Section 4.1]{ALS2} in terms of notations, domains and topologies. The convergence results pertaining to the Brownian loop measure and the excursion measure are given for example by \cite[Lemma 4.6]{ALS2}.

\subsubsection{Convergence of separation events in the loop soups}

In this section we will discuss the convergence theorems for Brownian loop soup and boundary-to-boundary excursions. 

First, we need a convergence of topological separation events for Brownian loop soup and excursion clusters in simply-connected domains. 

\begin{theorem}\label{thm:convbls}
Consider a bounded simply-connected domain $D$ and sequence $(\widetilde D_m)_{m\geq 1}$ of approximating metric graphs via the lattice $\frac{1}{m}\Z^2$. 

In these domains consider the Brownian loop soup 
$\cL_D$ at the critical intensity 
($\alpha=1/2$ or central charge $c=1$) and the metric graph loop soup $\widetilde \cL_{\widetilde D_m}$ also at the critical intensity. 
Let $\widetilde \Xi^a_{\widetilde D_m}$ denote the PPP of Brownian boundary to boundary excursions of intensity $a \geq 0$ on the metric graph $\widetilde D_m$, 
and $\Xi^a_{D}$ the corresponding PPP of Brownian boundary to boundary excursions in the continuum.

Finally, let $(\widetilde C_k^m)_{k \geq 1}$ be the clusters of 
$\widetilde \Xi^a_{\widetilde D_m} 
\cup \widetilde \cL_{\widetilde D_m}$, 
and $(C_k)_{k \geq 1}$ the clusters of 
$\Xi^a_{D} \cup \cL_D$,  listed say by the order of decreasing diameter. 
Denote by $\widetilde \gamma, \gamma$ closed paths on the metric graph and in the continuum respectively.

Then for any $z_1,\dots, z_n \in D$ distinct points, and any $\eps_1, \dots, \eps_n \geq 0$, we have that the probabilities
$$\tilde p_{a,m} := \P\left(\nexists \widetilde C_i^{m}\text{ with }\widetilde \gamma \in \widetilde C_i^{ m}\text{ such that} \sum_{j=1}^{n}
\indx_{\widetilde \gamma}(z_j)
\text{ odd}
 \: \& \: \widetilde C_i^{m} \cap \bigcup_{j=1}^n\D(z_j, \eps_j) = \emptyset\right)$$ 
converge, as $m \to \infty$ to 
$$p_a := \P\left(\nexists C_i\text{ with }\gamma \in C_i\text{ such that} \sum_{j=1}^{n}
\indx_{\gamma}(z_j)
\text{ odd}
 \: \& \: C_i \cap \bigcup_{j=1}^n\D(z_j, \eps_j) = \emptyset\right).$$
\end{theorem}

Second, we extend this theorem to the case of domains with $n$ holes.

\begin{theorem}\label{thm:convbls2}
For $z_1, \dots, z_n \in D$ distinct points, and 
$\eps_1, \dots, \eps_n > 0$, let 
$D_{\eps_1,\dots, \eps_n}(z_1, \dots, z_n)$ be the 
$n+1$-connected domain 
$D \setminus \bigcup_{j = 1}^n \overline{\D(z_j,\eps_j)}$. 
Let $\widetilde D^m_{\eps_1, \dots, \eps_n}$ be the corresponding approximating metric graphs.

In these domains consider the Brownian loop soup 
$\cL_{D_{\eps_1,\dots, \eps_n}(z_1, \dots, z_n)}$ at the critical intensity, and the metric graph loop soup 
$\widetilde \cL_{\widetilde D^m_{\eps_1, \dots, \eps_n}}$, also at the critical intensity. 
Let 
$\widetilde \Xi^a_{\widetilde D^m_{\eps_1, \dots, \eps_n}}$ denote the PPP of Brownian boundary to boundary excursions of intensity equal to $a \geq 0$ on the outer boundary and equal to $0$ on the inner boundaries on the metric graph 
$\widetilde D^m_{\eps_1, \dots, \eps_n}$. 
We let $\Xi_{D_{\eps_1,\dots, \eps_n}(z_1, \dots, z_n)}^a$ denote the corresponding PPP of Brownian boundary to boundary excursions in the continuum.

Finally, let 
$(\widetilde C_k^{\eps_1, \dots, \eps_n, m})_{k \geq 1}$ 
be the clusters of 
$\widetilde \Xi^a_{\widetilde D^m_{\eps_1, \dots, \eps_n}} 
\cup \widetilde \cL_{\widetilde D^m_{\eps_1, \dots, \eps_n}}$, 
and $(C_k^{\eps_1, \dots, \eps_n})_{k \geq 1}$ the clusters of $\Xi_{D_{\eps_1,\dots, \eps_n}(z_1, \dots, z_n)}^a 
\cup \cL_{D_{\eps_1,\dots, \eps_n}(z_1, \dots, z_n)}$,  
listed say by the order of decreasing diameter. 
Let $\widetilde \gamma, \gamma$ denote closed paths on the metric graph and in the continuum respectively.

Then we have that the probabilities
$$\tilde p^{\rm mult}_{a,m} := \P\left(\nexists \widetilde C_i^{\eps_1, \dots, \eps_n, m}\text{ with }\widetilde \gamma \in \widetilde C_i^{\eps_1, \dots, \eps_n, m}\text{ such that} \sum_{j=1}^{n}
\indx_{\widetilde \gamma}(z_j)
\text{ odd}
\right),$$ 
converge, as $m \to \infty$, to 
$$p^{\rm mult}_a := \P\left(\nexists C_i^{\eps_1, \dots, \eps_n}\text{ with }\gamma \in C_i^{\eps_1, \dots, \eps_n}\text{ such that} \sum_{j=1}^{n}
\indx_{\gamma}(z_j)
\text{ odd}\right).$$
\end{theorem}

We will first discuss how the simply-connected case can be combined directly from the results and proofs existing in the literature. Thereafter, we argue how to extend it to the case of $n$-connected domains via an absolute continuity argument.

A key topological input is the following simple lemma.
This is most likely classical and follows from results in planar topology / homology even under weaker hypotheses, but for us the following version suffices and let us give a hands-on proof. 

\begin{lemma}\label{lem:topo}
Let $z_1, \dots, z_n$ be distinct points in $D$. Let $A \subseteq D$ be a closed connected set not containing any of the $z_j$, such that $D \setminus A$ has countably many connected components,  the boundary of each connected component is a Jordan curve, and for every $\eps > 0$ only finitely many of these curves have diameter more than 
$\eps$. 

Let $\widehat A$ be the set obtained as follows: we take the connected components $O_j$ of $z_j$ in $D \setminus A$, the connected component $O_0$ whose boundary is $\partial D$ and set 
$\widehat A := \overline{D \setminus \bigcup_{j=0}^n O_j}$.

Then there exists a closed path $\gamma \subseteq A$ with $\sum_{j=1}^{n}
\indx_{\gamma}(z_j) \text{ odd}$ if and only if there exists a closed path $\widehat \gamma \subseteq \widehat A$ with $\sum_{j=1}^{n}
\indx_{\widehat \gamma}(z_j) \text{ odd}$. 
\end{lemma}
This is the pedestrian way to state this avoiding topological vocabulary. We will also sketch a proof in similar lines, although heuristically what is just happening is that adding 2-cells doesn't change 1-homology - the difference w.r.t. textbooks being the presence of infinitely many holes.
\begin{proof}
One direction is clear as $A \subseteq \widehat A$.
For the other direction we observe that given $\widehat \gamma \in \widehat A$ with the given condition, we can construct $\gamma$ heuristically as follows: $\gamma(t)$ is equal to $\widehat \gamma(t)$ as long as it is inside $A$. When it enters a connected component of $D \setminus A$, then we follow the boundary of this connected component until $\widehat \gamma$ enters $A$ again. 
This can be made precise via a limiting argument by defining $\gamma_\eps(t)$ such that $\gamma_\eps$ does not enter the connected components with diameter more than 
$\eps$. The exact argument is given in the proof of the second statement of \cite[Proposition 3.12]{AS2}.
\end{proof}

An immediate corollary that we will use is as follows.
\begin{cor}\label{cor:topo}
Let $z_1, \dots, z_n$ be distinct points in $D$. Let $A \subseteq D$ be a closed connected set not containing any of the $z_j$, such that $D \setminus A$ has countably many connected components,  the boundary of each connected component is a Jordan curve, and for every $\eps > 0$ only finitely many of these curves have diameter more than 
$\eps$. 
Let $(A_m)_{m \geq 1}$ be closed sets that converge to $A$ in the sense that $A_m$ converges to $A$ in the Hausdorff distance, as do the outer boundary and all boundaries of the connected components of $z_1, \dots, z_n$ in $D\setminus A_m$. 

Then there exists a closed path $\gamma \subseteq A$ with $\sum_{j=1}^{n}
\indx_{\gamma}(z_j) \text{ odd}$ and $\gamma \cap \bigcup_{j=1}^n\D(z_j, \eps_j) = \emptyset$ if for all $m \geq m_0$ the same holds for some $\gamma \in A_m$. 
\end{cor}

\paragraph{Simply-connected case: proof of Theorem \ref{thm:convbls}.}
First, \cite[Theorem 3]{ALS4} shows that the (sign) excursion decomposition of the metric graph GFF does converge to that of the continuum GFF in the unit disk. It is based on convergence results of \cite{LawlerFerreras2007RWLoopSoup, BrugCamiaLis2014RWLoopsCLE,LupuConvCLE,ALS2} and can be stated as follows.

\begin{theorem}[Convergence of the excursion decomposition]\label{thm:conva}
Let $\Phi_D$ be a zero boundary GFF on a bounded simply-connected domain $D$ and $\tilde \phi_m$ be a sequence of zero boundary metric graph GFFs on $\widetilde D_m$ that are coupled with the GFF $\Phi_D$ such that a.s., 
$\tilde \phi_m \to \Phi$ in $H^{-\varepsilon}(\C)$, for some $\eps>0$.
Further, take 
$(\widetilde C_k^m,  \tilde\nu_k^m,\tilde s_k^m)
_{k \geq 1}$ 
the excursion decomposition of $\tilde \phi_m$, where 
$\widetilde C_k^m$ are the connected components of 
$\{\tilde \phi_m \neq 0\}$, 
$\tilde \nu_k^m$ are given by $|\tilde \phi_m|$ restricted to $\widetilde C_k^m$, and $\tilde s_k^m$ are the sign of 
$\phi_m$ on $\widetilde C_k^m$. 
Moreover, the clusters are ordered by decreasing size of diameter.

For every $k>0$,  we have the joint convergence in probability of $ \widetilde C_k^{(m)}\to C_k$, 
$\tilde\nu_k^{(m)}\to \nu_k$ and $\tilde s_k^{(m)}\to s_k$ as $m\to \infty$, in the Hausdorff topology\footnote{This is the topology on closed sets given by the Hausdorff distance.} for the first component, and in the weak topology of measures for the second component.
\end{theorem}

Now, combined with \cite[Theorem 2]{ALS4}, which says that the excursion clusters are given by the critical Brownian loop soup clusters in the continuum and on the metric graph, this gives us the Hausdorff convergence of Brownian loop soup clusters in the case no boundary excursions are present. This convergence is not sufficient for the convergence of separation events in 
Theorem \ref{thm:convbls}, 
as clusters could pinch together in the limit. 
However, in the proof of Theorem \ref{thm:conva}
(i.e. \cite[Theorem 3]{ALS4}) more is argued - as explained in the proof, the same argument showing the convergence of cluster outer boundaries in \cite{BrugCamiaLis2014RWLoopsCLE,LupuConvCLE}, applies to show the convergence of cluster boundaries around any finite number of fixed points in the domain. 
In view of Lemma \ref{lem:topo}, this suffices to conclude the $a = 0$ case of Theorem \ref{thm:convbls}.

We now argue how to get the case $a > 0$. As after considering the boundary cluster, the other clusters are given by Brownian loop soup in the complementary domains, it suffices to prove the following.

\begin{lemma}
Suppose we are in the set-up of Theorem \ref{thm:convbls} with $a > 0$. Consider the boundary cluster that is given by the closed union of all critical Brownian loop soup clusters that intersect at least one boundary to boundary excursion, denote it by $\widetilde C_{a}^{m}$ for the metric graph, and by $C_{a}$ for the continuum.

Then for any $z_1, \dots, z_n \in D$, and any 
$\eps_1, \dots, \eps_n > 0$, 
we have that the probabilities
$$\tilde p^{\rm bd}_{a,m} := \P\left(\nexists \widetilde \gamma \in \widetilde C_{a}^m\text{ such that} \sum_{j=1}^{n}
\indx_{\widetilde \gamma}(z_j)
\text{ odd}
 \: \& \: \widetilde C_{a}^{m} \cap \bigcup_{j=1}^n\D(z_j, \eps_j) = \emptyset\right),$$  
converge, as $m \to \infty$, to 
$$p^{\rm bd}_a := \P\left(\nexists \gamma \in C_{a}\text{ such that} \sum_{j=1}^{n}
\indx_{\gamma}(z_j)
\text{ odd}
 \: \& \: C_{a} \cap \bigcup_{j=1}^n\D(z_j, \eps_j) = \emptyset\right).$$  
\end{lemma}

We again know a convergence on the side of the GFF. Indeed, the so-called First passage set (FPS) of level $\sqrt{2a}$ is given, both on the metric graph level and in the continuum, by taking a Poisson point process of Brownian excursions of intensity $a$ together with critical Brownian soup clusters that intersect at least one excursion. Its Hausdorff convergence was proved in \cite{ALS2}, but does not suffice for the same reasons as above - the convergence in Hausdorff distance of a set does not imply convergence of separation events of the type 'a cluster separates $z_1$ and $z_2$'. Yet again we can improve this convergence to include also those of boundaries around any finite number of fixed points $z_1, \dots, z_n$ using the inputs from \cite{BrugCamiaLis2014RWLoopsCLE}, which will suffice thanks to Corollary \ref{cor:topo}.

\begin{proof}
As mentioned, in \cite{ALS2} the Hausdorff convergence of 
$\widetilde C_{a}^m$ to $C_a$ was shown.
Because of this Hausdorff convergence, we always have that $p^{\rm bd}_a \geq \limsup_{m \to \infty}
\tilde p^{\rm bd}_{a,m}$. 
Thus, it remains to argue for the reverse inequality.
For this, thanks to Corollary \ref{cor:topo}, it suffices to lift this convergence to the simultaneous convergence of also the boundaries of the FPS around a finite number of fixed points  $z_1, \dots, z_n$. 

The idea is to just follow through the strategy for proving Theorem 2.1 in \cite{BrugCamiaLis2014RWLoopsCLE}. Adding the analogous conditions for the connected components of the complement around $z_1, \dots, z_n$, their Proposition 3.6 and Lemma 3.13 extend to include also the Hausdorff convergence of cluster boundaries around these points. Lemma 4.2 needs a small modification: in addition to the loops, one needs to also consider a finite number of boundary to boundary excursions. Appropriate versions of Lemmas 4.3 and 4.4 when we consider all clusters of the loop soup that touch at least one excursion together with excursions work. Thereafter the modified and extended version of Theorem 4.1, where we include the convergence of boundaries around $z_1, \dots, z_n$, can be concluded. 
The Sections 5 and 6 work verbatim and thus one can conclude their equivalent of Theorem 2.1, i.e. the convergence of the closure of loop soup clusters containing at least one excursion together with their boundaries around fixed points $z_1, \dots, z_n$, when restricting to non-microscopic loops. As adding microscopic loops can only make $\tilde p^{\rm bd}_{a,m}$ larger, we conclude that 
$\liminf_{m\to \infty} \tilde p^{\rm bd}_{a,m} 
\geq p^{\rm bd}_a$ as desired.

\end{proof}
\paragraph{Domain with $n$ holes: proof of Theorem \ref{thm:convbls2}.}
To argue the convergence of probabilities of separation events in $D_{\eps_1, \dots, \eps_n}(z_1,\dots,z_n)$, 
the domain with $n$ holes, we will use an absolute continuity argument. Recall the notation $C_k, \widetilde C_k^m$ for the clusters in the disk and $C_k^{\eps_1, \dots, \eps_n}, \widetilde C_k^{\eps_1, \dots, \eps_n,m}$ for the clusters in the disk with holes.

\begin{lemma}
Suppose we are in the set-up of Theorems \ref{thm:convbls} and \ref{thm:convbls2}.
For $\eta > 0 $ very small, on the metric graph 
$\widetilde D^m$, $\widetilde D_{\eps_1, \dots, \eps_n}^m$ respectively, let
$$\widetilde F_m^{\eta} := \{\exists i, \widetilde \gamma \in \widetilde C_i^{ m},\text{ such that} \sum_{j=1}^{n}
\indx_{\widetilde \gamma}(z_j)
\text{ odd}
 \: \& \: \widetilde C_i^{m} \cap \bigcup_{j=1}^n\D(z_j, \eps_j+\eta) = \emptyset\},$$
and
$$\widetilde F_m^{\eta,\rm{mult}} := \{\exists i, \widetilde \gamma \in \widetilde C_i^{\eps_1, \dots, \eps_n, m}\text{ such that} \sum_{j=1}^{n}
\indx_{\widetilde \gamma}(z_j)
\text{ odd}
\: \& \: \widetilde C_i^{\eps_1, \dots, \eps_n, m} \cap \bigcup_{j=1}^n\D(z_j, \eps_j+\eta) = \emptyset
\}.$$
Define similarly the events 
$F^{\eta}, F^{\eta,\rm{mult}}$ in the continuum.
We have that:
$$\P(\widetilde F_m^\eta) 
= \E[\1_{\widetilde F_m^{\eta,\rm{mult}}}\widetilde R_\eta^m]$$
and
$$\P(F^\eta) = \E[\1_{F^{\eta,\rm{mult}}}R_\eta ],$$
where the Radon-Nikodym derivatives $\widetilde R_\eta^m, R_\eta$ are given as follows. 

Condition on all the loops and excursions of the Brownian loop soup and PPP of excursions that stay inside
$\widetilde{D}_{\eps_1+\eta, \dots, \eps_n+\eta}^m$,
respectively
$D_{\eps_1+\eta, \dots, \eps_n+\eta}(z_1,\dots,z_n)$. 
Then $\widetilde R_\eta^m, R_\eta$ are given by the conditional probability that the union of loops and excursions of the PPP that hit 
$\bigcup_{j=1}^n \D(z_j, \eps_j)$ 
intersects none of the odd clusters formed by the above loops and excursions. 

The only difference between $\widetilde R_\eta, R_\eta$ is that in one case we consider the loop soups and boundary excursions on the metric graph, in the other case in the continuum.
\end{lemma}

\begin{proof}
This follows from the fact that the Brownian loop soup and the PPP of outer boundary excursions in the domain $D$ can be sampled in two steps: we first sample it in the domain $D_{\eps_1, \dots, \eps_n}(z_1,\dots,z_n)$, and in the second step sample all the loops that also touch the closure of $\bigcup_{j=1}^n \D(z_j, \eps_j)$. It then remains to observe that 
\begin{itemize}
    \item if there is a cluster containing a path with odd index inside 
    $D_{\eps_1+\eta, \dots, \eps_n+\eta}(z_1,\dots,z_n)$, then it is already formed in the first step, 
    \item and a given separation cluster will remain a separation cluster after the second step if and only if it is not touched by any of the additional loops touching $\bigcup_{j=1}^n \D(z_j, \eps_j)$.
    \end{itemize}
\end{proof}

Given this lemma, we can deduce Theorem \ref{thm:convbls2} from Theorem \ref{thm:convbls}.

\begin{proof}[Proof of Theorem \ref{thm:convbls2}]
The proof follows in two steps, first the case of fixed 
$\eta > 0$ and then we let $\eta \to 0$ to conclude.

\paragraph{The case $\eta > 0$.}
Using the notations of the lemma above, we first argue that for any fixed $\eta > 0$, we have that 
$\P(\widetilde F_m^{\eta,\rm{mult}})$ converges to 
$\P(F^{\eta,\rm{mult}})$. 
First, by Theorem \ref{thm:convbls}, we already know the convergence of $\P(\widetilde F_m^{\eta}) \to \P(F^{\eta})$, and moreover we can couple the loop soups on the metric graphs and continuum such that all clusters surrounding at least one of $\D(z_j, \eps_j)$ converge almost surely in Hausdorff distance and so that their boundaries towards any $z_1, \dots, z_n$ also converge in Hausdorff distance. Hence it remains to just argue the convergence of the Radon-Nikodym derivatives, that are in any case bounded by $1$. But this follows directly from the convergence of macroscopic loops and excursions under the Brownian loop measure and the excursion measure, 
from metric graphs to the continuum; 
see e.g. \cite[Lemma 4.6]{ALS2}.

\paragraph{Letting $\eta \to 0$.}

We now argue that by letting $\eta > 0$, we can deduce the convergence 
$\tilde p^{\rm mult}_{a,m} \to p^{\rm mult}_a$ 
and conclude Theorem \ref{thm:convbls2}. 
By the first part, the fact that there are only finitely many clusters of diameter $\eps > 0$ and the fact that almost surely no cluster touches the inner boundaries, we have the limit 
$$\lim_{\eta \to 0}\lim_{m\to\infty} \P(\widetilde F_m^{\eta, \rm{mult}}) = 1-p_a^{\rm mult}.$$
In order to conclude 
$\tilde p_{a,m}^{\rm mult} \to p_a^{\rm mult}$, 
it suffices to be able to exchange the limits, as 
$1-\tilde{p}_{a,m}^{\rm mult} = \lim_{\eta \to 0}\P(\widetilde F_m^{\eta, \rm{mult}})$.

To do this, it suffices to show uniformity in $m$ for the convergence as $\eta \to 0$: there exists $f(\eta) $ with $f(\eta) \to 0 $ as $\eta \to 0 $, such that 
$|\P(\widetilde F_m^{\eta, \rm{mult}})  
- (1-p_{a,m}^{\rm mult})| < f(\eta)$ for all $m$. 
This is encoded in the following uniform estimate, similar to \cite[Lemma 4.13]{ALS2}:

			\begin{align*}
			\lim_{\zeta\to 0}
            \sup_{m\in\mathbb{N}\setminus\{0\}}
			\mathbb{P}\left (\exists \widetilde C_i^{\eps_1, \dots, \eps_n,m},
			d\Big(\widetilde C_i^{\eps_1, \dots, \eps_n,m},\bigcup_{j=1}^n \D(z_j, \eps_j)\Big)\leq \zeta, \exists \widetilde \gamma \subseteq \widetilde C_i^{\eps_1, \dots, \eps_n,m}, \sum_{j=1}^{n}
\indx_{\widetilde \gamma}(z_j)
\text{ odd}\right )=0.
			\end{align*}
We now consider two different cases. The first case is that this cluster $C_i^{\eps_1, \dots, \eps_n,m}$ also contains a point that is macroscopically away from the disks. In this case we are exactly in the context of  
\cite[Lemma 4.13]{ALS2}, and the uniform estimate follows from that Lemma.
The other case is that the macroscopic cluster with odd index remains close to the union of the disks, which means remains close to exactly one of the disks $\D(z_j, \eps_j)$. We will in fact show that there is no cluster remaining close to any of these disks. More formally, for all $j \in\{1, \dots, n\}$,
\begin{align*}
			\lim_{\zeta\to 0}
            \sup_{m\in\mathbb{N}\setminus\{0\}}
			\mathbb{P}\left (\exists \widetilde C_i^{\eps_1, \dots, \eps_n,m} \text{ s.t. }
			\widetilde C_i^{\eps_1, \dots, \eps_n,m} \subseteq \D(z_j, \eps_j+ \zeta)\right )=0.
			\end{align*}
To prove this, one can follow exactly the same strategy as the proof of \cite[Lemma 4.13]{ALS2}, 
with one modification that we will explain now. 
The rest of the details are the same and thus not repeated. In that lemma one argues that there are no clusters that come infinitesimally close but also go macroscopically away from the boundary as follows. One adds boundary-to-boundary excursions of intensity $0 < u \ll 1$, and argues that as $u \to 0$ the boundary cluster - that contains any cluster coming infinitesimally close - is uniformly not macroscopic. This is done by using the identification of this boundary cluster with the First passage set of the GFF, using the convergence of FPS 
\cite[Proposition 4.7]{ALS2},
and the fact that as $u \to 0$, the FPS converges to the boundary. The only modification we need to make is that we want to consider a different type of FPS that catches the cluster surrounding $\D(z_j, \eps_j)$. To do this pick two diametrically opposite arcs of length, say, $\eps_j/10$ on the boundary of $\D(z_j, \eps_j)$, and now consider the FPS that has boundary condition $0 < u \ll 1$ on both of these segments. The convergence of the FPS from metric graph to the continuum still follows from 
\cite[Proposition 4.7]{ALS2}. 
In fact by \cite[Corollary 4.12]{ALS2}.
Further, as $u \to 0$ the continuum FPS minus the boundary of the circle will converge to the two disjoint boundary segments - this follows from the monotonicity of FPS \cite[Proposition 4.5]{ALS1}, the fact that FPS of level $0$ is by definition just the boundary component, and  \cite[Corollary 4.12]{ALS1} which says that the continuum FPS will stay at positive distance from any compact set in the complement of the arcs in 
$\partial \D(z_j, \eps_j)$. The crux is that in the case there was a cluster surrounding $\D(z_j, \eps_j)$ and coming close to its boundary, the FPS would need to be connected. Based on above, this cannot happen with positive probability. See \cite[Lemma 4.13]{ALS2} for details.

\end{proof}

\subsection{Topological events for loop soup clusters}
\label{Subsec topo}

In this section we will work in the following context. For $n > 0$ we let $z_1, \dots, z_n \in D$ distinct points and $\eps_1, \dots, \eps_n > 0$, let $D_{\eps_1,\dots, \eps_n}(z_1, \dots, z_n)$ be the $n+1$-connected domain $D \setminus \bigcup_{j = 1}^n \overline{\D(z_j,\eps_j)}$. 
Let $\widetilde D_{\eps_1, \dots, \eps_n}$ be some corresponding approximating metric graph. 

We first state the main result of \cite{LupuTopoAIHP} on the metric graphs, then discuss how to extend it to include boundary-to-boundary excursions, and finally deduce -- using the results of previous section -- the equivalent results in the continuum limit.

\subsubsection{The case of the metric graphs}

Consider the metric graph loop soup 
$\widetilde \cL_{\widetilde D_{\eps_1, \dots, \eps_n}}$ at the critical intensity
($\alpha=1/2$), and let
$(\widetilde C_k^{\eps_1,\dots,\eps_n})_{k\geq 1}$ denote its clusters.
Further, let $\tilde \gamma$ denote a closed path on the metric graph.
One of the main results of \cite{LupuTopoAIHP} can be stated as follows. 

\begin{theorem}[Lupu, \cite{LupuTopoAIHP}] \label{thm:topevent}
Let $n\geq 1$ be an integer. The following equality holds:
\begin{multline*}
-\log\P
\left(
\nexists \widetilde C_i^{\eps_1,\dots,\eps_n}
\text{ with a closed path }
\tilde \gamma \subset \widetilde C_i^{\eps_1,\dots,\eps_n}
\text{ such that }
\sum_{j=1}^{n}
\indx_{\tilde \gamma}(z_j)
\text{ is odd}
\right)
=
\\
\tilde\mu^{\rm loop}_{\widetilde D_{\varepsilon_1,\ldots,\varepsilon_n}}
\Big(
\sum_{j=1}^n \indx_{\tilde\gamma}(z_j)
\text{ odd}
\Big).
\end{multline*}
\end{theorem}

This is a restatement of \cite[Theorem 1]{LupuTopoAIHP} in our more specific planar setting.
That theorem is stated in \cite{LupuTopoAIHP} in the language of gauge fields
with values in $\{-1,1\}$,
and does not require planarity.
According to it,
the probability that all the clusters formed by the metric graph loop soup are trivial for a fixed gauge field equals the exponential of minus the mass of loops with non-trivial holonomy for the gauge field.
There is a doubling of exponents going on,
on the loop soup side the intensity parameter being $\alpha=1/2$,
and on the loop measure side $\alpha=1=2\times 1/2$.
In our particular planar setting here, 
the gauge field is represented by $n$ defect lines, 
each joining a circle $\partial \D(z_j,\eps_j)$
to the boundary of the domain $\partial D$.
One multiplies the holonomy by $-1$ each time one crosses a defect line.
In this way, the holonomy of a closed loop $\tilde\gamma$ equals
$-1$ to the power $\sum_{j=1}^n \indx_{\tilde\gamma}(z_j)$.
So it is non-trivial if and only if
$\sum_{j=1}^n \indx_{\tilde\gamma}(z_j)$ is odd.

\paragraph{With boundary conditions.}

Further, for $v\geq 0$, denote by
$\widetilde \Xi^{v^2/2}_{\widetilde D,\widetilde D_{\varepsilon_1,\dots,\varepsilon_n}}$
the Poisson point process of boundary excursions in
$\widetilde D$ with intensity
$$
\dfrac{v^2}{2}
\1_{\tilde\gamma\subset \widetilde D_{\varepsilon_1,\dots,\varepsilon_n}}
\tilde\mu^{\rm exc}_{\widetilde D}(d\tilde\gamma).
$$
Here $\widetilde D$ is a metric graph approximation of the domain without holes $D$,
and $\widetilde D_{\varepsilon_1,\dots,\varepsilon_n}\subset \widetilde D$.
All the excursions in 
$\widetilde \Xi^{v^2/2}_{\widetilde D,\widetilde D_{\varepsilon_1,\dots,\varepsilon_n}}$
have both endpoints in $\partial \widetilde D$,
which is also the outer boundary of $\widetilde D_{\varepsilon_1,\dots,\varepsilon_n}$.
There is no excursion ending at an inner boundary component of
$\widetilde D_{\varepsilon_1,\dots,\varepsilon_n}$.

We take the Brownian loop soup and the PPP of boundary excursions to be independent of each other, and let
$(\widetilde C_k^{\eps_1,\dots,\eps_n,v})_{k\geq 1}$ denote the clusters formed by the union of the two families of trajectories:
\[
\widetilde \cL_{\widetilde D_{\eps_1, \dots, \eps_n}}
\cup
\widetilde \Xi^{v^2/2}_{\widetilde D,\widetilde D_{\varepsilon_1,\dots,\varepsilon_n}}.
\]
In terms of random walk representations of the GFF
(also called ``isomorphism theorems''),
this collection of trajectories corresponds to a metric graph GFF on 
$\widetilde D_{\varepsilon_1,\dots,\varepsilon_n}$
with boundary condition $v$ on the outer boundary,
and $0$ on the inner boundary components;
see \cite[Section 2.3]{ALS2}.
Further, observe that given $\tilde \gamma$ an excursion from $\partial \widetilde D$ to 
$\partial \widetilde D$ in $ \widetilde D$,
since $\tilde \gamma$ is not necessarily a closed loop,
the indices
$\indx_{\widetilde \gamma}(z_i)$ are not defined. 
However, if $n$ is \textbf{even}, the parity of the sum
$\sum_{j =1}^n \indx_{\tilde \gamma}(z_j)$
is defined in a unique way for every $\tilde \gamma \subseteq \widetilde D_{\varepsilon_1,\dots,\varepsilon_n}$ with both endpoints in $\partial \widetilde D$.
Indeed, one can close the loop by adding an arc of
$\partial \widetilde D$ that joins the two endpoints of
$\tilde \gamma$.
There are actually two choices of arcs, but both
give the same parity for the sum.
This is because for the loop corresponding to 
$\partial \widetilde D$,
the sum of indices is $n$, which is even.

Then, the previous theorem can be easily extended to include also boundary-to-boundary excursions
\begin{prop}
\label{Prop n pt cut disks bc metric}
Assume that $n$ is \textbf{even}.
Then the following equality holds:
\begin{multline*}
-\log\P
\left(
\nexists \widetilde C_i^{\eps_1,\dots,\eps_n,v}
\text{ with a closed path }
\tilde \gamma \subset \widetilde C_i^{\eps_1,\dots,\eps_n,v}
\text{ such that }
\sum_{j=1}^{n}
\indx_{\tilde \gamma}(z_j)
\text{ is odd}
\right)
=
\\
\tilde\mu^{\rm loop}
_{\widetilde D_{\varepsilon_1,\ldots,\varepsilon_n}}
\Big(
\sum_{j=1}^n \indx_{\tilde\gamma}(z_j)
\text{ odd}
\Big)
+
v^2
\tilde\mu^{\rm exc}_{\widetilde D}
\Big(
\sum_{j=1}^n \indx_{\tilde\gamma}(z_j) \text{ odd},
\tilde\gamma\subset 
\widetilde D_{\varepsilon_1,\ldots,\varepsilon_n}
\Big)
.
\end{multline*}
\end{prop}

\begin{proof}
In the metric graph 
$\widetilde D_{\varepsilon_1,\dots,\varepsilon_n}$,
we identify all the points of the outer boundary
$\partial \widetilde D$
into one single point,
which we denote $\dagger$.
We will denote by 
$\widetilde D_{\varepsilon_1,\dots,\varepsilon_n}^\dagger$
the corresponding quotient metric graph.
The boundary of $\widetilde D_{\varepsilon_1,\dots,\varepsilon_n}^\dagger$
now corresponds to only the inner boundary components of
$\widetilde D_{\varepsilon_1,\dots,\varepsilon_n}$:
we see $\dagger$ as an inner point of
$\widetilde D_{\varepsilon_1,\dots,\varepsilon_n}^\dagger$.
Let $\tilde\phi^\dagger$
be the metric graph GFF on 
$\widetilde D_{\varepsilon_1,\dots,\varepsilon_n}^\dagger$
with $0$ boundary condition.
Then the metric graph GFF on
$\widetilde D_{\varepsilon_1,\dots,\varepsilon_n}$
with $0$ boundary conditions on inner components of
$\partial \widetilde D_{\varepsilon_1,\dots,\varepsilon_n}$,
and condition $v$ on
$\partial D$,
corresponds, by quotient,
to $\tilde\phi^\dagger$ conditioned on the event
$\{\tilde\phi^\dagger(\dagger)=v\}$.
Then, by random walk representations of the GFF
(see e.g. \cite[Proposition 2.4]{ALS2}),
the probability we are interested in equals the
following conditional probability:
$$
\P\Big(
\nexists \widetilde C
\text{ sign cluster of }
\tilde\phi^\dagger
\text{ with a closed path }
\tilde \gamma \subset \widetilde C
\text{ such that }
\sum_{j=1}^{n}
\indx_{\tilde \gamma}(z_j)
\text{ is odd}
\Big\vert 
\tilde\phi^\dagger(\dagger)=v
\Big).
$$
By the symmetry of the law of 
$\tilde\phi^\dagger$,
this is the same as
$$
\P\Big(
\nexists \widetilde C
\text{ sign cluster of }
\tilde\phi^\dagger
\text{ with a closed path }
\tilde \gamma \subset \widetilde C
\text{ such that }
\sum_{j=1}^{n}
\indx_{\tilde \gamma}(z_j)
\text{ is odd}
\Big\vert 
\vert\tilde\phi^\dagger(\dagger)\vert =v
\Big).
$$
Then \cite[Theorem 1]{LupuTopoAIHP} (topological probabilities) has a conditional version where there is a conditioning on the absolute value of the free field.
This is explained in \cite[Section 4]{LupuTopoAIHP},
and is related to the Ising model.
In this way, the above probabilities equal
\begin{equation}
\label{Eq Cond loc time dagger}
\P\Big(
\nexists \tilde \gamma
\in
\widetilde{\cL}_{\widetilde D_{\varepsilon_1,\dots,\varepsilon_n}^\dagger}
\text{ such that }
\sum_{j=1}^{n}
\indx_{\tilde \gamma}(z_j)
\text{ is odd}
\Big\vert
L_\dagger(\widetilde{\cL}_{\widetilde D_{\varepsilon_1,\dots,\varepsilon_n}^\dagger}) = v^2/2
\Big)^2,
\end{equation}
where $\widetilde{\cL}_{\widetilde D_{\varepsilon_1,\dots,\varepsilon_n}^\dagger}$
is the critical metric graph loop soup on
$\widetilde D_{\varepsilon_1,\dots,\varepsilon_n}^\dagger$,
and $L_\dagger(\widetilde{\cL}_{\widetilde D_{\varepsilon_1,\dots,\varepsilon_n}^\dagger})$
is its local time in $\dagger$.
Notice also the square above.
Now, 
$\widetilde{\cL}_{\widetilde D_{\varepsilon_1,\dots,\varepsilon_n}^\dagger}$
consists of two types of trajectories:
\begin{itemize}
    \item The loops that do not visit $\dagger$:
    these are distributed as $\widetilde{\cL}_{\widetilde D_{\varepsilon_1,\dots,\varepsilon_n}}$,
    and are independent from $L_\dagger(\widetilde{\cL}_{\widetilde D_{\varepsilon_1,\dots,\varepsilon_n}^\dagger})$.
    \item The loops that do visit $\dagger$.
    Conditionally on 
    $L_\dagger(\widetilde{\cL}_{\widetilde D_{\varepsilon_1,\dots,\varepsilon_n}^\dagger}) = v^2/2$,
    these can be decomposed into a PPP of excursions from
    $\dagger$ to $\dagger$,
    which, when lifted to 
    $\widetilde D_{\varepsilon_1,\dots,\varepsilon_n}$,
    is the same as
    $\widetilde \Xi^{v^2/2}_{\widetilde D,\widetilde D_{\varepsilon_1,\dots,\varepsilon_n}}$.
\end{itemize}
Moreover, the two families of trajectories are independent.
So \eqref{Eq Cond loc time dagger} equals
$$
\P\Big(
\nexists \tilde \gamma
\in
\widetilde{\cL}_{\widetilde D_{\varepsilon_1,\dots,\varepsilon_n}}
\text{ s.t. }
\sum_{j=1}^{n}
\indx_{\tilde \gamma}(z_j)
\text{ is odd}
\Big)^2
\P\Big(
\nexists \tilde \gamma
\in
\widetilde \Xi^{v^2/2}_{\widetilde D,\widetilde D_{\varepsilon_1,\dots,\varepsilon_n}}
\text{ s.t. }
\sum_{j=1}^{n}
\indx_{\tilde \gamma}(z_j)
\text{ is odd}
\Big)^2
.
$$
This proves the Proposition.
\end{proof}

\subsubsection{Identities in the continuum}

We now state the key ingredient behind our calculations in the continuum. Denote by 
$\cL^{1/2}_{D_{\varepsilon_1,\ldots,\varepsilon_n}(z_1,\ldots,z_n)}$ the Brownian loop soup in 
$D_{\varepsilon_1,\ldots,\varepsilon_n}(z_1,\ldots,z_n)$
with intensity
$\frac{1}{2}
\mu^{\rm loop}_{D_{\varepsilon_1,\ldots,\varepsilon_n}(z_1,\ldots,z_n)}$.
For $v\geq 0$, denote by
$\Xi^{v^2/2}_{D, D_{\varepsilon_1,\ldots,\varepsilon_n}(z_1,\ldots,z_n)}$
the Poisson point process of boundary excursions in
$D$ with intensity
$$
\dfrac{v^2}{2}
\1_{\gamma\subset D_{\varepsilon_1,\ldots,\varepsilon_n}(z_1,\ldots,z_n)}
\mu^{\rm exc}_{D}(d\gamma).
$$
We also take 
$\Xi^{v^2/2}_{D, D_{\varepsilon_1,\ldots,\varepsilon_n}(z_1,\ldots,z_n)}$
to be independent from
$\cL^{1/2}_{D_{\varepsilon_1,\ldots,\varepsilon_n}(z_1,\ldots,z_n)}$.

As above, given $\gamma$ an excursion from $\partial D$ to $\partial D$ in $D$,
the indices
$\indx_{\gamma}(z_i)$
are not uniquely defined but,
when $n$ is \textbf{even}, the parity of the sum 
$\sum_{j=1}^n \indx_{\gamma}(z_j)$ is well defined.

 Combining the result on the metric graphs, Proposition \ref{Prop n pt cut disks bc metric}, and the convergence result, Theorem \ref{thm:convbls2}, from the previous section, we obtain the following.

\begin{prop}
\label{Prop n pt cut disks bc continuum}
Assume that $v = 0$ and $n \geq 1$ or $v \neq 0$ and $n$ is \textbf{even}.
The following equality holds:
\begin{multline*}
\P
\Big(
\nexists C
\text{ cluster of } 
\cL_{D_{\varepsilon_1,\ldots,\varepsilon_n}(z_1,\ldots,z_n)}
\cup
\Xi^{v^2/2}_{D, D_{\varepsilon_1,\ldots,\varepsilon_n}(z_1,\ldots,z_n)}
\text{ with a closed path } \gamma\subset C
\text{ s.t. }
\sum_{j=1}^n \indx_{\gamma}(z_j) \text{ odd}
\Big)
\\
=
\exp\Big(
-
\mu^{\rm loop}_{D_{\varepsilon_1,\ldots,\varepsilon_n}(z_1,\ldots,z_n)}
\Big(
\sum_{j=1}^n \indx_{\gamma}(z_j)
\text{ odd}
\Big)
-
v^2
\mu^{\rm exc}_{D}
\Big(\sum_{j=1}^n \indx_{\gamma}(z_j) \text{ odd},
\gamma\subset D_{\varepsilon_1,\ldots,\varepsilon_n}(z_1,\ldots,z_n)
\Big)
\Big)
.
\end{multline*}
In particular, in the case $n=2$,
\begin{multline*}
-\log\P
\Big(
\text{No cluster of } 
\cL_{D_{\varepsilon_1,\varepsilon_2}(z_1,z_2)}
\cup
\Xi^{v^2/2}_{D, D_{\varepsilon_1,\varepsilon_2}
(z_1,z_2)}
\text{ disconnects } z_1 \text{ from } z_2
\Big)
=
\\
\mu^{\rm loop}
_{D_{\varepsilon_1,\varepsilon_2}(z_1,z_2)}
\Big(
\indx_{\gamma}(z_1)+\indx_{\gamma}(z_2)
\text{ odd}
\Big)
+
v^2
\mu^{\rm exc}_{D}
\Big(
\indx_{\gamma}(z_1)+\indx_{\gamma}(z_2) \text{ odd},
\gamma\subset 
D_{\varepsilon_1,\varepsilon_2}(z_1,z_2)
\Big)
.
\end{multline*}
\end{prop}

\subsection{Some calculations using Jacobi theta functions and elliptic integrals}

Here we collect some small relations on Jacobi theta functions that we did not find explicitly in the literature. 

\begin{claim}
Let $q \in (0,1)$ be a nome. Then it holds that
\begin{equation}\label{eq:thetasquares}
    \theta_2(q)^2+\theta_3(q)^2 = \theta_3(q^{1/2})^2
\end{equation}
and
\begin{equation}\label{eq:thetalinear}
\theta_2(q) + \theta_3(q) = \theta_3(q^{1/4}) .
\end{equation}
\end{claim}
\begin{proof}
The first follows by combining formulas 1.4.43 and 1.4.44 from \cite{lawden2013elliptic}. The second follows directly from the definition
\end{proof}

\begin{lemma}\label{lem:parametr}
Let $z_1 \neq z_2$ be two points in a simply-connected domain $D$.
The equation 
$$\exp(\pi G_D(z_1,z_2)) = \frac{\theta_3(q^{1/2})}{\theta_2(q^{1/2})}$$
determines uniquely the nome $q \in (0,1)$ which satisfies 
$\theta_2(q)/\theta_3(q) = r$, where $r > 0$ is defined as follows: there is a unique conformal mapping 
from $(D, z_1, z_2)$ to the disk with two marked points $(\D, -r,r)$.
\end{lemma}
\begin{proof}
It is known that there is a unique correspondence between a nome $q \in (0,1)$ and the elliptic parameter $k \in (0,1)$. Indeed, we have $k = \theta^2_2(q)/\theta^2_3(q)$ and in the other direction $q=\exp(-\pi K'(k)/K(k))$ - see Sections 1 and 2 in \cite{lawden2013elliptic}. 
As $\exp(\pi G_D(z_1, z_2)) \in (1, \infty)$, 
this implies that there is a bijective correspondence.

To verify the relation to $r$, we observe that by conformal invariance we can use the disk domain $(\D, -r, r)$, 
where explicitly 
$\exp(2\pi G_\D(-r,r)) = \frac{1+r^2}{2r}$.
We can now apply Equation \eqref{eq:thetasquares}
and  
\begin{equation}\label{eq:thetadupl}
2\theta_2(q)\theta_3(q) = \theta_2(q^{1/2})^2,
\end{equation}
which can be found as Equation 1.4.28 in \cite{lawden2013elliptic}.
\end{proof}


\section{Twist field correlations via the Brownian loop soup}
\label{Sec Twist loop soup}

In this section we will define what we call the twist field correlations in simply connected domains using two different regularisations of the Brownian loop soup.
We will also prove the conformal covariance of these correlations and finally provide a calculation of two-point correlations in the disk. 

These correlations have been defined using a different cut-off - the cut-off in diameter of loops - also in \cite{CamiaGandolfiKleban16}, where it would be the winding field of the Brownian loop soup with parameter $\beta = \pi$. An underlying random field has been constructed in \cite{BrugCamiaLis2018winding}, again using different cut-offs.

We will first give two equivalent definitions of the zero-boundary twist field correlations using first the cut-off in time, 
and then using the Brownian loop soup restricted to loops that do not touch small disks around the marked points. As mentioned, yet another way to renormalise would be to use a cut-off by diameter for the loops, this is done in \cite[Section 4.2]{CamiaGandolfiKleban16} in a larger generality, and it seems clear that it would, 
up to a universal constant, give the same limit. 

The relevance of the two chosen cut-offs is as follows: the cut-off in time allows for explicit calculations done at the end of this Section \ref{Sec Twist loop soup}, 
and the cut-off by small disks helps build a bridge to CLE$_4$ in the next Section \ref{Sec Twist CLE4}. 
Diffeomorphism invariance and Weyl scaling will be simple consequences of the constructions.

Thereafter we will explain how to define the twist-field correlations with non-zero constant boundary conditions. We recall the notation: $D_{\eps_1, \dots, \eps_n}(z_1, \dots, z_n) := D \setminus (\D(z_1,\eps_1) \cup \dots \cup \D(z_n, \eps_n))$.

\begin{theorem}[Definition of twist fields correlations with zero boundary condition]\label{thm:deftwist}
The (zero boundary) twist field correlations $\sigma_{\rm tw}^D(z_{1}, z_{2},\dots, z_{n})$ can be defined via three equivalent ways by taking the following limits. 
\begin{enumerate}[i)]

\item As $\delta\to 0$, we consider the limit of
\begin{equation}\label{def:twist1bd}
\delta^{-n/16}\exp(n\cbls)\exp\left(-
\mu^{\rm loop}_{D}
\Big(\sum_{j=1}^{n}
\indx_{\gamma}(z_j)
\text{ odd},
t_{\gamma}>\delta
\Big)\right),
\end{equation}
where $\cbls$ is an explicit constant 
that can be expressed through the Brownian loop measures
in the unit disk $\D$ and its complement
$\C\setminus\overline{\D}$:
$$\cbls
=
\mu^{\rm loop}_{\D}(t_{\gamma}>1,
\indx_{\gamma}(0) \text{ odd})
-
\mu^{\rm loop}_{\C}(
\gamma \cap (\C\setminus\overline{\D})\neq\emptyset,
t_{\gamma}<1,
\indx_{\gamma}(0) \text{ odd}).
$$
\item
As $\max_{1\leq j\leq n}\varepsilon_j\to 0$, we consider the limit of
\begin{equation}\label{def:twist2bd}
\prod_{j=1}^{n}
\Big(
\frac{1}{\eps_j^{1/8}\sqrt{|\log \eps_j|}}
\Big)
\left(\pi/2\right)^{n/2}\exp\left(-\mu^{\rm loop}_{D_{\eps_1, \dots, \eps_n}(z_1, \dots, z_n)}
\Big(\sum_{j=1}^{n}
\indx_{\gamma}(z_j)
\text{ odd}
\Big)\right).
\end{equation}
\item Consider the Brownian loop soup 
$\cL_{D_{\varepsilon_1,\dots, \varepsilon_n}(z_1,\dots, z_n)}$
in the domain with holes
$D_{\varepsilon_1,\dots, \varepsilon_n}(z_1,\dots, z_n)$,
and let $(\overline C_i^{\varepsilon_1,\dots, \varepsilon_n})_{i \geq 1}$ denote the collection of topological closures of its clusters. 
(It is an infinite countable collection of random compact disjoint subsets in $D$.) 
As $\max_{1\leq j\leq n}\varepsilon_j\to 0$, we consider the limit of
\begin{equation}\label{def:twist3bd}
\prod_{j=1}^{n}
\Big(
\frac{1}{\eps_j^{1/8}\sqrt{|\log\eps_j|}}
\Big)
\left(\pi/2\right)^{n/2}\P\left(\nexists \text{ path }\gamma\subseteq
\bigcup_{i \geq 1}
\overline{C}_i^{\varepsilon_1,\dots, \varepsilon_n}
\text{ with }\sum_{j=1}^{n}
\indx_{\gamma}(z_j)
\text{ odd}
\right).
\end{equation}
\end{enumerate}
\end{theorem}

\begin{remark}
A few remarks are in order.
\begin{itemize}
\item Here there is an implicit normalisation of the 
1-point function, which is 
$\sigma_{\rm tw}^\D(0) = 1$,
$\D$ being the unit disk. 
\item There is also an implicit choice of metric in the definition: $\sigma_{\rm tw}^D = 
\sigma_{\rm tw}^{D,g_0}$ where $g_0$ is the usual Euclidean metric on the domain. To introduce the metric parameter $\sigma_{\rm tw}^{D,g}$, one just considers the Brownian motion w.r.t the metric $g$ in the above definitions.
\item The second point could also be equivalently stated using the Brownian loop soup in the domain $D$, by restricting to loops that avoid the $\eps_j$-disks:
$$ \mu^{\rm loop}_{D_{\eps_1, \dots, \eps_n}(z_1, \dots, z_n)}\Big(\sum_{j=1}^{n}
\indx_{\gamma}(z_j)
\text{ odd}\Big) = \mu^{\rm loop}_{D}
\Big(\sum_{j=1}^{n}
\indx_{\gamma}(z_j)
\text{ odd},
\gamma\cap\Big(\bigcup_{j=1}^{n}\D(z_j,\varepsilon_j)\Big)
=\emptyset
\Big).$$
\end{itemize}
\end{remark}

 An easy corollary of the calculations leading to the theorem are the diffeomorphism invariance and Weyl scaling, which when combined give rise also to conformal covariance. For a different cut-off regularisation, conformal covariance was shown also in 
 \cite[Section 4.2]{CamiaGandolfiKleban16}. 

\begin{prop}[Diffeomorphism invariance, Weyl scaling and conformal covariance]
The twist fields $\sigma_{\rm{tw}}^{D}(z_{1}, z_{2},\dots, z_{n})$
satisfy the following properties.
\begin{enumerate}[i)]
\item Diffeomorphism invariance: Let $(D_1, g_1)$ and $(D_2, g_2)$ be two Riemann surfaces with boundary and $\psi: D_1 \to D_2$ a diffeomorphism such that $g_1 := \psi^*g_2$. Then 
$$\sigma_{\rm tw}^{D_1, g_1}(z_1, \dots, z_n) 
= \sigma_{\rm tw}^{D_2, g_2}(\psi(z_1), \dots, \psi(z_n)). 
$$ 
\item Weyl scaling: Let $\Omega: D \to \R_+$. Then 
$$\sigma_{\rm tw}^{D,g\Omega^2}(z_1, \dots, z_n) 
= 
\Big(\prod_{j=1}^{n} \Omega(z_j)^{-1/8}\Big)
\sigma_{\rm tw}^{D, g}(z_1, \dots, z_n). $$
\item Conformal covariance: Let $D_1$ and $D_2$ be two simply connected domains,
$D_1,D_2\neq\C$.
Take a conformal mapping
$\psi:D_1\rightarrow D_2$.
Let $z_1,z_2,\dots,z_n\in D_1$ be
two by two distinct.
Then,
\begin{equation}
\label{Eq:covartau}
\sigma_{\rm tw}^{D_2}(\psi(z_{1}), \psi(z_{2}),\dots, \psi(z_{n}))
=
\Big(
\prod_{j=1}^{n} \vert\psi'(z_j)\vert^{-1/8}
\Big)
\sigma_{\rm tw}^{D_1}(z_{1}, z_{2},\dots, z_{n}).
\end{equation}
\end{enumerate}
\end{prop}

Importantly, we also manage to explicitly calculate the two-point correlations in simply-connected domains. 
For a related but different model, 
the so-called layering field that takes into account only the outer boundaries of Brownian loops and not the winding, 
this was achieved in \cite{CamiaGandolfiKleban16}.

\begin{theorem}[Calculation of twist two-point correlations]\label{thm:twopointtwist}
Let $z_1 \neq z_2 \in D$. Then
$$\sigma_{\rm tw}^D(z_1, z_2) = \Ctw\CR(z_1,D)^{-1/8}\CR(z_2,D)^{-1/8}\frac{1}{\theta_2(q^{1/4})\sqrt{|\log q|}},$$
where by Lemma \ref{lem:parametr} the nome $q \in (0,1)$ is fixed uniquely by 
$$\exp\Big(\pi G_D(z_1,z_2)\Big) 
= \dfrac{\theta_3(q^{1/2})}{\theta_2(q^{1/2})}.$$

\end{theorem}

The rest of the section is dedicated to proving these results and is structured as follows. 

In Section 3.1 we study the different cut-offs, prove Theorem \ref{thm:deftwist} and the conformal covariance. We start with the case of one point in both set-ups, and then show how to connect the multipoint definitions.

In Section 3.2 we derive the exact formula by mapping the set-up to a strip and doing an explicit calculation using the Brownian loop in the strip.

Finally, in Section 3.3 we consider twist fields with constant non-zero boundary conditions.

\subsection{Twist fields via different cut-offs: proof of Theorem \ref{thm:deftwist}}

The equivalence of ii) and iii) follows from the $v = 0$ case of Proposition \ref{Prop n pt cut disks bc continuum}. 

It thus remains to prove the equivalence between i) and ii), i.e. of the different cut-offs for the loops under the loop measure.

\subsubsection{The case of the disk with one puncture}

\paragraph{Cut-off by time}

Recall that given $\gamma$ a time-parametrized loops,
we will denote by
$t_{\gamma}$ its time-length and
for $z\in\C$,
$\indx_{\gamma}(z)\in\Z$
its index around the point $z$.

Consider the punctured disk $\D\setminus \{ 0\}$
and the mass of Brownian loops contained in $\D$,
of duration $t_{\gamma}>\delta$,
and with an odd index around the point $0$:
$$
\mu^{\rm loop}_{\D}
(t_{\gamma}>\delta, \indx_{\gamma}(0) \text{ odd}).
$$
We will prove the following asymptotic expansion when
$\delta\to 0$. 

\begin{prop}[One point expansion]
\label{Prop disk 1pt time}
As $\delta\to 0$, we have that
$$
\mu^{\rm loop}_{\D}
(t_{\gamma}>\delta, \indx_{\gamma}(0) \text{ odd})
=
\dfrac{1}{16}\log(\delta^{-1})
+
\cbls + o(1),
$$
where 
\begin{equation}
\label{Eq explicit cbls}
\cbls
=
\frac32\,\zeta_{\rm R}'(-1)
+\frac14\log \pi
-\frac1{48}\log 2
-\frac{\gamma_{\rm EM}}{16}.
\end{equation}
\end{prop}
\begin{remark}
In the first version of the article we used the following expression for the above constant.
\begin{equation}
\label{Eq C D 0 diff}
\cbls
=
\mu^{\rm loop}_{\D}(t_{\gamma}>1,
\indx_{\gamma}(0) \text{ odd})
-
\mu^{\rm loop}_{\C}(
\gamma \cap (\C\setminus\overline{\D})\neq\emptyset,
t_{\gamma}<1,
\indx_{\gamma}(0) \text{ odd}).
\end{equation}
In the opinion of the authors it is quite nice that this expression can be explicitly calculated.
\end{remark}
We will also need to do the same calculation in the case where the disk has a conical singularity.

\begin{prop}[Conical one point expansion]
\label{Prop disk 1pt time conic}
Let $g_{\rm cone}$ be the metric given by 
$\vert z\vert^{-1}|dz|^2$.
As $\delta\to 0$,
\begin{equation}\label{eq:timecutoffconic}
\mu^{\rm loop}_{\D,g_{\rm cone}}
(t_{\gamma}>\delta, \indx_{\gamma}(0) \text{ odd})
=
\dfrac{1}{8}\log(\delta^{-1})
+
\cblsg + o(1),
\end{equation}
where
\begin{equation}
\label{Eq explicit cblsg}
\cblsg
=
\frac14\log \pi
-\frac18\log 2
-\frac{\gamma_{\rm EM}}8 .
\end{equation}
\end{prop}
The proofs rely on calculations using spectral zeta functions. 

To set this up let \(\D_{{\rm cone},v}=\{0<r\le 1,\,0\le\theta<2\pi\}\) be the finite flat cone with metric
\[g_{{\rm cone}, v}=dr^2+v^{-2}r^2d\theta^2.\]
After a change of coordinate, the metric can also be given in the form $|z|^{2(v^{-1}-1)}|dz|^2$. Notice that the unit $z$-disk with metric $|z|^{-1}|dz|^2$ has $r \leq 2$.

Further let $\Delta_{v}$ denote the (Friedrichs extension of the) positive Dirichlet Laplacian
on the $\D_{{\rm cone},v}$; for $v= 1$ this amounts to the usual Dirichlet Laplacian. Finally, let $\zeta_{\Delta_v}$ denote the spectral zeta function associated to this Laplacian \cite{carleman1935,ms1949,cheeger1979}, given for $\Re(s)$ large enough by $\sum_n \lambda_n^{-s}$, where the sum is over the eigenvalues. This can be analytically extended to a meromorphic function on $\C$ that is regular at $0$ \cite{Spreafico05}. 
 
Our first lemma provides an expansion for the mass of loops of odd winding using the value and the derivative of the spectral zeta function at $0$ \cite{raysinger}. 
\begin{lemma}\label{lem:hkexpansion}
As $\delta \to 0$, we have the following expansions
$$
\mu^{\rm loop}_{\D}
(t_{\gamma}>\delta, \indx_{\gamma}(0) \text{ odd})= (\zeta_{\Delta/2}(0)-\frac{1}{2}\zeta_{\Delta_{1/2}/2}(0))\left(\log (\delta^{-1})-\gamma_{\rm EM}\right) + \zeta'_{\Delta/2}(0)-\frac{1}{2}\zeta'_{\Delta_{1/2}/2}(0) +o(1)$$
and
$$
\mu^{\rm loop}_{\D,g_{\rm cone}}
(t_{\gamma}>\delta, \indx_{\gamma}(0) \text{ odd})
= (\zeta_{\Delta_2/2}(0)-\frac{1}{2}\zeta_{\Delta/2}(0))\left(\log (\delta^{-1})-\gamma_{\rm EM}+2\log 2\right) + \zeta'_{\Delta_2/2}(0)-\frac{1}{2}\zeta'_{\Delta/2}(0) +o(1)$$
\end{lemma}
Second, we use explicit calculations for the spectral zeta function in the case of disk with conical metric \cite{Spreafico05}. Indeed, in this article the author gives explicit formulas on finite flat cones of $\zeta_{\Delta_v}(0)$ and $\zeta'_{\Delta_v}(0)$. Indeed, \cite[Theorem 1]{Spreafico05} gives that for all $v > 0$
\begin{equation}
\label{Eq Spreafico zeta zero revised}
\zeta_{\Delta_{v}}(0)
=
\frac1{12}\left(v+\frac1v\right).
\end{equation}
Further, \cite[Corollary 1]{Spreafico05} gives us
\begin{equation}
\label{Eq Spreafico derivative laplacian}
\zeta'_{\Delta}(0) = -\frac13\log 2 +\frac12\log(2\pi) +\frac5{12}+2\zeta_{\rm R}'(-1).
\end{equation}
For other derivatives the formulas are quite complicated, but it comes out that in the special cases $v \in \{1/2, 1, 2\}$ they can be further simplified \footnote{We are thankful to ChatGPT 5.5 for noticing this possible further simplification.}; below $\zeta_R$ is the usual Riemann zeta function. 
\begin{lemma}\label{lem:dercalcul}
We have that
\begin{align}
\label{Eq Spreafico values revised}
\zeta'_{\Delta_{1/2}}(0) &= \zeta_{\rm R}'(-1) +\frac12\log(2\pi) +\frac56,
\\
\zeta'_{\Delta_{2}}(0) &= \zeta_{\rm R}'(-1) +\frac12\log\pi +\frac5{24} -\frac5{12} \log 2 .
\end{align}
\end{lemma}
Observe that the propositions follow by simply inserting the calculations of Equations \eqref{Eq Spreafico derivative laplacian},\eqref{Eq Spreafico zeta zero revised} and Lemma  \ref{lem:dercalcul} to the expressions of Lemma \ref{lem:hkexpansion}.

We now proceed to prove the two lemmas.
\begin{proof}[Proof of Lemma \ref{lem:hkexpansion}]
We will only prove the first of the two claims, as the second follows exactly the same calculation. We have
$$\mu^{\rm loop}_{\D}
(t_{\gamma}>\delta, \indx_{\gamma}(0) \text{ odd}) = \mu^{\rm loop}_{\D}
(t_{\gamma}>\delta) - \frac{1}{2}\mu^{\rm loop}_{\D,g_{\rm cone, 1/2}}(t_{\gamma}>\delta)$$
which in terms of heat kernel gives
$$\int_\delta^\infty {\rm Tr}(\exp(-t\Delta/2))\frac{dt}t - \frac{1}{2}\int_\delta^\infty {\rm Tr}(\exp(-t\Delta_{1/2}/2))\frac{dt}t.$$

Now, on a conical disk with Dirichlet boundary conditions we have the following heat trace expansion as $t \to 0$:
\begin{equation}\label{eq:heatexp}
{\rm Tr}(\exp(-t\Delta_v)) = a_{-1,v}t^{-1}+a_{-1/2,v}t^{-1/2}+a_{0,v}+\sum_{k \geq 1}a_{k/2, v}t^{k/2}.
\end{equation}
This follows from direct calculations using the conical heat kernel \cite{Carslaw} . It is also stated in \cite[Theorem 4.9]{Gil} in a larger generality and follows from seminal works by \cite{Cheeger83, BruningSeeley}.  Also, the coefficients $a_{-1,v}, a_{-1/2,v}, a_{0,v}$ are also well-known and given in the Dirichlet disk by (see e.g. \cite[Section 6]{BordagKirstenDowker}):
\begin{equation}\label{eq:heatcoeff}
a_{-1,v} = \frac{|\D_{g_v}|}{4\pi} \quad ; \qquad \qquad 
a_{-1/2,v} = -\frac{|\partial \D_{g_v}|}{8\sqrt{\pi}} \quad ; \qquad \qquad
a_{0,v} = \frac{1}{12}(v+\frac1v).
\end{equation}
We now consider Mellin transform of the spectral zeta function
$$\zeta(s) = \Gamma(s)^{-1}\int_0^\infty t^{s-1}{\rm Tr}(\exp(-t\Delta_v))dt,$$
that is a priori well defined only for $\Re(s) > 1$. Recall that $\Gamma(s)^{-1}$ is entire with near-zero expansion $\Gamma(s)^{-1} = s + \gamma_{\rm EM}s^2 + O(s^3)$. From the heat trace expansion we see that the Mellin transform can be meromoprhically continued to $\C$ and is in fact analytic at $s = 0$ with $\zeta_{\Delta_v}(0) = a_{0,v}$.

Further, denoting
$$R(t) := {\rm Tr}(\exp(-t\Delta_v))-a_{0,v} - a_{-1,v}t^{-1}-a_{-1/2,v}t^{-1/2}$$
the heat kernel expansion also tells us that the derivative $\zeta_{\Delta_v}'(0)$ can be expressed as 
$$- a_{-1,v}-2a_{-1/2,v}+\gamma_{\rm EM}a_{0,v}+\int_0^1 R(t)\frac{dt}t+\int_1^\infty{\rm Tr}(\exp(-t\Delta_v))\frac{dt}t$$
This tells us that as $\delta \to 0$, we have the expansion
$$\int_\delta^{\infty}{\rm Tr}(\exp(-t\Delta_v))\frac{dt}t = \zeta_{\Delta_v}(0) (|\log \delta|- \gamma_{\rm EM}) + a_{-1,v}\delta^{-1} + 2a_{-1/2,v}\delta^{-1/2} + \zeta_{\Delta_v}'(0) +o(1).$$
Noticing now that the first two terms cancel when taking the difference, and using formulas $\zeta_{\Delta/2}(0) = \zeta_{\Delta/2}(0)$ and $\zeta_{\Delta_{v}/2}'(0) = \zeta_{\Delta_{v}/2}'(0) + (\log 2)\zeta_{\Delta_v}(0) $, we conclude.

Finally, the extra $2 \log 2$ term in the expression for the conical disk comes from the fact that the conical disk in the $z$-metric, has radius $r = 2$ in the parametrization used above. Thus one needs to also use the scaling of zeta function given (in the obvious notation) by $\zeta_{\Delta_{v,2}}(0) = \zeta_{\Delta_{v}}(0) $ and 
$\zeta'_{\Delta_{v,2}}(0) = \zeta'_{\Delta_{v}}(0) + (2\log 2)\zeta_{\Delta_{v,2}}(0)$
\end{proof}

\begin{proof}[Proof of Lemma \ref{lem:dercalcul}]
We use the notation
\[
\chi_v(s) = \sum_{n,k\ge1}(k+v n)^{-s}.
\]
Combining \cite[Lemma~2 ]{Spreafico05}, that gives an expression for $\chi_v'(0)$, and the paragraph below it, which gives an alternative expression for the derivative of the spectral zeta function, we obtain
\begin{align*}
\zeta'_{\Delta_v}(0)
= -\frac{\log 2}{6}(v+\frac1v)+v-1 + \frac{7}{12v} +2\chi_v'(0)-\frac12\log(2\pi)+2 +\frac2v\zeta_{\rm H}(-1,v+1).
\end{align*}
Here $\zeta_{\rm H}(s,a)$ is the Hurwitz zeta function again defined for $\Re(s)$ large by
$$\sum_{n \geq 0}(n+a)^{-s}$$
and extended meromorphically, like in the case of the Riemann zeta function.
Using
\[
\zeta_{\rm H}(-1,a) = -\frac12 B_2(a) = -\frac12(a^2-a+1/6),
\]
with $a=v+1$, gives
\[
\frac2v\zeta_{\rm H}(-1,v+1) = -v-1-\frac1{6v}.
\]
In particular
\begin{equation}
\label{Eq Phi via chi}
\zeta'_{\Delta_v}(0)
= -\frac{\log 2}6\left(v+\frac1v\right) + 2\chi_v'(0) + \frac5{12v} - \frac12\log(2\pi).
\end{equation}
It remains to compute $\chi_v'(0)$ for $v=2$ and $v=1/2$.
For $\Re s>2$, we can write
\[
\chi_2(s) = \sum_{n,k\ge1}(k+2n)^{-s} = \sum_{m=3}^\infty
\left\lfloor\frac{m-1}{2}\right\rfloor m^{-s}.
\]
Since
\[
\left\lfloor\frac{m-1}{2}\right\rfloor
=\frac{m-1}{2} - \frac12\1_{\{m\ {\rm even}\}}, \]
we get, first for $\Re s>2$ and then by meromorphic continuation,
\[
\chi_2(s) = \frac12\zeta_{\rm R}(s-1) - \left(\frac12+2^{-s-1}\right)\zeta_{\rm R}(s).
\]
Therefore, using
\[
\zeta_{\rm R}(0)=-\frac12,
\qquad
\zeta_{\rm R}'(0)=-\frac12\log(2\pi),
\]
we find
$$\chi_2'(0) = \frac12\zeta_{\rm R}'(-1) - \frac14\log2 + \frac12\log(2\pi).
$$
We conclude that 
$$\zeta'_{\Delta_2}(0) = \zeta_{\rm R}'(-1) +\frac12\log\pi +\frac5{24} -\frac5{12} \log 2$$
Similarly,
\[
\chi_{1/2}(s) = \sum_{n,k\ge1}\left(k+\frac n2\right)^{-s} = 2^s\sum_{n,k\ge1}(2k+n)^{-s}= 2^s\chi_2(s).
\]
Writing $\chi_2(0) = \frac12\zeta_{\rm R}(-1)-\zeta_{\rm R}(0) =
\frac{11}{24}$. 
we obtain
\[
\chi_{1/2}'(0) = \log2\,\chi_2(0)+\chi_2'(0) = \frac12\zeta_{\rm R}'(-1) +
\frac5{24}\log2 + \frac12\log(2\pi)
\]
and conclude that, as promised,
\[
\zeta'_{\Delta_{1/2}}(0) = \zeta_{\rm R}'(-1) +\frac12\log(2\pi) +\frac56.
\]
\end{proof}


\paragraph{Cut-off by removing a small disk}

Now we consider the mass of Brownian loops contained in 
$\D\setminus(\varepsilon \overline{\D})$,
and with an odd index around the point $0$:
\begin{equation}
\label{Eq mu loop disk eps disk}
\mu^{\rm loop}_{\D\setminus(\varepsilon \overline{\D})}
(\indx_{\gamma}(0) \text{ odd}).
\end{equation}
Again, we are interested in the asymptotic
expansion as $\varepsilon\to 0$. 
\begin{prop}
\label{Prop exp disk eps disk}
As $\varepsilon\to 0$,
$$
\mu^{\rm loop}_{\D\setminus(\varepsilon \overline{\D})}
(\indx_{\gamma}(0) \text{ odd})
=
\dfrac{1}{8}
\log(\varepsilon^{-1})
-
\dfrac{1}{2}
\log\log(\varepsilon^{-1})
+
\dfrac{1}{2}\log(\pi/2)
+ o(1).
$$
\end{prop}

This can be directly proved from an explicit expression,
derived in \cite[Proposition 2.18]{ALS3}
and \cite[Proposition 5.2]{LupuTopoAIHP}. Let us recall this.

For $\varepsilon\in (0,1)$, denote
$$
L(\varepsilon)= \dfrac{1}{2\pi} \log(\varepsilon^{-1}).
$$
This is the extremal length (or extremal distance)
between $\partial\D$ and
$\varepsilon\partial\D$
in the annulus $\D\setminus(\varepsilon \overline{\D})$.
Let $(\widehat{W}^{(\varepsilon)}_{t})_{0\leq t\leq L(\varepsilon)}$
be a one-dimensional standard Brownian bridge from $0$
to $0$ of time-length $L(\varepsilon)$. 
The following has been proved.

\begin{prop}[Aru-Lupu-Sep\'ulveda \cite{ALS3},
Lupu \cite{LupuTopoAIHP}]
For every $\varepsilon\in (0,1)$,
$$
\mu^{\rm loop}_{\D\setminus(\varepsilon \overline{\D})}
(\indx_{\gamma}(0) \text{ odd})
=
-
\log
\mathbb{P}\big((\widehat{W}^{(\varepsilon)}_{t})_{0\leq t\leq L(\varepsilon)}
\text{ stays in } (-\sqrt{\pi/2},\sqrt{\pi/2})\big).
$$
\end{prop}

Further, denote by
$p_{\R}(t,x,y)$ the heat kernel on $\R$,
and by
$p_{(-\sqrt{\pi/2},\sqrt{\pi/2})}(t,x,y)$
the heat kernel in the interval
$(-\sqrt{\pi/2},\sqrt{\pi/2})$
with $0$ boundary conditions.
We have that
\begin{align*}
\mathbb{P}\big((\widehat{W}^{(\varepsilon)}_{t})_{0\leq t\leq L(\varepsilon)}
\text{ stays in } (-\sqrt{\pi/2},\sqrt{\pi/2})\big)
&=
\dfrac{p_{(-\sqrt{\pi/2},\sqrt{\pi/2})}(L(\varepsilon),0,0)}
{p_{\R}(L(\varepsilon),0,0)}
\\
&= 
\sqrt{2\pi L(\varepsilon)}
p_{(-\sqrt{\pi/2},\sqrt{\pi/2})}(L(\varepsilon),0,0).
\end{align*}
Moreover,
$$
p_{(-\sqrt{\pi/2},\sqrt{\pi/2})}(L(\varepsilon),0,0)
=
\sqrt{\dfrac{2}{\pi}}
\sum_{n\geq 0}
e^{-\frac{\pi}{4}(2n+1)^{2} L(\varepsilon)} .
$$
This is the Fourier decomposition of the heat kernel.
See \cite[Appendix A]{ALS3}.
So we get that
$$
-
\log\mathbb{P}\big((\widehat{W}^{(\varepsilon)}_{t})_{0\leq t\leq L(\varepsilon)}
\text{ stays in } (-\sqrt{\pi/2},\sqrt{\pi/2})\big)
=
\dfrac{\pi}{4}L(\varepsilon)
-
\dfrac{1}{2}\log(2\pi L(\varepsilon))
+
\dfrac{1}{2}\log(\pi/2)
+ o(1)
$$
and obtain the asymptotic expansion for
\eqref{Eq mu loop disk eps disk}.

\subsubsection{The case of a general simply-connected domain with boundary and finitely many punctures}\label{subsec:cutoffmanypoints}

We now treat the case of finitely many punctures and prove Theorem \ref{thm:deftwist}. Let $D\subset \C$ be an open non-empty connected and simply connected domain, with $D\neq\C$.
Let $n\geq 1$ and
$z_{1}, z_{2},\dots, z_{n}\in D$
be finitely many punctures, two by two disjoint.
We consider the masses of Brownian loops
$$
\mu^{\rm loop}_{D}
\Big(\sum_{j=1}^{n}
\indx_{\gamma}(z_j)
\text{ odd},
t_{\gamma}>\delta
\Big)
$$
and
$$
\mu^{\rm loop}_{D}
\Big(\sum_{j=1}^{n}
\indx_{\gamma}(z_j)
\text{ odd},
\gamma\cap\Big(\bigcup_{j=1}^{n}\D(z_j,\varepsilon_j)\Big)
=\emptyset
\Big),
$$
and we are interested in asymptotic expansions
up to and including the constant order,
as $\delta$, respectively
$\max_{1\leq j\leq n}\varepsilon_j$ tends to $0$.

Fix $r_0\in (0,1]$ small enough such that
for every $j\in\{1,\dots, n\}$,
$$
\D(z_j,r_0)
\subset
D
\setminus
\Big(
\bigcup_{\substack{1\leq i\leq n\\ i\neq j}}
\D(z_i,r_0)
\Big)
.
$$
Define
$$
\widehat{\Ftw}^{D, r_0}_{n}(z_{1}, z_{2},\dots, z_{n})
=
\mu^{\rm loop}_{D}
\Big(\sum_{j=1}^{n}
\indx_{\gamma}(z_j)
\text{ odd},
\gamma\not\subset
\bigcup_{j=1}^{n}
\D(z_j,r_0)
\Big).
$$

\begin{lemma}
\label{Lem tilde F finit}
The quantity 
$\widehat{\Ftw}^{D, r_0}_{n}(z_{1}, z_{2},\dots, z_{n})$
is finite.
This includes the case when $D$ is unbounded.
\end{lemma}
\begin{proof}
First consider the case of $D$ bounded.
In this case there is no
infrared divergence (large loops) for the integral
$$\widehat{\Ftw}^{D, r_0}_{n}(z_{1}, z_{2},\dots, z_{n}) =  \int_D \int_0^\infty\P_{\C,t}^{z_0,z_0}(\sum_{j=1}^{n}
\indx_{\gamma}(z_j)
\text{ odd},
\gamma\not\subset
\bigcup_{j=1}^{n}
\D(z_j,r_0)) \frac{dt}t^2 dz_0.$$
So let us also exclude the ultraviolet divergence
(small loops). Consider a loop $\gamma$ such that
$$
\sum_{j=1}^{n}
\indx_{\gamma}(z_j)
\text{ odd},
\gamma\not\subset
\bigcup_{j=1}^{n}
\D(z_j,r_0).
$$
We will distinguish two cases according to the position of the
root $z_0 := \gamma(0)$.
In this first case, 
$$
\gamma(0)\in D\setminus\Big(\bigcup_{j=1}^{n}
\overline{\D(z_j,r_0/2)}\Big).
$$
Then in order to surround one of the $z_j$-s,
the loop $\gamma$ has to travel a distance
at least $r_0/2$ from $\gamma(0)$.
In the second case,
$$
\gamma(0)\in \bigcup_{j=1}^{n}
\D(z_j,r_0/2).
$$
Since $\gamma$ also intersects
$$
D\setminus\Big(\bigcup_{j=1}^{n}
\overline{\D(z_j,r_0)}\Big),
$$
it again has to travel a distance at least
$r_0/2$ from $\gamma(0)$.

Thus in both cases we obtain for all $t$ small enough an upper bound of the form
$e^{-c/t}$ for
$$\P_{\C,t}^{z_0,z_0}(\sum_{j=1}^{n}
\indx_{\gamma}(z_j)
\text{ odd},
\gamma\not\subset
\bigcup_{j=1}^{n}
\D(z_j,r_0)),$$
ensuring the convergence of the integral 
$$\widehat{\Ftw}^{D, r_0}_{n}(z_{1}, z_{2},\dots, z_{n}) =  \int_D \int_0^\infty\P_{\C,t}^{z_0,z_0}(\sum_{j=1}^{n}
\indx_{\gamma}(z_j)
\text{ odd},
\gamma\not\subset
\bigcup_{j=1}^{n}
\D(z_j,r_0)) \frac{dt}t^2 dz_0$$ also for small $t$.

Now consider that $D$ is unbounded.
Compared to the previous case,
we have to rule out the infrared divergence related to
$D$ being unbounded.
Set
$$
R=
\max((\vert z_j\vert)_{1\leq j\leq n},
d(0,\partial D)).
$$
Since $D$ is simply connected,
the complementary
$\C\setminus D$
connects $R\partial \D$ to $\infty$.
Moreover,
$\C\setminus(R\overline{\D})$
does not contain any of the $z_j$.
Therefore, topologically, a loop
$\gamma$ cannot surround any of the $z_j$
by staying in 
$D\cap(\C\setminus(R\overline{\D}))$.
It follows that
\begin{align*}
\mu^{\rm loop}_{D}
\Big(
\gamma(0)\in \C\setminus(2R\overline{\D}),
\sum_{j=1}^{n}
\indx_{\gamma}(z_j)
\text{ odd}
\Big)
&\leq
\mu^{\rm loop}_{D}
\Big(
\gamma(0)\in \C\setminus(2R\overline{\D}),
\gamma\cap(R\overline{\D})\neq \emptyset.
\Big)
\\
&=
\int_{\C\setminus(2R\overline{\D})}
\int_{0}^{\infty}
\P_{\C,t}^{z,z}
(\gamma\cap(R\overline{\D})\neq\emptyset)
\dfrac{dt}{2\pi t^2} d^2 z\,.
\end{align*}
The above integral is finite since one can bound
$$
\P_{\C,t}^{z,z}
(\gamma\cap(R\overline{\D})\neq\emptyset)
\leq
C
\exp\big(-c (\vert z\vert -R)^2/t \big),
$$
for some constants $C,c>0$.
\end{proof}

Now set 
\begin{equation}
\label{Eq F D n z_j}
\Ftw^{D}_{n}(z_{1}, z_{2},\dots, z_{n})
=
\widehat{\Ftw}^{D, r_0}_{n}(z_{1}, z_{2},\dots, z_{n})
-
n 
\mu^{\rm loop}_{\D}(
\gamma\cap(\D\setminus r_0\overline{\D})
\neq\emptyset,
\indx_{\gamma}(0) \text{ odd}).
\end{equation}
Then the twist field correlation functions will be given by 
$$\sigma_{\rm tw}^D(z_1, z_2, \dots, z_n) 
:= \exp(-\Ftw^{D}_{n}(z_{1}, z_{2},\dots, z_{n})),$$
but it is a bit easier in the calculations to work directly on the level of the $\Ftw^{D}_{n}$. 
Observe that by definition we have normalised such that $\Ftw_1^\D(0) = 0$, 
and thus $\sigma_{\rm tw}^\D(0) = 1$.
A first important check is that 
\begin{lemma}
$\Ftw^{D}_{n}(z_{1}, z_{2},\dots, z_{n})$
does not depend on $r_0$.
\end{lemma}

\begin{proof}

Consider such an $r_0$ and take
$r_1\in (0,r_0)$.
We want to show that
\begin{multline*}
\widehat{\Ftw}^{D, r_1}_{n}(z_{1}, z_{2},\dots, z_{n})
-
\widehat{\Ftw}^{D, r_0}_{n}(z_{1}, z_{2},\dots, z_{n})
=
\\
n(
\mu^{\rm loop}_{\D}(
\gamma\cap(\D\setminus r_1\overline{\D})
\neq\emptyset,
\indx_{\gamma}(0) \text{ odd})
-
\mu^{\rm loop}_{\D}(
\gamma\cap(\D\setminus r_0\overline{\D})
\neq\emptyset,
\indx_{\gamma}(0) \text{ odd})
).
\end{multline*}
By definition,
\begin{multline*}
\widehat{\Ftw}^{D, r_1}_{n}(z_{1}, z_{2},\dots, z_{n})
-
\widehat{\Ftw}^{D, r_0}_{n}(z_{1}, z_{2},\dots, z_{n})
\\=
\mu^{\rm loop}_{D}
\Big(\sum_{j=1}^{n}
\indx_{\gamma}(z_j)
\text{ odd},
\gamma\subset 
\bigcup_{j=1}^{n}
\D(z_j,r_0)
,
\gamma\not\subset
\bigcup_{j=1}^{n}
\D(z_j,r_1)
\Big).
\end{multline*}
Since the disks $\D(z_j,r_0)$
are two by two disjoint,
the latter quantity is also
$$
\sum_{j=1}^{n}
\mu^{\rm loop}_{\D(z_j,r_0)}
\Big(
\indx_{\gamma}(z_j)
\text{ odd},
\gamma\not\subset
\D(z_j,r_1)
\Big)
=
n
\mu^{\rm loop}_{r_0\D}(
\gamma\cap(\D\setminus r_1\overline{\D})
\neq\emptyset,
\indx_{\gamma}(0) \text{ odd}).
$$
Since $r_0\leq 1$, this is also
\begin{displaymath}
n(
\mu^{\rm loop}_{\D}(
\gamma\cap(\D\setminus r_1\overline{\D})
\neq\emptyset,
\indx_{\gamma}(0) \text{ odd})
-
\mu^{\rm loop}_{\D}(
\gamma\cap(\D\setminus r_0\overline{\D})
\neq\emptyset,
\indx_{\gamma}(0) \text{ odd})
).
\qedhere
\end{displaymath}
\end{proof}
\noindent To prove Theorem \ref{thm:deftwist}, it now suffices to exhibit the following two asymptotic expansions.
\begin{itemize}
\item As $\delta\to 0$,
$$
\mu^{\rm loop}_{D}
\Big(\sum_{j=1}^{n}
\indx_{\gamma}(z_j)
\text{ odd},
t_{\gamma}>\delta
\Big)
=
\dfrac{n}{16}\log(\delta^{-1})
+
n \cbls + 
\Ftw^{D}_{n}(z_{1}, z_{2},\dots, z_{n})
+o(1),
$$
where  the constant $\cbls$ is given by
\eqref{Eq C D 0 diff}.
\item As $\max_{1\leq j\leq n}\varepsilon_j\to 0$,
\begin{multline}\label{eq:loopcutdisk}
\mu^{\rm loop}_{D}
\Big(\sum_{j=1}^{n}
\indx_{\gamma}(z_j)
\text{ odd},
\gamma\cap\Big(\bigcup_{j=1}^{n}\D(z_j,\varepsilon_j)\Big)
=\emptyset
\Big)
=
\\
\dfrac{1}{8}
\sum_{j=1}^{n}
\log(\varepsilon_j^{-1})
-
\dfrac{1}{2}
\sum_{j=1}^{n}
\log\log(\varepsilon_j^{-1})
+
\dfrac{n}{2}\log(\pi/2)
+ 
\Ftw^{D}_{n}(z_{1}, z_{2},\dots, z_{n})
+o(1).
\end{multline}
\end{itemize}
To prove these, we decompose
\begin{align*}
\mu^{\rm loop}_{D}
\Big(\sum_{j=1}^{n}
\indx_{\gamma}(z_j)
\text{ odd},
t_{\gamma}>\delta
\Big) =& 
\mu^{\rm loop}_{D}
\Big(\sum_{j=1}^{n}
\indx_{\gamma}(z_j)
\text{ odd},
\gamma\not\subset
\bigcup_{j=1}^{n}
\D(z_j,r_0),
t_{\gamma}>\delta
\Big)
\\
&+
n
\mu^{\rm loop}_{r_0\D}
(t_{\gamma}>\delta, 
\indx_{\gamma}(0)
\text{ odd}).
\end{align*}
Similarly,
\begin{multline*}
\mu^{\rm loop}_{D}
\Big(\sum_{j=1}^{n}
\indx_{\gamma}(z_j)
\text{ odd},
\gamma\cap\Big(\bigcup_{j=1}^{n}\D(z_j,\varepsilon_j)\Big)
=\emptyset
\Big)
=
\\
\mu^{\rm loop}_{D}
\Big(\sum_{j=1}^{n}
\indx_{\gamma}(z_j)
\text{ odd},
\gamma\not\subset
\bigcup_{j=1}^{n}
\D(z_j,r_0),
\gamma\cap\Big(\bigcup_{j=1}^{n}\D(z_j,\varepsilon_j)\Big)
=\emptyset
\Big)
+
\sum_{j=1}^{n}
\mu^{\rm loop}_{r_0\D\setminus(\varepsilon_j\overline{\D})}
(\indx_{\gamma}(0)
\text{ odd}).
\end{multline*}
Further,
\begin{multline*}
\lim_{\delta\to 0}
\mu^{\rm loop}_{D}
\Big(\sum_{j=1}^{n}
\indx_{\gamma}(z_j)
\text{ odd},
\gamma\not\subset
\bigcup_{j=1}^{n}
\D(z_j,r_0),
t_{\gamma}>\delta
\Big)=
\\
\lim_{\max((\varepsilon_j)_{1\leq j\leq n}) \to 0}
\mu^{\rm loop}_{D}
\Big(\sum_{j=1}^{n}
\indx_{\gamma}(z_j)
\text{ odd},
\gamma\not\subset
\bigcup_{j=1}^{n}
\D(z_j,r_0),
\gamma\cap\Big(\bigcup_{j=1}^{n}\D(z_j,\varepsilon_j)\Big)
=\emptyset
\Big)
=
\\
\widehat{\Ftw}^{D, r_0}_{n}(z_{1}, z_{2},\dots, z_{n}).
\end{multline*}
Then,
\begin{multline*}
\mu^{\rm loop}_{r_0\D}
(t_{\gamma}>\delta, 
\indx_{\gamma}(0)
\text{ odd})
=
\\
\mu^{\rm loop}_{\D}
(t_{\gamma}>\delta, 
\indx_{\gamma}(0)
\text{ odd})
-
\mu^{\rm loop}_{\D}(
\gamma\cap(\D\setminus r_0\overline{\D})
\neq\emptyset,
\indx_{\gamma}(0) \text{ odd},t_{\gamma}>\delta),
\end{multline*}
and
$$
\mu^{\rm loop}_{r_0\D\setminus(\varepsilon_j\overline{\D})}
(\indx_{\gamma}(0)
\text{ odd})
=
\mu^{\rm loop}_{\D\setminus(\varepsilon_j\overline{\D})}
(\indx_{\gamma}(0)
\text{ odd})
-
\mu^{\rm loop}_{\D\setminus(\varepsilon_j\overline{\D})}
(\gamma\cap(\D\setminus r_0\overline{\D})
\neq\emptyset,
\indx_{\gamma}(0) \text{ odd}).
$$
We have
\begin{align*}
\lim_{\delta\to 0}
\mu^{\rm loop}_{\D}(
\gamma\cap(\D\setminus r_0\overline{\D})
\neq\emptyset,
\indx_{\gamma}(0) \text{ odd},t_{\gamma}>\delta)
&=
\lim_{\eps_j\to 0}
\mu^{\rm loop}_{\D\setminus(\varepsilon_j\overline{\D})}
(\gamma\cap(\D\setminus r_0\overline{\D})
\neq\emptyset,
\indx_{\gamma}(0) \text{ odd})
\\
&=
\mu^{\rm loop}_{\D}(
\gamma\cap(\D\setminus r_0\overline{\D})
\neq\emptyset,
\indx_{\gamma}(0) \text{ odd}).
\end{align*}
Finally, we apply Proposition \ref{Prop disk 1pt time} to
$
\mu^{\rm loop}_{\D}
(t_{\gamma}>\delta, 
\indx_{\gamma}(0)
\text{ odd}),
$
and Proposition \ref{Prop exp disk eps disk} to
$
\mu^{\rm loop}_{\D\setminus(\varepsilon_j\overline{\D})}
(\indx_{\gamma}(0)
\text{ odd}).
$

\subsubsection{Diffeomorphism invariance, conformal covariance and Weyl scaling}

\paragraph{Diffeomorphism invariance}
Diffeomorphism invariance just stems from the fact that under a diffeomorphism $\psi : D_1 \to D_2$ the Brownian motion in $(D_1, g_1)$ 
(i.e. associated to the Laplace-Beltrami operator of the metric $g_1$)
maps to the Brownian motion to $(D_2, g_2)$ with $g_1 = \psi^*g_2$. \\

The Weyl scaling
and conformal covariance both follow from the following immediate observation that we separate for clarity. Here we denote by $\D_g(z, \eps)$ 
the metric ball for the metric $g$.

As in the Euclidean case, we have an expansion
as $\max_{1\leq j\leq n}\eps_j\to 0$:
\begin{multline*}
\mu^{\rm loop}_{D,g}
\Big(\sum_{j=1}^{n}
\indx_{\gamma}(z_j)
\text{ odd},
\gamma\cap\Big(\bigcup_{j=1}^{n}\D_g(z_j,\varepsilon_j)\Big)
=\emptyset
\Big)
=
\\
\sum_{j=1}^n\dfrac{1}{8}
\log(\varepsilon_j^{-1})
-
\sum_{j=1}^n\dfrac{1}{2}
\log\log(\varepsilon_j^{-1})
+
\dfrac{n}{2}\log(\pi/2)
+ \Ftw^{D,g}_{n}(z_{1}, z_{2},\dots, z_{n})
+ o(1)
.
\end{multline*}
The following is clear.
\begin{claim}\label{claim:weyl}
Let $g_1, g_2$ be two metrics on $D$ related by $g_2 = \Omega^2 g_1$.
Suppose that for all $j \in \{1 \dots n\}$ there are functions $0 < r_j^-(\eps) \leq r_j^+(\eps)< \infty$ with $r_j^-(\eps) = r_j^+(\eps)(1+o(1))$ and $\D_{g_1}(z_j, r_j^-(\eps)) \subseteq \D_{g_2}(z_j, \eps) \subseteq \D_{g_1}(z_j, r_j^+(\eps))$. Then 
\begin{multline*}
\mu^{\rm loop}_{D, g_1}
\Big(\sum_{j=1}^{n}
\indx_{\gamma}(z_j)
\text{ odd},
\gamma\cap\Big(\bigcup_{j=1}^{n}\D_{g_1}(z_j,r_{j}^{+}(\varepsilon_j))\Big)
=\emptyset
\Big)
\\
\leq\mu^{\rm loop}_{D, g_2}
\Big(\sum_{j=1}^{n}
\indx_{\gamma}(z_j)
\text{ odd},
\gamma\cap\Big(\bigcup_{j=1}^{n}\D_{g_2}(z_j,\varepsilon_j)\Big)
=\emptyset
\Big)
\\
\leq\mu^{\rm loop}_{D, g_1}
\Big(\sum_{j=1}^{n}
\indx_{\gamma}(z_j)
\text{ odd},
\gamma\cap\Big(\bigcup_{j=1}^{n}\D_{g_1}(z_j,r_{j}^{-}(\varepsilon_j))\Big)
=\emptyset
\Big),
\end{multline*}
and as $\max_{1\leq j\leq  n}\eps_j \to 0$,
$$\Ftw^{D,g_2}_{n}(z_{1}, z_{2},\dots, z_{n}) 
- \Ftw^{D,g_1}_{n}(z_{1}, z_{2},\dots, z_{n}) = \sum_{j=1}^n\dfrac{1}{8}
\log(\varepsilon_j^{-1}r_j^+(\eps_j))
+
\sum_{j=1}^n\dfrac{1}{2}
\log\dfrac{\log(\varepsilon_j^{-1})}{\log (r_j^+(\eps_j)^{-1})}+o(1).
$$
\end{claim}

\paragraph{Weyl scaling}
Let $\Omega:D\to [0,\infty)$ be continuous and consider $g_2 = \Omega^2 g_1$. Then by continuity of $\Omega$ we can find for each $j$ the functions $r_j^-(\eps)$ and $r_j^+(\eps)$ with $r_j^-(\eps)/\eps \to \Omega(z_j)$ satisfying the conditions of the claim. 
Thus, Weyl scaling follows from the claim. 

\paragraph{Conformal covariance}
To prove conformal covariance we can just combine diffeomorphism invariance and 
Weyl scaling by considering $D_2$ with the Euclidean metric $g_1$ and with the metric $g_2 = g_1 \Omega^2$,
where $\Omega = |\psi'(z)|$ and $\psi: D_1 \to D_2$ is the relevant conformal map.

\subsection{The explicit calculation of two-point correlations}

In this subsection we will prove Theorem \ref{thm:twopointtwist}.
This will be a calculation based on the first definition of the twist fields, i.e. the cut-off in time. To be able to push it through, we first have to reduce the calculation to a suitable parametrization - this will be the strip with its Euclidean metric - and then we proceed to computations.

Thus, in the first subsection we give the reduction and the proof of Theorem \ref{thm:twopointtwist} modulo an annulus calculation and then in the second subsection we give the calculation in the annulus.

\subsubsection{Reduction to a calculation in the annulus and proof of Theorem \ref{thm:twopointtwist}}\label{sec:reduction}
We can use conformal covariance 
to map the domain $(D, z_1, z_2)$ to $(\D, -r, r)$. 
In this set-up, 
$\CR(r, \D) = \CR(-r,\D) = 1-r^2$.
It then remains to prove that 
\begin{equation}\label{eq:twisteq}
\sigma_{\rm tw}^\D(-r,r) = \Ctw\frac{1}{(1-r^2)^{1/4}\theta_2(q^{1/4})\sqrt{|\log q|}},
\end{equation}
where as above the nome $q$ is uniquely fixed by Lemma \ref{lem:parametr}, 
via $\theta^2_3(q^{1/2})/\theta^2_2(q^{1/2}) 
= \frac{1+r^2}{2r}$.

As mentioned, we will do the calculation in an annulus 
$\ann_R = \{ w\in\C \vert \, R^{-1}<\vert w\vert < R\}$
endowed with the cylindrical metric
$\vert w\vert^{-2}\, d^2 w$ - the latter is chosen so that its lift to the strip
via the universal covering is the Euclidean metric. 
This annulus will provide a double cover for the disk $(\D, -r,r)$ with branch-points at $\pm r$ mapping to $\pm 1$ and the line segment between $\pm r$ mapping to the circle of radius $1$. 
To describe this double cover, 
let us introduce some notations.
\begin{fleqn}
    \begin{align*}
&\text{Complete elliptic integral in Jacobi notation: } \\
&\;\;(\star)\;\;K(k) = \int_{0}^{1}
        \frac{dx}{\sqrt{(1-k^{2}x^2)(1-x^2)}},\;K'(k):=K(\sqrt{1-k^{2}}).\\
&\text{Incomplete elliptic integral in Jacobi notation for: }z\in\HH,\\ &\;\;(\star)\;\;F(z,k)
  := \int_{0}^{z}\frac{dx}{\sqrt{(1-k^{2}x^2)(1-x^2)}}, \\
  &\text{ where we integrate along any path } \gamma \text{ from }0\to z\text{ remaining in }\HH \text{ other than at its endpoints.} 
\end{align*}
\end{fleqn}
\begin{lemma}[Branched double cover]\label{lemma:cover}
The mapping composed of first mapping the disk to the ellipse via the Schwarz map
$$ \psi(z)=\sin\Big(\frac{\pi}{2K(r^2)}F(z/r, r^2)\Big),$$
and then following with the Joukowsky map
$$\Juk(z) = z + \sqrt{z^2-1},$$
gives a branched double cover of the unit disk $\D$ by the annulus $\ann_R$ with branch-points at $\pm r$. 
The radius $R$ is given by $\log R = \pi K'(r^2)/4K(r^2)$.
The corresponding covering automorphism is $w \to w^{-1}$.

Further, if $\ann_R$ is endowed with the cylindrical metric $g_{\rm cyl}(w) := \vert w\vert^{-2}\, |dw|^2$, then its pull-back behaves around $\pm r$ like $\hat g = \hat \Omega^2 |dz|^2$, with 
$$\hat \Omega^2(\pm r + \eps \omega) =  
\eps^{-1}\fconf(r)+O(\eps^{-1/2}),
\qquad
\fconf(r) = \frac{\pi^2}{8K(r^2)^2 r(1-r^4)},
$$ 
for $|\omega| = 1$, as $\eps \to 0$.
\end{lemma}

\begin{proof}
The conformal map from the unit disk to the ellipse mapping $\pm r$ to the foci $\pm 1$ is given by \cite{Schwarz, Szego} 
$$ \psi(z)=\sin\Big(\frac{\pi}{2K(r^2)}F(z/r, r^2)\Big).$$
It comes from the following steps:
\begin{itemize}
\item we scale $z \to z/r$;
\item we use the Schwarz-Christoffel map $z \to F(z, r^2)$ to map the upper half of the disk of radius $1/r$ to the rectangle with corners $\pm K(r^2)$ and $\pm K(r^2)+iK'(r^2)/2$. The corners are images of $\pm 1$ and $\pm r^{-1}$;
\item we scale $z \to z\pi/2K(r^2)$ to rectangle with vertices $\pm \pi/2$ and $\pm \pi/2 + iM$ with $M = \dfrac{\pi K'(r^2)}{4K(r^2)}$;
\item we consider $z \to \sin(z)$ to map to the upper half of the ellipse and then use Schwarz reflection to obtain the full map to ellipse.
\end{itemize}
From here we see that the axes of the ellipse are given by $$a = \cosh(M), b = \sinh(M),$$
where $M=\dfrac{\pi K'(r^2)}{4K(r^2)}$.

The branched double cover from the ellipse with branch points at the foci $\pm 1$ mapping to an annulus with radii $R^{-1}, R$, and sending $\pm 1 \to \pm 1$ is the (inverse) Joukowsky map 
\cite{Joukowsky}:
$$\Juk(z) = z + \sqrt{z^2-1}.$$
By comparing the conformal moduli we obtain the relation 
$\log R = M = \frac{\pi K'(r^2)}{4K(r^2)}$. 

Finally, it remains to determine the behaviour of the pull-back metric at $\pm r$, when starting from the cylindrical metric on the annulus. By symmetry it clearly suffices to only consider the case of $r$. The conformal factor is given by 
$$\hat \Omega^2(r+\eps \omega) = 
\frac{|\Juk'(\psi(r + \eps \omega))
\psi'(r + \eps \omega)|^2}
{\vert\Juk(\psi(r+\eps \omega))\vert^2}.$$

First, expanding carefully each step gives that
$$\psi(r+\eps \omega) = 
1 + \psi'(r)\eps \omega + O(\eps^2),$$
where 
$$\psi'(r) = \frac{\pi^2}{4K(r^2)^2r(1-r^4)}.$$
Here crucially step 4 balances the blow-up of the derivative from step 2 in the above decomposition of the map.

Second, one can calculate that
$$\Juk(1 +z) = 1 + \sqrt{2z} + O(z),$$
and
$$\Juk'(1+z) = 1/\sqrt{2z} + O(1).$$
Putting everything together
we get that
$$|\Juk(\psi(r+\eps \omega))|^2 = 1 + O(\eps^{1/2}),$$
and
$$|\Juk'(\psi(r+\eps \omega))\psi'(r+\eps \omega)|^2 = \frac{\pi^2}{\eps 8K(r^2)^2 r(1-r^4)}+O(\eps^{-1/2}).$$
Thus,
\begin{equation}\label{eq:scaling}
\hat \Omega^2(\pm r +\eps \omega) = \frac{\pi^2}{\eps 8K(r^2)^2 r(1-r^4)}+O(\eps^{-1/2}),
\qedhere
\end{equation}
proving the lemma.
\end{proof}

This allows us to conclude Theorem \ref{thm:twopointtwist} modulo an explicit calculation in the next subsection.

\paragraph{Proof of Theorem \ref{thm:twopointtwist}.}
From the previous lemma and diffeomorphism invariance of the Brownian loop soup, we deduce that to calculate
\begin{equation}
\label{eq:annulus1}    
\lim_{\delta \to 0}\mu^{\rm loop}_{\D, \hat g}
\Big(
\indx_{\gamma}(r)+\indx_{\gamma}(-r)
\text{ odd},
t_{\gamma}>\delta
\Big)$$
it suffices to calculate 
$$
\lim_{\delta \to 0}
\int_{\ann_R^+}
\int_{\delta}^{+\infty}
p_{\ann_R}^{\rm cyl}(t,w,w^{-1})
\dfrac{dt}{t}
\,\dfrac{d^2 w}{\vert w\vert^2},\end{equation}
where by 
$$
\ann_R^+
=
\{ w\in\C \vert \, 1<\vert w\vert < R\}$$
we denote a fundamental domain of the annulus and by $\hat g$ the pull-back metric of the cylindrical metric, under the map described in the lemma.
By $p_{\ann_R}^{\rm cyl}$ we denote the heat kernel on the annulus
$\ann_R$ w.r.t. the cylindrical metric $g_{\rm cyl}$, 
and with $0$ boundary conditions.

We would like to swap the loop measure under metric $\hat g$ for the standard loop measure. Notice that for the cut-off using disks this is a rather simple corollary of Claim \ref{claim:weyl}: using that claim and then Equation \eqref{eq:loopcutdisk}, we can write
\begin{align}\label{eq:toannulus0}
&\lim_{\eps \to 0}\mu^{\rm loop}_{\D, \hat g}\Big(
    \indx_{\gamma}(r)+\indx_{\gamma}(-r)\ \text{odd},\ 
    \gamma \cap \bigl(\D_{\hat g}(r,\eps)\cup \D_{\hat g}(-r,\eps)\bigr)=\emptyset
\Big) \nonumber\\
&\qquad
=\lim_{\eps \to 0}\mu^{\rm loop}_{\D}\Big(
    \indx_{\gamma}(r)+\indx_{\gamma}(-r)\ \text{odd},\ 
    \gamma \cap 
    \Big(\D\big(r, \frac{\eps^2}{4\fconf(r)}\big)
    \cup 
    \D\big(-r, \frac{\eps^2}{4\fconf(r)}\big)\Big)
    =\emptyset
\Big) \nonumber\\
\\
&\qquad
=
-\log \sigma_{\rm tw}^\D(-r,r)
+ \frac{1}{2}\log (\eps^{-1}) 
- \log \log (\eps^{-1})
+\frac{1}{4} \log (\fconf(r))
-\frac{\log 2}{2} + \log(\pi/2) + o(1).
\end{align}
However, the simple calculations that we perform in the next subsection work with the cut-off in time. So we will in fact also need to connect the spatial cut-off for the above loop measure to the cut-off in time. 

To do this, we observe that the localisation argument of Section \ref{subsec:cutoffmanypoints}, i.e. the passage from one point to many, 
works also in the case of this modified metric.
Denote by $g_{\rm cone}$ the metric
$\vert z \vert^{-1}|dz|^2$ (conic singularity at $0$),
which, up to a constant, describes the local behaviour of
$\hat g$ near $\pm r$.
By localisation, we obtain:
\begin{align}
\label{eq:toannulus2}
&\lim_{\eps \to 0}\Big(
    \mu^{\rm loop}_{\D, \hat g}\big(
    \indx_{\gamma}(r)+\indx_{\gamma}(-r)\ \text{odd},\ 
    \gamma \cap \bigl(\D_{\hat g}(r,\eps)\cup \D_{\hat g}(-r,\eps)\bigr)=\emptyset
\big)\nonumber\\
&- \mu^{\rm loop}_{\D, \hat g}\big(
    \indx_{\gamma}(r)+\indx_{\gamma}(-r)\ \text{odd},t_\gamma > \eps^2\big)
    \Big)
    \nonumber\\
    &= 2\lim_{\eps \to 0}
    \Big(
    \mu^{\rm loop}_{\D\setminus
    \overline{\D_{g_{\rm cone}}(0,\eps)}}
(\indx_{\gamma}(0) \text{ odd}) 
- \mu^{\rm loop}_{\D, g_{\rm cone}}
(t_{\gamma}>\eps^2, \indx_{\gamma}(0) \text{ odd})\Big).
\end{align}

Further Claim \ref{claim:weyl} gives us directly the 1-point function for $g_{\rm cone}$ for cut-off in space: 
\begin{equation}
\label{eq:toannulus1pt}
\mu^{\rm loop}
_{\D\setminus\overline{\D_{g_{\rm cone}}(0,\eps)}}
(\indx_{\gamma}(0) \text{ odd})
=
\dfrac{1}{4}
\log(\varepsilon^{-1})
-
\dfrac{1}{2}
\log\log(\varepsilon^{-1})
+\dfrac{\log 2}{4}+
\dfrac{1}{2}\log(\pi/2)
+ o(1),
\end{equation}
where the $\eps$-disk is naturally also w.r.t. the 
$g_{\rm cone}$ metric. 
Further, the cut-off in time for the cone metric is given by Proposition \ref{Prop disk 1pt time conic}. 

Thus, by combining Equations \eqref{eq:annulus1}, \eqref{eq:toannulus0}, \eqref{eq:toannulus2}, \eqref{eq:toannulus1pt} and \eqref{eq:timecutoffconic}, 
we now obtain
$$\log \sigma_{\rm tw}^{\D}(-r,r) = 
-\log 2 +
2\cblsg
+\frac{1}{4}\log\fconf(r) +
\lim_{\eps \to 0}\left(\frac{1}{2}\log \eps^{-1}
-\int_{\ann_R^+}
\int_{\eps^2}^{+\infty}
p_{\ann_R}^{\rm cyl}(t,w,w^{-1})
\dfrac{dt}{t}
\,\dfrac{d^2 w}{\vert w\vert^2}\right),$$
where
$$
\int_{\ann_R^+}
\int_{\eps^2}^{+\infty}
p_{\ann_R}^{\rm cyl}(t,w,w^{-1})
\dfrac{dt}{t}
\,\dfrac{d^2 w}{\vert w\vert^2}
=
\mu^{\rm loop}_{\D, \hat g}\big(
    \indx_{\gamma}(r)+\indx_{\gamma}(-r)\ \text{odd},t_\gamma > \eps^2\big).
$$
Theorem \ref{Thm 2 pts disk R} provides us with the following expression as $\delta \to 0$:
\begin{equation}\label{eq:annuluscalc}
\int_{\ann_R^+}
\int_{\delta}^{+\infty}
p_{\ann_R}^{\rm cyl}(t,w,w^{-1})
\dfrac{dt}{t}
\,\dfrac{d^2 w}{\vert w\vert^2}
=
\dfrac{1}{4}\log(\delta^{-1})
+
\dfrac{1}{2}\log \log (R)
-\dfrac{1}{4}(\log (\pi/4) + \gamma_{\rm EM})
+ o(1),
\end{equation}
where $\gamma_{\rm EM}$ is the Euler–Mascheroni constant. 
This gives us an expression
$$\sigma_{\rm tw}^{\D}(-r,r)
=
\dfrac{1}{2}
\exp\!\left(2\cblsg +\frac{\gamma_{\rm EM}}{4}\right)\,
\left(
\frac{\pi^3}{32\,K(r^2)^2\,r(1-r^4)\,(\log R)^2}
\right)^{1/4},$$
which by using the formula 
$\log R = \dfrac{\pi K'(r^2)}{4 K(r^2)}$
can be simplified to
$$ \sigma_{\rm tw}^{\D}(-r,r)
=
\dfrac{1}{2}
\exp\!\left(2\cblsg +\frac{\gamma_{\rm EM}}{4}\right)\,
\left(
\frac{\pi}{2\,r(1-r^4)\,K'(r^2)^2}
\right)^{1/4}.$$
Recall that in this set-up, 
$\CR(r, \D) = \CR(-r,\D) = 1-r^2$. 
Further, we have by Formula 2.2.3 in \cite{lawden2013elliptic}
$$K' = -(\log q) \theta_3(q)^2/2,$$
and thus combining the relation \eqref{eq:rtheta} with the formulas \eqref{eq:thetasquares} and \eqref{eq:thetadupl}, we can write
$$r(1+r^2)K'(r^2)^2 = \frac{1}{8}\log(q)^2\theta_2(q^{1/2})^2\theta_3(q^{1/2})^2,$$
and obtain
$$ \sigma_{\rm tw}^{\D}(-r,r)
=
\frac{\pi^{1/4}}{4}\exp\!\left(2\cblsg
+\frac{\gamma_{\rm EM}}{4}\right)\,
\frac{1}{(1-r^2)^{1/4}\sqrt{\theta_2(q^{1/2})\theta_3(q^{1/2})|\log q|}}.$$
Now applying Equation \eqref{eq:thetadupl} to write $2\theta_2(q^{1/2})\theta_3(q^{1/2}) = \theta_2(q^{1/4})^2$ and plugging in the value of $\cblsg$ from Proposition \ref{Prop disk 1pt time conic} to get
\[
\frac{\pi^{1/4}}{2\sqrt2}e^{2\cblsg+\gamma_{\rm EM}/4}
=\frac{\pi^{1/4}}{2\sqrt2}\pi^{1/2}2^{-1/4}
=\frac{\pi^{3/4}}{2^{7/4}},
\]
we arrive at the desired formula.

\subsubsection{The calculation in the annulus / strip}

Recall that $p_{\ann_R}^{\rm cyl}(t,w_1,w_2)$ denotes the heat kernel 
on $\ann_R = \{ w\in\C \vert \, R^{-1}<\vert w\vert < R\}$, with $0$ boundary
conditions on $\partial \ann_R$,
with respect to the cylindrical metric
$g_{\rm cyl}(w)=\vert w\vert^{-2}\, d^2 w$,
and associated to the infinitesimal generator
$
\dfrac{\vert w\vert^2}{2} \Delta
$. This means we are working
 with the time-changed Brownian motion according to the cylindrical metric,
killed upon exiting $\ann_R$. A fundamental domain of the double covering above
is given by
$$
\ann_R^+
=
\{ w\in\C \vert \, 1<\vert w\vert < R\}.
$$We then calculate the following asymptotic
expansion.

\begin{theorem}
\label{Thm 2 pts disk R}
Fix $R>1$.
Then, as $\delta\to 0$,
$$
\int_{\ann_R^+}
\int_{\delta}^{+\infty}
p_{\ann_R}^{\rm cyl}(t,w,w^{-1})
\dfrac{dt}{t}
\,\dfrac{d^2 w}{\vert w\vert^2}
=
\dfrac{1}{4}\log(\delta^{-1})
+
\dfrac{1}{2}\log \log(R)
-\dfrac{1}{4}(\log (\pi/4) + \gamma_{\rm EM})
+ o(1),
$$
where $\gamma_{\rm EM}$ is the Euler–Mascheroni constant.
\end{theorem}
The calculation passes by the vertical strip 
$$
\strip_R
=
\{ z\in\C \vert \, -\log(R) <\Re(z) < \log(R)\}.
$$
The exponential map $z\mapsto e^z$
induces the universal covering of $\ann_R$
by $\strip_R$.
It sends the euclidean metric $d^2 z$ on $\strip_R$
to the cylindrical metric
$\vert w\vert^{-2}\, d^2 w$ on $\ann_R$.
A fundamental domain in $\strip_R$
for the double-covering is given by
the rectangle
$$
\cD
= (0,\log(R)) \times (0, 2\pi)\, .
$$
Denote by $p_{\strip_R}(t,z_1,z_2)$ the heat kernel 
on $\strip_R$, with $0$ boundary
conditions on $\partial \strip_R$,
with respect to the euclidean metric
$d^2 z$,
and associated to the infinitesimal generator
$\frac{1}{2}\Delta$
(standard Brownian motion
killed upon exiting $\strip_R$).

Then, for every $\delta>0$,
\begin{equation}
\label{Eq mass neg hol delta}
\int_{\ann_R^+}
\int_{\delta}^{+\infty}
p_{\ann_R}^{\rm cyl}(t,w,w^{-1})
\dfrac{dt}{t}
\,\dfrac{d^2 w}{\vert w\vert^2}
=
\sum_{k\in\Z}
\int_{\cD}
\int_{\delta}^{+\infty}
p_{\strip_R}(t,z,-z + 2i k \pi)
\dfrac{dt}{t}
\,d^2 z
\, .
\end{equation}
In the strip the heat kernel has a nice factorisation. Denote by $I_R$ the interval
$I_R = (-\log(R),\log(R))$,
so that $\strip_R = I_R \times \R$.
Denote $p_\R(t,x,y)$,
respectively $p_{I_R}(t,x,y)$,
the heat kernel associated to the standard 
one-dimensional Brownian
motion on $\R$,
respectively on $I_R$ killed upon exiting the interval.
Then the heat kernel on $\strip_R$
factorizes in an obvious way:
$$
p_{\strip_R}(t,z_1,z_2)
=
p_{I_R}(t,\Re(z_1),\Re(z_2))
\,
p_\R(t,\Im(z_1),\Im(z_2)).
$$
Therefore,
\eqref{Eq mass neg hol delta} equals
$$
\sum_{k\in\Z}
\int_{0}^{\log R}
dx
\int_{0}^{2\pi}
dy
\int_{\delta}^{+\infty}
p_{I_R}(t,x,-x)
\,
p_\R(t,y,-y + 2 k \pi)
\dfrac{dt}{t}\, .
$$
By translation invariance,
$$
p_\R(t,y,-y + 2 k \pi)
=
p_\R(t,0,-2y + 2 k \pi).
$$
Further, for every $t>0$,
$$
\sum_{k\in\Z}
\int_{0}^{2\pi}
p_\R(t,0,-2y + 2 k \pi) dy
=
2 \int_{\R}
p_\R(t,0,2y) dy \, = 1.
$$
Therefore, \eqref{Eq mass neg hol delta} equals
$$
\int_{0}^{\log R}
\int_{\delta}^{+\infty}
p_{I_R}(t,x,-x)
\dfrac{dt}{t}\, 
dx.
$$
Now we perform a change in scale and time
$x = u \log(R)$,
$t=s(\log R)^2$,
and get
$$
\int_{0}^{\log R}
\int_{\delta}^{+\infty}
p_{I_R}(t,x,-x)
\dfrac{dt}{t}\, 
dx
=
\int_{0}^{1}
\int_{t(\delta,R)}^{+\infty}
p_{(-1,1)}(s,u,-u)
\dfrac{ds}{s}\, 
du,
$$
where
$$
t(\delta,R)
=
\dfrac{\delta}{(\log R)^2}.
$$
So we reduced to obtaining the expansion of
\begin{equation}
\label{Eq int -1 1 delta}
\int_{0}^{1}
\int_{\delta}^{+\infty}
p_{(-1,1)}(t,x,-x)
\dfrac{dt}{t}
\, dx
\end{equation}
as $\delta\to 0$.

Further, we write $p_{(-1,1)}(t,x,-x)$
via the Fourier decomposition of the
heat kernel on $(-1,1)$:
$$
p_{(-1,1)}(t,x,-x)
=
\sum_{n=0}^{+\infty}
\cos\Big(\dfrac{(2n+1)\pi}{2}\,x\Big)^{2}
e^{-(2n+1)^{2}\pi^{2} t/8}
-
\sum_{n=1}^{+\infty}
\sin(n\pi\,x)^{2}
e^{-n^{2}\pi^{2} t/2}
\,.
$$
But for every $n\geq 0$ and $m\geq 1$,
$$
\int_{0}^{1}
\cos\Big(\dfrac{(2n+1)\pi}{2}\,x\Big)^{2}
\, dx
=
\int_{0}^{1}
\sin(m\pi\,x)^{2}
\,
dx
= \dfrac{1}{2}.
$$
Therefore,
\eqref{Eq int -1 1 delta} equals
$$
\dfrac{1}{2}
\Big(
\sum_{n=0}^{+\infty}
\int_{\delta}^{+\infty}
e^{-(2n+1)^{2}\pi^{2} t/8}
\,
\dfrac{dt}{t}
-
\sum_{n=1}^{+\infty}
\int_{\delta}^{+\infty}
e^{-n^{2}\pi^{2} t/2}
\,
\dfrac{dt}{t}
\Big)\,.
$$
Then, for every $a>0$ and $\delta>0$,
$$
\int_{\delta}^{+\infty}
e^{-at} \dfrac{dt}{t}
=
\int_{a\delta}^{+\infty}
e^{-t} \dfrac{dt}{t}
\, .
$$
Therefore,
\eqref{Eq int -1 1 delta} equals
$$
\dfrac{1}{2}
\Big(
\sum_{n=0}^{+\infty}
\int_{(2n+1)^{2}\pi^{2}\delta/8}^{+\infty}
e^{- t}
\,
\dfrac{dt}{t}
-
\sum_{n=1}^{+\infty}
\int_{n^{2}\pi^{2}\delta /2}^{+\infty}
e^{- t}
\,
\dfrac{dt}{t}
\Big)\, ,
$$
which can be further recast as
$$
\dfrac{1}{2}
\sum_{n=0}^{+\infty}
\int_{(2n+1)^{2}\pi^{2}\delta/8}^{(2n+2)^{2}\pi^{2}\delta/8}
e^{- t}
\,
\dfrac{dt}{t}\, .
$$
This in turn equals
$$
\dfrac{1}{4}
\int_{\pi^{2}\delta/8}^{+\infty}
e^{- t}
\,
\dfrac{dt}{t}
\,
+
\dfrac{1}{4}
\sum_{n=0}^{+\infty}
\Big(
\int_{(2n+1)^{2}\pi^{2}\delta/8}^{(2n+2)^{2}\pi^{2}\delta/8}
e^{- t}
\,
\dfrac{dt}{t}
\,
-
\int_{(2n+2)^{2}\pi^{2}\delta/8}^{(2n+3)^{2}\pi^{2}\delta/8}
e^{- t}
\,
\dfrac{dt}{t}
\Big).
$$

We will study the two terms above separately.
For the first one,
we have
\begin{align*}
\dfrac{1}{4}
\int_{\pi^{2}\delta/8}^{+\infty}
e^{- t}
\,
\dfrac{dt}{t}
&=
\dfrac{1}{4}
\int_{\pi^{2}\delta/8}^{1}
e^{- t}
\,
\dfrac{dt}{t}
\,
+
\dfrac{1}{4}
\int_{\1}^{+\infty}
e^{- t}
\,
\dfrac{dt}{t}
\\
&=
\dfrac{1}{4}
\log(\delta^{-1})
-
\dfrac{1}{4}
\log\Big(\dfrac{\pi^2}{8}\Big)
-
\dfrac{1}{4}
\int_{0}^{1}
\dfrac{1-e^{-t}}{t}
\, dt\,
+
\dfrac{1}{4}
\int_{\1}^{+\infty}
e^{- t}
\,
\dfrac{dt}{t}
\,
+ o(1)\,.
\end{align*}

Now let us deal with the second term.

\begin{lemma}
\label{Lem lim sum diff}
The following convergence holds:
\begin{equation}
\label{Eq conv n delta int log}
\lim_{\delta\to 0}
\sum_{n=0}^{+\infty}
\Big(
\int_{(2n+1)^{2}\pi^{2}\delta/8}^{(2n+2)^{2}\pi^{2}\delta/8}
e^{- t}
\,
\dfrac{dt}{t}
\,
-
\int_{(2n+2)^{2}\pi^{2}\delta/8}^{(2n+3)^{2}\pi^{2}\delta/8}
e^{- t}
\,
\dfrac{dt}{t}
\Big)
=
\sum_{n=0}^{+\infty}
\log\Big(
\dfrac{(2n+2)^2}{(2n+1)(2n+3)}
\Big)
.
\end{equation}
\end{lemma}

\begin{proof}
First note that
$$
\log\Big(
\dfrac{(2n+2)^2}{(2n+1)(2n+3)}
\Big)
=
\int_{(2n+1)^{2}\pi^{2}\delta/8}^{(2n+2)^{2}\pi^{2}\delta/8}
\dfrac{dt}{t}
\,
-
\int_{(2n+2)^{2}\pi^{2}\delta/8}^{(2n+3)^{2}\pi^{2}\delta/8}
\dfrac{dt}{t}\, ,
$$
and the value does not depend on $\delta$.
It is also
$$
\lim_{\delta\to 0}
\Big(
\int_{(2n+1)^{2}\pi^{2}\delta/8}^{(2n+2)^{2}\pi^{2}\delta/8}
e^{- t}
\,
\dfrac{dt}{t}
\,
-
\int_{(2n+2)^{2}\pi^{2}\delta/8}^{(2n+3)^{2}\pi^{2}\delta/8}
e^{- t}
\,
\dfrac{dt}{t}
\Big).
$$
Also,
$$
\dfrac{(2n+2)^2}{(2n+1)(2n+3)}
= 1 + O\Big(\dfrac{1}{n^2}\Big),
$$
so the series
$$
\sum_{n=0}^{+\infty}
\log\Big(
\dfrac{(2n+2)^2}{(2n+1)(2n+3)}
\Big)
$$
is absolutely convergent.

In the sequel, to simplify the notations, we will denote
$$
a_{n,\delta}
=
\int_{(2n+1)^{2}\pi^{2}\delta/8}^{(2n+2)^{2}\pi^{2}\delta/8}
e^{- t}
\,
\dfrac{dt}{t}
\,
-
\int_{(2n+2)^{2}\pi^{2}\delta/8}^{(2n+3)^{2}\pi^{2}\delta/8}
e^{- t}
\,
\dfrac{dt}{t},
\qquad
a_{n,0}
= \log\Big(
\dfrac{(2n+2)^2}{(2n+1)(2n+3)}
\Big).
$$
We will also introduce integers $N_\delta$
depending on $\delta$,
with $N_\delta\to +\infty$
as $\delta\to 0$,
such that
\begin{equation}
\label{Eq lim sum n leq N delta}
\lim_{\delta\to 0}
\sum_{n=0}^{N_\delta}
\vert a_{n,\delta}-a_{n,0}\vert
= 0
\end{equation}
and
\begin{equation}
\label{Eq lim sum n geq N delta}
\lim_{\delta\to 0}
\sum_{n=N_\delta + 1}^{+\infty}\vert a_{n,\delta}\vert = 0.
\end{equation}
The existence of such a family $(N_\delta)_{\delta>0}$
guarantees the convergence \eqref{Eq conv n delta int log}.
For clarity, we will first derive a set of sufficient conditions for $(N_\delta)_{\delta>0}$,
and then give $(N_\delta)_{\delta>0}$
satisfying all these conditions.

First of all, we have a trivial bound
$$
\vert a_{n,\delta} \vert
\leq \dfrac{C}{n}e^{-c n^{2}\delta}
$$
for constants $c,C>0$
depending neither on $n\geq 1$ nor $\delta>0$.
Then,
$$
\sum_{n=N_\delta + 1}^{+\infty}\vert a_{n,\delta}\vert
\leq
\dfrac{C}{N_\delta}\sum_{n=N_\delta + 1}^{+\infty}
e^{-c \delta N_\delta n }
\leq
C\dfrac{e^{-c \delta N_\delta^2 }}{N_\delta (1-e^{-c \delta N_\delta})}\, .
$$
Therefore,
to get \eqref{Eq lim sum n geq N delta}
it suffices that
$N_\delta\gg\delta^{-1/2}$.

Now, let $f(t)$ be the function
$$
f(t) = \dfrac{1-e^{-t}}{t} \, .
$$
It extends to an analytic function on $\C$.
The derivative is given by
$$
f'(t) = - \dfrac{1-e^{-t} - te^{-t}}{t^2} \, ,
$$
and it is bounded on $[0,+\infty)$.
We will denote by
$\Vert f\Vert_{\infty}$
and $\Vert f'\Vert_{\infty}$
the sup norms of $f$, respectively $f'$,
on $[0,+\infty)$.
In this way,
\begin{align*}
\vert a_{n,\delta}-a_{n,0}\vert
&=
\Big\vert
\int_{(2n+1)^{2}\pi^{2}\delta/8}^{(2n+2)^{2}\pi^{2}\delta/8}
f(t)\,dt
\,
-
\int_{(2n+2)^{2}\pi^{2}\delta/8}^{(2n+3)^{2}\pi^{2}\delta/8}
f(t)\,dt
\Big\vert
\\
&\leq
\int_{(2n+1)^{2}\pi^{2}\delta/8}^{(2n+2)^{2}\pi^{2}\delta/8}
\vert f(t) - f(t+n\pi^2 \delta/2 +3\pi^2\delta/8)\vert\, dt
\,
+
\int_{\substack{(2n+2)^{2}\pi^{2}\delta/8\\+n\pi^2 \delta/2 +3\pi^2\delta/8}}
^{(2n+3)^{2}\pi^{2}\delta/8}
f(t)\, dt
\\
&\leq(n\pi^2 \delta/2 +3\pi^2\delta/8)^2
\Vert f'\Vert_{\infty}
+
(3\pi^2\delta/4)
\Vert f\Vert_{\infty}\, .
\end{align*}
So, for $n\leq N_\delta$,
$$
\vert a_{n,\delta}-a_{n,0}\vert
\leq 
(N_\delta\pi^2 \delta/2 +3\pi^2\delta/8)^2
\Vert f'\Vert_{\infty}
+
(3\pi^2\delta/4)
\Vert f\Vert_{\infty}\, ,
$$
and
$$
\sum_{n=0}^{N_\delta}
\vert a_{n,\delta}-a_{n,0}\vert
\leq
(N_\delta + 1)
((N_\delta\pi^2 \delta/2 +3\pi^2\delta/8)^2
\Vert f'\Vert_{\infty}
+
(3\pi^2\delta/4)
\Vert f\Vert_{\infty})\,.
$$
So to get \eqref{Eq lim sum n leq N delta},
it suffices that
$N_\delta^3 \delta^2 = o(1)$
and $N_\delta \delta = o(1)$,
that is to say
$N_\delta \ll \delta^{-2/3}$.
So to get both \eqref{Eq lim sum n leq N delta}
and \eqref{Eq lim sum n geq N delta}
it suffices to take $N_\delta$ such that
$\delta^{-1/2}\ll N_\delta \ll \delta^{-2/3}$,
for instance
$N_\delta = \lfloor \delta^{-7/12}\rfloor$.
\end{proof}
By combining everything we obtain
$$
\int_{\ann_R^+}
\int_{\delta}^{+\infty}
p_{\ann_R}^{cyl}(t,w,w^{-1})
\dfrac{dt}{t}
\,\dfrac{d^2 w}{\vert w\vert^2}
=
\dfrac{1}{4}\log(\delta^{-1})
+
\dfrac{1}{2}\log\log(R)
+ C
+ o(1),
$$
where
$$
C
=
-
\dfrac{1}{4}
\log\Big(\dfrac{\pi^2}{8}\Big)
-
\dfrac{1}{4}
\int_{0}^{1}
\dfrac{1-e^{-t}}{t}
\, dt\,
+
\dfrac{1}{4}
\int_{1}^{+\infty}
e^{- t}
\,
\dfrac{dt}{t}
\,
+
\dfrac{1}{4}
\sum_{n=0}^{+\infty}
\log\Big(
\dfrac{(2n+2)^2}{(2n+1)(2n+3)}
\Big)
.
$$
First note that
$$
\gamma_{\rm EM}
=
\int_{0}^{1}
\dfrac{1-e^{-t}}{t}
\, dt\,
-
\int_{1}^{+\infty}
e^{- t}
\,
\dfrac{dt}{t}\, .
$$
Further, by Wallis' formula 
(see e.g. \cite{WallisArithmeticaInfinitorum,Moll23WallisIntegral}),
$$
\sum_{n=0}^{+\infty}
\log\Big(
\dfrac{(2n+2)^2}{(2n+1)(2n+3)}
\Big)
=
\log(\pi/2).
$$
So actually,
$
C = -\frac{1}{4}(\log (\pi/4) + \gamma_{\rm EM}),
$
and the theorem follows.

\subsection{Twist fields with constant non-zero boundary conditions}
In this section we will introduce the twist field 
correlation functions for the GFF with non-zero boundary condition. 
In the case of two points, introducing boundary conditions is an ad-hoc way to talk about twist fields for the compactified free field - we will properly handle compactified twisted free fields in an upcoming work. Thus we will remain brief and restrict ourselves to two points and state the results used in our calculation in one theorem. \\

\noindent Recall the notation $D_{\varepsilon_1,\varepsilon_2}(z_1,z_2)$
for the domain with two holes
\begin{equation}
\label{Eq D eps1 eps2 z1 z2}
D_{\varepsilon_1,\varepsilon_2}(z_1,z_2)
=D\setminus(\overline{\D(z_1,\varepsilon_1)}
\cup
\overline{\D(z_2,\varepsilon_2)}).
\end{equation}
Also recall that for $\gamma$ an excursion
from $\partial D$ to $\partial D$ inside $D$,
the parity of the sum of indices
$\indx_\gamma(z_1)+\indx_\gamma(z_2)$
is well defined;
see Section \ref{Subsec topo}.

\begin{theorem}[Two-point correlation of twist field with boundary data]
\label{thm:twistbc}
The twist field two-point correlation 
$\sigma_{\rm tw}^{D,v}(z_{1}, z_{2})$ for constant boundary condition $v \in \R$ can be defined in the following equivalent ways.
\begin{enumerate}[i)]

\item
As $\max(\eps_1,\eps_2)\to 0$, we consider the limit of
\begin{equation}\label{eqdef:twistbc}
\begin{aligned}
&\frac{\pi/2}{(\varepsilon_1\varepsilon_2)^{1/8}\sqrt{|\log \varepsilon_1|\,|\log \varepsilon_2|}} 
\exp\Bigg[
-\mu^{\rm loop}_{D_{\varepsilon_1,\varepsilon_2}(z_1,z_2)}
\bigl(\indx_\gamma(z_1)+\indx_\gamma(z_2)\text{ odd}\bigr) \\
&\qquad\qquad\qquad\qquad\qquad\qquad
-v^2\mu^{\rm exc}_{D}
\Bigl(
\indx_\gamma(z_1)+\indx_\gamma(z_2)\text{ odd },\,
\gamma\subset D_{\varepsilon_1,\varepsilon_2}(z_1,z_2)
\Bigr)
\Bigg].
\end{aligned}
\end{equation}
\item
Consider
$\cL_{D_{\varepsilon_1,\varepsilon_2}(z_1,z_2)}$
the Brownian loop soup in
$D_{\varepsilon_1,\varepsilon_2}(z_1,z_2)$
with intensity
$\frac{1}{2}\mu^{\rm loop}_{D_{\varepsilon_1,\varepsilon_2}(z_1,z_2)}$.
For $v\geq 0$, denote by
$\Xi^{v^2/2}_{D, D_{\varepsilon_1,\varepsilon_2}(z_1,z_2)}$
the Poisson point process of boundary excursions in
$D$ with intensity
\[
\dfrac{v^2}{2}\,
\1_{\{\gamma\subset D_{\varepsilon_1,\varepsilon_2}(z_1,z_2)\}}
\mu^{\rm exc}_{D}(d\gamma).
\]
As $\max(\eps_1,\eps_2)\to 0$, we consider the limit of
\begin{equation}\label{def:twist3}
\frac{\pi/2}{(\eps_1\eps_2)^{1/8}\sqrt{|\log \eps_1|\,|\log \eps_2|}}
\P\Big(
\text{No cluster of }
\cL_{D_{\varepsilon_1,\varepsilon_2}(z_1,z_2)}
\cup
\Xi^{v^2/2}_{D, D_{\varepsilon_1,\varepsilon_2}(z_1,z_2)}
\text{ disconnects } z_1 \text{ from } z_2
\Big).
\end{equation}
\end{enumerate}
\noindent Moreover, for all $v\in\R$, its value is equal to
\begin{equation}\label{eq:twistbnd}
\sigma_{\rm tw}^{D,v}(z_1,z_2)
=
\sigma_{\rm tw}^{D}(z_1,z_2)
\exp\bigl(-v^2 \Modd_D(z_1,z_2)\bigr),
\end{equation}
where $\Modd_D(z_1,z_2) = \mu^{\rm exc}_{D}
\bigl(\indx_\gamma(z_1)+\indx_\gamma(z_2)\text{ odd}\bigr)$.
Further, $\exp(-\frac{\pi}{4}\Modd_D(z_1,z_2)) = \hat q$,
where $\hat q$ is the complementary nome to nome $q$ (i.e. $\log q \log \hat q = \pi^2$), that in turn is by Lemma \ref{lem:parametr} uniquely defined via 
$$\exp\Big(\pi G_D(z_1,z_2)\Big) 
= \dfrac{\theta_3(q^{1/2})}{\theta_2(q^{1/2})}.$$
\end{theorem}

The theorem follows directly by combining the result for zero boundary conditions, Theorem \ref{thm:deftwist}, with Proposition \ref{Prop n pt cut disks bc continuum} and the following calculation.

\begin{lemma}\label{lem:Modd}
Denote
\[
\Modd_D(z_1,z_2)
=
\mu^{\rm exc}_{D}
\bigl(\indx_\gamma(z_1)+\indx_\gamma(z_2)\text{ odd}\bigr).
\]
We have that
\[
\Modd_D(z_1,z_2)<+\infty
\]
and
\[
\Modd_D(z_1,z_2)
=
\lim_{\max(\varepsilon_1,\varepsilon_2)\to 0}
\mu^{\rm exc}_{D}
\bigl(
\indx_\gamma(z_1)+\indx_\gamma(z_2)\text{ odd},\,
\gamma\subset D_{\varepsilon_1,\varepsilon_2}(z_1,z_2)
\bigr).
\]
Further, we have the explicit formula
$$\exp\Big(-\frac{\pi}{4}\Modd_D(z_1,z_2)\Big) = \hat q$$
where $\hat q$ is the complementary nome to nome $q$ that is by Lemma \ref{lem:parametr} uniquely defined via 
$$\exp\Big(\pi G_D(z_1,z_2)\Big) 
= \dfrac{\theta_3(q^{1/2})}{\theta_2(q^{1/2})}.$$
\end{lemma}
\begin{proof}

By conformal invariance it suffices to prove the claim for $D = \D$ and $z_1 = r, z_2 = -r$. We can use the branched double cover from Section \ref{sec:reduction} to conformally double cover $D$ with the annulus $\ann_R = \{ w\in\C \vert \, R^{-1}<\vert w\vert < R\}$, where $\log R = \pi K'(r^2) / 4K(r^2)$, branch points are at $\pm r$ and $K(r^2), K'(r^2)$ are the elliptic integrals defined in that section.

Observe that under this map 
$$\mu^{\rm exc}_{\D}
(\indx_\gamma(-r)+\indx_\gamma(r)\text{ odd}) = 
\frac{1}{2}\mu^{\rm exc}_{\ann_R}(\text{endpoints of }\gamma\text{ are on opposite boundaries}),$$
where the factor $1/2$ comes from the fact that each path has two lifts starting from each of the boundary components.

But the boundary-to-boundary excursion mass $\mu^{\rm exc}_{\ann_R}(\text{endpoints of }\gamma\text{ are on opposite boundaries})$ is known to be equal to twice the inverse of extremal length of the annulus (e.g. Section 2.1 Equation 2.6 in \cite{ALS1} and Theorem 4-5 in \cite{Ahlfors2010ConfInv}). 
As the extremal length of the annulus 
$\ann_R$ is given by (see Lemma \ref{lemma:cover})
$$\frac{1}{2\pi}\log ( R/R^{-1}) = \frac{\log R}{\pi} = K'(r^2)/4K(r^2),$$
we conclude that
\[
\Modd_D(z_1,z_2)=\frac{4K(r^2)}{K'(r^2)}.
\]
It remains to now recall that the nome $q$ given in the statement can be equivalently written as $q = \exp(-\pi K'(r^2)/K(r^2))$.

Finally, the convergence statement then just follows from the Monotone convergence theorem.
\end{proof}

\section{Twist field correlations in terms of excursion sets of the GFF and CLE$_4$}
\label{Sec Twist CLE4}

Denote by 
$\cL_{D_{\varepsilon_1,\varepsilon_2}(z_1,z_2)}$ the Brownian loop soup in 
$D_{\varepsilon_1,\varepsilon_2}(z_1,z_2)$
with intensity
$\frac{1}{2}
\mu^{\rm loop}_{D_{\varepsilon_1,\varepsilon_2}(z_1,z_2)}$.
For $v\geq 0$, denote by
$\Xi^{v^2/2}_{D, D_{\varepsilon_1,\varepsilon_2}(z_1,z_2)}$
the Poisson point process of boundary excursions in
$D$ with intensity
$$
\dfrac{v^2}{2}
\1_{\gamma\subset D_{\varepsilon_1,\varepsilon_2}(z_1,z_2)}
\mu^{\rm exc}_{D}(d\gamma).
$$
We also take 
$\Xi^{v^2/2}_{D, D_{\varepsilon_1,\varepsilon_2}(z_1,z_2)}$
to be independent from
$\cL_{D_{\varepsilon_1,\varepsilon_2}(z_1,z_2)}$. 
Denote by $\Phi_{D}^{(v)}$ the continuum GFF on $D$
with boundary condition $v\geq 0$ on $\partial D$.

\begin{theorem}\label{thm:equivtwist}
The twist correlation function 
$\sigma_{\rm tw}^{D,v}(z_1,z_2)$ for the field with constant boundary condition $v \in \R$ can also be equivalently expressed as the following renormalised probabilities:
\begin{itemize}
\item Let $E_{\eps_1,\eps_2}^{v,\rm{holes}}$ 
be the event that no cluster of  
$\cL_{D_{\varepsilon_1,\varepsilon_2}(z_1,z_2)}
\cup
\Xi^{v^2/2}_{D, D_{\varepsilon_1,\varepsilon_2}(z_1,z_2)}$
disconnects $\D(z_1, \eps_1)$ from $\D(z_2, \eps_2)$. 
Then we have  
$$\sigma_{\rm tw}^{D,v}(z_1,z_2) 
= \lim_{\eps_1,\eps_2 \to 0} 
\dfrac{\pi}{2} 
(\eps_1\eps_2)^{-1/8}|\log \eps_1|^{-1/2}|\log \eps_2|^{-1/2}
\P(E_{\eps_1,\eps_2}^{v,\rm{holes}}).$$

\item Let $E_{\eps_1,\eps_2}^{v}$ be the event that 
$\text{no cluster of }
\cL_{D}
\cup
\Xi^{v^2/2}_{D}
\text{ disconnects } \D(z_1,\eps_1) 
\text{ from } \D(z_2,\eps_2)$.
Then 
$$\sigma_{\rm tw}^{D,v}(z_1,z_2) 
= \lim_{\eps_1,\eps_2 \to 0} (\CSLE)^{-2}(\eps_1\eps_2)^{-1/8}\P(E_{\eps_1,\eps_2}^{v}).$$

\item Let $E_{\eps_1,\eps_2}^{v,\rm{exc}}$ be the event that no sign excursion cluster of the GFF in $D$ with constant boundary condition $v \in \R$ disconnects 
$\D(z_1,\eps_1)$ from $\D(z_2, \eps_2)$. Then we have  
$$\sigma_{\rm tw}^{D,v}(z_1,z_2) 
= \lim_{\eps_1,\eps_2 \to 0} (\CSLE)^{-2} (\eps_1\eps_2)^{-1/8}\P(E_{\eps_1,\eps_2}^{v,\rm{exc}}).$$

\item Let $E_{\eps_1,\eps_2}^{v,\rm{CLE}}$ 
be the event that a $0$-height CLE$_4$ gasket in the nested CLE$_4$ coupling with $\Phi_D^{(v)}$
intersects both $\D(z_1,\eps_1)$ and $\D(z_2, \eps_2)$. 
Then we have  
$$\sigma_{\rm tw}^{D,v}(z_1,z_2) 
= \lim_{\eps_1,\eps_2 \to 0} (\CSLE)^{-2} (\eps_1\eps_2)^{-1/8}
\P(E_{\eps_1,\eps_2}^{v,\rm{CLE}}).$$
\end{itemize}
The constant $\CSLE$ above is given by 
Proposition \ref{Prop asymp 1 pt ED Ver 2}.
\end{theorem}
\noindent In fact several parts of this theorem are immediately clear.
\begin{itemize}
\item The first point follows directly from the definition of twist fields correlations 
with boundary conditions in Theorem \ref{thm:twistbc} and Proposition \ref{Prop n pt cut disks bc continuum}.
\item The equivalence of the second and third point, i.e. 
$\P(E_{\eps_1,\eps_2}^v) = 
\P(E_{\eps_1,\eps_2}^{v,\rm{exc}})$ follows from the fact that the sign excursion clusters of the GFF with constant boundary condition $v$ in a simply-connected domain are given exactly by the clusters of 
$\cL_{D}
\cup
\Xi^{v^2/2}_{D}$, see \cite[Theorem 2]{ALS4} for the zero boundary case, and combine it with \cite[Proposition 5.3]{ALS2} for the general case.  
\item Further, as by Propositions \ref{prop:CLEcoupling1} and \ref{prop:CLEcoupling2}, the $0$-level CLE$_4$ gaskets form the dual of the sign excursion clusters of $\Phi_D^{(v)}$, 
then by planar duality we also have that the events 
$E_{\eps_1,\eps_2}^{v,\rm{exc}} 
= E_{\eps_1,\eps_2}^{v,\rm{CLE}}$. 
\end{itemize}
Thus it remains to prove the equivalence between the two first bullet points. \\

We begin by stating a precise first moment equivalent in the setting of 
the unit disk $\D$. 
For $r\in (0,1)$, we define the event
$$
E_{\D}(r)=
\{
\text{No cluster of }
\cL_{\D}
\text{ surrounds }
\D(0,r)
\}.
$$
The one point estimate would be direct from \cite{SSW09CR, ALS3} if we used conformal radius to measure the distance - i.e. if we considered the event that the conformal radius of $0$ in the connected component of the complement of clusters is less than $r$. However, when using the actual Euclidean distance, the one-point function includes an explicit functional of the SLE$_4$ loop measure.
\begin{prop}
\label{Prop asymp 1 pt ED}
As $r\to 0$,
$$
\P(E_{\D}(r))
=
\CSLE~
r^{1/8}(1+o(1)),
$$
where
$\CSLE$ is an expectation under 
$\mu_{\text{SLE}_4}^{\rm loop}$ of an explicit functional,
determined in Proposition \ref{Prop asymp 1 pt ED Ver 2}.
\end{prop}
Next, let us describe the two-point estimates that allow us to conclude the theorem. 
Recall that we denote by 
$E_{\varepsilon_1,\varepsilon_2}^{v}$ the event
$$
E_{\varepsilon_1,\varepsilon_2}^{v}
=
\{
\text{No cluster of }
\cL_{D}
\cup
\Xi^{v^2/2}_{D}
\text{ disconnects } \D(z_1,\varepsilon_1) 
\text{ from } \D(z_2,\varepsilon_2)
\}.
$$
Similarly, whe consider the Brownian loop soup and boundary excursions in the domain
$$
D_{\eps_1,\eps_2}(z_1,z_2)
=
D\setminus (\overline{\D(z_1,\eps_1)}
\cup
\overline{\D(z_2,\eps_2)}),
$$
that is 
$\cL_{D_{\varepsilon_1,\varepsilon_2}(z_1,z_2)}$ and $\Xi^{v^2/2}_{D, D_{\varepsilon_1,\varepsilon_2}(z_1,z_2)}$.
We have the analogous event
$$E_{\varepsilon_1,\varepsilon_2}^{v, \rm{holes}} 
= \{
\text{No cluster of }
\cL_{D_{\varepsilon_1,\varepsilon_2}(z_1,z_2)}
\cup
\Xi^{v^2/2}_{D, D_{\varepsilon_1,\varepsilon_2}(z_1,z_2)}\text{ disconnects } z_1 \text{ from }z_2\}.$$
Notice that the renormalised probability of this latter event is calculated in Theorem \ref{thm:twistbc}.

To connect the probabilities of $E_{\eps_1, \eps_2}^v$ and 
$E_{\eps_1, \eps_2}^{v, \rm{holes}}$, we will see 
$\cL_{D_{\varepsilon_1,\varepsilon_2}(z_1,z_2)}$
as a subfamily of
$\cL_{D}$ and $\Xi^{v^2/2}_{D, D_{\varepsilon_1,\varepsilon_2}(z_1,z_2)}$ as a subfamily of
$\Xi^{v^2/2}_{D}$. We will also define a separation scale between local separation events and more macroscopic ones. To do this, let $r_0\in (0,1]$ be small enough so that
$$
\overline{\D(z_j,r_0)}\subset D,
j\in\{1,2\},
\qquad
\overline{\D(z_1,r_0)}
\cap
\overline{\D(z_2,r_0)}
=\emptyset.
$$

Recall the notation
$$
D_{r_0,r_0}(z_1,z_2) = 
D\setminus(\D(z_1,r_0)\cup\D(z_2,r_0)).
$$
Denote by $\mathfrak{C}_{r_0}$ the collection of clusters of $\cL_{D}\cup\Xi^{v^2/2}_{D}$
that intersect $\overline{D_{r_0,r_0}(z_1,z_2)}$.
We see a cluster as a closed subset of $\overline{D}$, by taking the topological closure.
So $\mathfrak{C}_{r_0}$ is a random countable collection of disjoint closed subsets of
$\overline{D}$, each intersecting $\overline{D_{r_0,r_0}(z_1,z_2)}$.
Let $E_{r_0}$ be the following event:
\begin{equation}
\label{Eq def E r0}
E_{r_0}
=
\{
\text{No cluster of }
\mathfrak{C}_{r_0}
\text{ disconnects } z_1 
\text{ from } z_2
\}.
\end{equation}
We can now state the key propositions and give the asymptotic expansions of the probabilities of $E_{\eps_1, \eps_2}^v$ and $E_{\eps_1, \eps_2}^{v,\rm{holes}}$. 
We stress that in both of these statements the event $E_{r_0}$ depends on the loop soup of the whole domain.
\begin{prop}
\label{Prop local asymp expect CR 1 8}
There is an explicit constant $\CSLE$ such that as 
$\max(\varepsilon_1,\varepsilon_2)\to 0$, we have that
\begin{multline*}
\P(E_{\varepsilon_1,\varepsilon_2}^v)
=
\varepsilon_{1}^{1/8}
\varepsilon_{2}^{1/8}
(\CSLE)^2
\,
\E
\big[
\CR(z_1,\widehat{D}_{r_0}(z_1))^{-1/8}
\CR(z_2,\widehat{D}_{r_0}(z_2))^{-1/8}
\1_{E_{r_0}}
\big]
(1+o(1))
.
\end{multline*}
Above, $\widehat{D}_{r_0}(z_1)$ and
$\widehat{D}_{r_0}(z_2)$
are random simply connected domains around
$z_1$, respectively $z_2$,
corresponding to the connected component of
$z_1$, respectively $z_2$, in
$
D\setminus
\overline{
\bigcup_{\cC\in \mathfrak{C}_{r_0}}\cC
}
.
$
\end{prop}

\begin{prop}
\label{Prop asymp 2 pt cut disks CR 1 8}
As $\max(\varepsilon_1,\varepsilon_2)\to 0$, we have that 
\begin{multline*}
\P(E_{\varepsilon_1,\varepsilon_2}^{v, \rm{holes}})
=
\dfrac{2}{\pi}
\varepsilon_{1}^{1/8}
\varepsilon_{2}^{1/8}
\log(\varepsilon_1^{-1})^{1/2}
\log(\varepsilon_2^{-1})^{1/2}
\\
\times
\E
\big[
\CR(z_1,\widehat{D}_{r_0}(z_1))^{-1/8}
\CR(z_2,\widehat{D}_{r_0}(z_2))^{-1/8}
\1_{E_{r_0}}
\big]
(1+o(1))
.
\end{multline*}
\end{prop}
\noindent Combining these propositions gives us the missing step of Theorem \ref{thm:equivtwist}. The rest of this section is dedicated to proving these three propositions.

\subsection{Proof of Proposition \ref{Prop asymp 1 pt ED}}

We need to estimate
$\P(E_{\D}(r))$
for $r\in (0,1)$.
Denote by $\cC_{0}=\cC_{0}(\cL_{\D})$ the outermost cluster of $\cL_{\D}$
surrounding $0$.
Denote by $\partial_{e}\cC_{0}$
the outer boundary of $\cC_{0}$ and by $\mathcal{D}_0$ the simply-connected domain containing the origin and surrounded by $\cC_0$. By \cite{SheffieldWerner2012CLE}, 
the outer boundary $\cC_{0}$ is the same as the CLE$_{4}$ loop surrounding $0$.
Then the event $E_{\D}(r)$ is the same as
$$
E_{\D}(r) = 
\{d(0,\partial_{e}\cC_{0})<r\}.
$$
We will express the asymptotic of $\P(E_{\D}(r))$
through an invariant probability measure on SLE$_{4}$ loops
introduced in \cite{KemppainenWerner16CLE}. 
See Section \ref{Subsec stat SLE4} for further references. 

Consider Jordan loops 
$\Gamma$ in $\C$, not hitting $0$ and surrounding $0$, 
and an equivalence relation on such Jordan loops,
where we identify $\Gamma$ with all its rescales
$\lambda \Gamma$ for $\lambda>0$.
Denote by $[\Gamma]$ the equivalence class of
$\Gamma$.
We will refer to $[\Gamma]$ as the \textit{shape}
of $\Gamma$.
Recall the notations
$$
R^{-}(\Gamma) = d(0,\Gamma),
\qquad
R^{+}(\Gamma) = \max\{\vert z\vert \vert z\in \Gamma\}.
$$
Notice that the ratio $\dfrac{R^{+}(\Gamma)}{R^{-}(\Gamma)}$
is scale-invariant,
and thus, it depends only on the shape $[\Gamma]$. We denote this ratio by $\rho^{\pm}([\Gamma])$.
Recall also the functional \eqref{Eq intro Y Gamma}:
$$Y([\Gamma]) = \int_{0}^{2\pi}
G_{\Ext(\Gamma)}(R^{+}(\Gamma)e^{i\alpha},\infty)
d\alpha.$$   
We consider now the stationary SLE$_4$-loop probability distribution on shapes $[\Gamma]$, denoted by 
$\mu_{\text{SLE}_4}^{\rm loop}$. In these notations, the following more precise version of Proposition \ref{Prop asymp 1 pt ED} holds.
\begin{prop}
\label{Prop asymp 1 pt ED Ver 2}
Denote by $[\Gamma^{\ast}]$ a random shape distributed according to $\mu_{\text{SLE}_4}^{\rm loop}$.
Then, as $r\to 0$,
$$
\P(E_{\D}(r))
=
\CSLE~
r^{1/8}(1+o(1)),
$$
where
$$
\CSLE
=
\dfrac{8}{\pi^2}
\E\Big[
\rho^{\pm}([\Gamma^{\ast}])^{1/8}
e^{-Y([\Gamma^{\ast}])/8}
\Big]
<+\infty.
$$
\end{prop}
There are two key inputs into proving this proposition.
\begin{itemize}
\item First, the explicit asymptotics as $r \to 0$ for the probabilities of the events $\{\CR(0, \cD_0) < r\}$ and $\{2\pi \ED(\partial \D, \partial \cC_0) > - \log r\}$. 
\item Second, a more quantitative version of  
\cite[Corollary 2]{KemppainenWerner16CLE}, 
which says that as the extremal distance $\ED(\partial \D, \partial \cC_0)$ goes to $\infty$, the shape of the CLE$_4$ loop $\partial \cC_0$ converges to a shape distributed according to $\mu_{\text{SLE}_4}^{\rm loop}$.
\end{itemize}
The first input follows from \cite[Theorem 1.1]{ALS3}. The second input is packed in a deterministic statement given below. This statement helps us compare the CLE$_4$ loop in the unit disk surrounding the origin with a CLE$_4$ loop in a random domain whose boundary has a shape 
distributed according to $\mu_{\text{SLE}_4}^{\rm loop}$.
Notice that by \cite[Section 3.1]{KemppainenWerner16CLE} the shape of this latter CLE$_4$ loop is also distributed according to $\mu_{\text{SLE}_4}^{\rm loop}$. \\

So consider two disjoint Jordan loops
$\Gamma_1$ and $\Gamma_2$,
both avoiding and surrounding the point $0$,
and $\Gamma_1$ surrounding $\Gamma_2$.
We further assume that
$\CR(0,\C\setminus \Gamma_1)=1$.
Denote by $f_{\Gamma_1}$ the unique conformal mapping
uniformizing $\Int(\Gamma_1)$,
the interior of $\Gamma_1$, to the unit disk $\D$,
with $f_{\Gamma_1}(0)=0$ and
$f'_{\Gamma_1}(0)=1$. We further set
$
\Gamma = f_{\Gamma_1}(\Gamma_2),
$
so that $\Gamma$ is a Jordan loop in $\D$
surrounding $0$. 

By conformal invariance,
$$
\ED(\Gamma_2,\Gamma_1)
=
\ED(\Gamma,\partial\D) .
$$
We will be interested in comparing the shapes of $\Gamma$ and $\Gamma_2$ quantitatively as they become smaller and smaller. 
\begin{prop}
\label{Prop diff Y}
There are universal constants $R_1, C_1, C_2, C_3 > 0$ not depending on $\Gamma_1$ and $\Gamma_2$ such that the following holds.
\begin{enumerate}
\item 
If $R^{+}(\Gamma)\leq R_1$,
then
$\vert R^{+}(\Gamma) - R^{+}(\Gamma_2)\vert\leq C_1
R^{+}(\Gamma)^2$,
and
$\vert d(0,\Gamma)-d(0,\Gamma_2)\vert\leq
C_2d(0,\Gamma)^2$.
\item 
For every
$\Gamma_1,\Gamma_2$ as above,
and $\Gamma = f_{\Gamma_1}(\Gamma_2)$,
we have
$$
\vert Y([\Gamma_2]) - Y([\Gamma])\vert
\leq C_3
R^{+}(\Gamma)^{1/2}.
$$
\end{enumerate}
\end{prop}
We will postpone the proof of this proposition and first see how to conclude Proposition \ref{Prop asymp 1 pt ED Ver 2}.
\begin{proof}[Proof of Proposition \ref{Prop asymp 1 pt ED Ver 2}]
Let $\check{\Gamma}$ denote the CLE$_4$ loop in $\D$ surrounding $0$. We are interested in the probability of the event
$E_{\D}(r)$, which corresponds to
$$
E_{\D}(r) = \{d(0,\check{\Gamma})<r\}.
$$
Consider now the event
$$
Q(r) = \{R^{+}(\check{\Gamma})\leq 
r^{-1/2}d(0,\check{\Gamma})\}.
$$
By \eqref{Eq bounds ED} and Koebe quarter theorem,
on the complementary event $Q(r)^{\rm c}$,
$$
\log \CR(0,\C\setminus \check{\Gamma})^{-1}
-
2\pi \ED(\check{\Gamma},\partial \D)
\geq 
\dfrac{1}{2}\log r^{-1}
- 4\log(2).
$$
In \cite{ALS3} it was shown that the above random variable is distributed as $T_{\pi}^{\rm BES3}$, 
see also the proof of Proposition \ref{Prop law shape}.
Thus,
$$
\P(Q(r)^{\rm c})
\leq
\P\Big(T_{\pi}^{\rm BES3}\geq \dfrac{1}{2}\log r^{-1} - 4\log(2)\Big)
= r^{1/4 + o(1)}
= o(r^{1/8}).
$$
So it is enough to get an asymptotic for
$\P(E_{\D}(r)\cap Q(r))$.

The event $E_{\D}(r)$ can be rewritten as
$$
E_{\D}(r) = 
\{ 2\pi\ED(\check{\Gamma}_2,\check{\Gamma}_1)
>
\log r^{-1}
-
\log\rho^{\pm}([\check{\Gamma}_2])
+
Y([\check{\Gamma}_2])
+
\eta
\},
$$
where the random variable $\eta$ is
\begin{eqnarray*}
\nonumber
\eta &=& 
\big(
2\pi \ED(\check{\Gamma},\partial\D)
-
\log(R^{+}(\check{\Gamma})^{-1})
-
Y([\check{\Gamma}])
\big)
+
\log(R^{+}(\check{\Gamma}_2)/R^{+}(\check{\Gamma}))
\\
&& +
\log(d(0,\check{\Gamma})/d(0,\check{\Gamma}_2))
+
(Y([\check{\Gamma}])-Y([\check{\Gamma}_2]))
.
\end{eqnarray*}
Further, on the event $E_{\D}(r)\cap Q(r)$,
$$
R^{+}(\check{\Gamma})
\leq
r^{-1/2}d(0,\check{\Gamma})
< r^{1/2}.
$$
By applying Proposition \ref{Prop expansion ED} and
Proposition \ref{Prop diff Y}, we get that on the
event $E_{\D}(r)\cap Q(r)$,
$\eta = O(r^{1/4})$, with deterministic bounds.

Let $(W_t)_{t\geq 0}$ be a standard one-dimensional Brownian motion starting from $0$.
Let $\mathcal{T}_{-\pi,\pi}^{(0)}$ be the last visit time of level $0$ by $W_t$ before exiting the interval 
$(-\pi,\pi)$.
According to \cite[Theorem 1]{ALS3},
$2\pi\ED(\check{\Gamma}_2,\check{\Gamma}_1)$
is distributed like 
$\mathcal{T}_{-\pi,\pi}^{(0)}$.
Moreover, according to 
\cite[Proposition 4]{KemppainenWerner16CLE},
$\ED(\check{\Gamma}_2,\check{\Gamma}_1)$
is independent from
$-\log\rho^{\pm}([\check{\Gamma}_2])+Y([\check{\Gamma}_2])$.
The tail of $\mathcal{T}_{-\pi,\pi}^{(0)}$ is given by
$$
\P(\mathcal{T}_{-\pi,\pi}^{(0)} > t)\sim 
\dfrac{8}{\pi^2}
e^{-\frac{1}{8}t};
$$
see \cite[Appendix A]{ALS3}.
Thus, we can write,
\begin{multline*}
\P(2\pi\ED(\check{\Gamma}_2,\check{\Gamma}_1)
>
\log r^{-1}
-
\log\rho^{\pm}([\check{\Gamma}_2])
+
Y([\check{\Gamma}_2]) +
O(r^{1/4})
) \sim
\\
\P(2\pi\ED(\check{\Gamma}_2,\check{\Gamma}_1)
>
\log r^{-1}
-
\log\rho^{\pm}([\check{\Gamma}_2])
+
Y([\check{\Gamma}_2]))
\sim
\dfrac{8}{\pi^2}
r^{1/8}
\E\Big[
\rho^{\pm}([\check{\Gamma}_2])^{1/8}
e^{-Y([\check{\Gamma}_2])/8}
\Big],
\end{multline*}
where we used Corollary \ref{cor:expmoments}. 
This concludes.
\end{proof}

\subsubsection{Proof of Proposition \ref{Prop diff Y}}

The first element is rather simple.

\begin{proof}[Proof of Proposition \ref{Prop diff Y} part A]
We will focus only on the second inequality which is somewhat more subtle.  
Let $z\in\Gamma$ and $w\in \Gamma_2$
such that
$$
\vert z\vert = d(0,\Gamma),
\qquad
\vert w\vert = d(0,\Gamma_2).
$$
Denote
$$
z' = f_{\Gamma_1}^{-1}(z),
\qquad
w' = f_{\Gamma_1}(w).
$$
By Proposition \ref{Prop de Branges}, taking $R_1 = R_{\rm univ}$
$$
\vert z' -z\vert\leq C_{\rm univ}\vert z\vert^2,
\qquad
\vert w'-w\vert
\leq C_{\rm univ}\vert w\vert^2.
$$
The first bound gives us
$$
d(0,\Gamma_2)
\leq 
d(0,\Gamma)
+
C_{\rm univ}
d(0,\Gamma)^2 .
$$
The second bound gives
$$
d(0,\Gamma_2)
+
C_{\rm univ}
d(0,\Gamma_2)^2
\geq
d(0,\Gamma),
$$
which implies
$$
d(0,\Gamma_2)
\geq
d(0,\Gamma)
-
C'_{\rm univ}
d(0,\Gamma)^2
$$
for appropriate $C'_{\rm univ}\geq C_{\rm univ}$, not depending on the particular choice of loops. 
We set $C_1 = C_{\rm univ}$ and $C_2 = C'_{\rm univ}$.
\end{proof}
~\\To prove the element B, we will first state an auxiliary lemma. Denote the rescaled loops
$$
\widetilde{\Gamma}=
R^{+}(\Gamma)^{-1}\Gamma,
\qquad
\widetilde{\Gamma}_2=
R^{+}(\Gamma_2)^{-1}\Gamma_2.
$$
In this way,
$\widetilde{\Gamma}$ and $\widetilde{\Gamma}_2$
are contained in $\overline{\D}$ and
intersect $\partial\D$.
Moreover,
\begin{equation}
\label{Eq Y Gamma tilde Gamma}
Y([\Gamma]) = Y([\widetilde{\Gamma}]),
\qquad
Y([\Gamma_2]) = Y([\widetilde{\Gamma}_2]) .
\end{equation}

\begin{lemma}
\label{Lem one but not other}
Let $(B_t)_{t\geq 0}$ denote a standard Brownian motion on
$\C$,
and denote by
$\P_z$ the probability associated to the initial condition
$B_0=z$.
Denote by $T_{2\partial\D}$ the first time this Brownian motion hits the circle $2\partial\D$.
There is a universal constant 
$C''_{\rm univ}>0$,
such that for every
$\Gamma_1,\Gamma_2$ as above,
and the associated $\Gamma = f_{\Gamma_1}(\Gamma_2)$,
$\widetilde{\Gamma}$ and $\widetilde{\Gamma}_2$,
and for every $\theta\in [0,2\pi)$,
we have
\begin{equation}
\label{Eq Beurlin bound 1}
\P_{e^{i\theta}}
(B([0,T_{2\partial\D}])
\text{ intersects } \widetilde{\Gamma}
\text{ but not } \widetilde{\Gamma}_2
)
\leq 
C''_{\rm univ} R^{+}(\Gamma)^{1/2},
\end{equation}
and
\begin{equation}
\label{Eq Beurlin bound 2}
\P_{e^{i\theta}}
(B([0,T_{2\partial\D}])
\text{ intersects } \widetilde{\Gamma}_2
\text{ but not } \widetilde{\Gamma}
)
\leq 
C''_{\rm univ} R^{+}(\Gamma)^{1/2}.
\end{equation}
\end{lemma}
Before proving this lemma, let us see how the Proposition \ref{Prop diff Y} part B follows.
\begin{proof}[Proof of Proposition \ref{Prop diff Y} part B]
We will use the equalities \eqref{Eq Y Gamma tilde Gamma}.
Let $(B_t)_{t\geq 0}$ denote a standard Brownian motion on
$\C$,
and denote by
$\E_z$ the expectation associated to the initial condition
$B_0=z$.
We will consider the crossings between the circles
$\partial\D$ and $2\partial\D$.
Set $T_{\partial\D}^{(0)}=0$.
Denote by
$T_{2\partial\D}^{(k)}$ the first time after
$T_{\partial\D}^{(k)}$ when
$B_t$ hits the circle $2\partial\D$,
and $T_{\partial\D}^{(k+1)}$ the first time after
$T_{2\partial\D}^{(k)}$  when
$B_t$ hits the circle $\partial\D$.
Then one can decompose $Y([\widetilde{\Gamma}])$
as follows:
$$
Y([\widetilde{\Gamma}])
=
\sum_{k=0}^{+\infty}
\int_{0}^{2\pi}
\E_{e^{i\alpha}}
\Big[
\1_{B([0,T_{2\partial\D}^{(k)}])\cap\widetilde{\Gamma}
=\emptyset}
G_{\C\setminus\overline{\D}}
(B_{T_{2\partial\D}^{(k)}},\infty)
\Big]
d\alpha .
$$
For this we count the number of crossings
between $\partial\D$ and $2\partial\D$
of a Brownian excursion in $\Ext(\widetilde{\Gamma})$
from $e^{i\alpha}$ and $\infty$.
A similar formula holds for $Y([\widetilde{\Gamma}_{2}])$.
But
$$
G_{\C\setminus\overline{\D}}
(B_{T_{2\partial\D}^{(k)}},\infty)
=
\dfrac{\log(2)}{2\pi} .
$$
Thus,
$$
Y([\widetilde{\Gamma}])
=
\dfrac{\log(2)}{2\pi}
\sum_{k=0}^{+\infty}
\int_{0}^{2\pi}
\P_{e^{i\alpha}}
\big(
B([0,T_{2\partial\D}^{(k)}])\cap\widetilde{\Gamma}
=\emptyset
\big)
d\alpha .
$$
It follows that
\begin{multline*}
\dfrac{2\pi}{\log(2)}
\vert Y([\widetilde{\Gamma}_2]) - Y([\widetilde{\Gamma}])\vert
\leq
\\
\sum_{\substack{k\geq 0\\0\leq j\leq k}}
\int_{0}^{2\pi}
\P_{e^{i\alpha}}
\big(
B([0,T_{2\partial\D}^{(j-1)}])\cap
(\widetilde{\Gamma}\cup\widetilde{\Gamma}_2)
=
\emptyset,
B([T_{\partial\D}^{(j)},T_{2\partial\D}^{(j)}])
\cap\widetilde{\Gamma}\neq \emptyset,
B([T_{\partial\D}^{(j)},T_{2\partial\D}^{(k)}])
\cap\widetilde{\Gamma}_2 = \emptyset
\big)
d\alpha
\\
+
\sum_{\substack{k\geq 0\\0\leq j\leq k}}
\int_{0}^{2\pi}
\P_{e^{i\alpha}}
\big(
B([0,T_{2\partial\D}^{(j-1)}])\cap
(\widetilde{\Gamma}\cup\widetilde{\Gamma}_2)
=
\emptyset,
B([T_{\partial\D}^{(j)},T_{2\partial\D}^{(j)}])
\cap\widetilde{\Gamma}_2\neq \emptyset,
B([T_{\partial\D}^{(j)},T_{2\partial\D}^{(k)}])
\cap\widetilde{\Gamma} = \emptyset
\big)
d\alpha .
\end{multline*}
Denote
$$
p_{\rm sep}
=
\P_{1}\big(B([0,T_{2\partial\D}^{(0)}])\cap\overline{\D}
\text{ disconnects } 0 \text{ from } \partial\D
\big)
>0.
$$
If $B([0,T_{2\partial\D}^{(0)}])\cap\overline{\D}$
disconnects $0$ from $\partial\D$,
then necessarily $B([0,T_{2\partial\D}^{(0)}])$
intersects both $\widetilde{\Gamma}$
and $\widetilde{\Gamma}_2$.
So, by combining with the Markov property and
Lemma \ref{Lem one but not other},
we get that
\begin{displaymath}
\vert Y([\widetilde{\Gamma}_2]) - Y([\widetilde{\Gamma}])\vert
\leq
\dfrac{\log(2)}{\pi}
\Big(
\sum_{k=0}^{+\infty}(k+1)
(1-p_{\rm sep})^{k}
\Big)
C''_{\rm univ} R^{+}(\Gamma)^{1/2}
.
\qedhere
\end{displaymath}
\end{proof}
It remains to prove the lemma.
\begin{proof}[Proof of Lemma \ref{Lem one but not other}]
We can assume that 
$R^{+}(\Gamma)\leq R_{\rm univ}$
and that $C_{\rm univ}R^{+}(\Gamma)\leq 1/4$.
The second condition will become clear further down the line.
Larger values of $R^{+}(\Gamma)$ can be absorbed into a larger constant $C''_{\rm univ}$.
Denote by $\tilde{f}_{\Gamma_1}$ the mapping
$\tilde{f}_{\Gamma_1}(z) = 
R^{+}(\Gamma)^{-1} f_{\Gamma_1}(R^{+}(\Gamma_2)z)$.
Denote by $T_{\widetilde{\Gamma}}$
the first hitting time of $\widetilde{\Gamma}$,
and by $T_{\widetilde{\Gamma}_2}$ the first hitting time of
$\widetilde{\Gamma}_2$.
Let $\theta\in [0,2\pi)$ and set $B_0 = e^{i\theta}$.

First we consider the event 
$\{T_{\widetilde{\Gamma}}<T_{2\partial\D}\}$.
Denote $Z=B_{T_{\widetilde{\Gamma}}}\in\widetilde{\Gamma}$ 
and 
$Z_2 = \tilde{f}_{\Gamma_1}^{-1}(Z)\in\widetilde{\Gamma}_2$.
We have that
$$
\vert Z_2 - Z\vert
\leq
\vert
R^{+}(\Gamma_2)^{-1}
-
R^{+}(\Gamma)^{-1}
\vert
\vert f_{\Gamma_1}^{-1}(R^{+}(\Gamma)Z)\vert
+
R^{+}(\Gamma)^{-1}
\vert f_{\Gamma_1}^{-1}(R^{+}(\Gamma)Z) - 
R^{+}(\Gamma)Z
\vert.
$$
Since $\vert R^{+}(\Gamma)Z\vert\leq R_{\rm univ}$,
by Proposition \ref{Prop de Branges}, we get that
$$
\vert f_{\Gamma_1}^{-1}(R^{+}(\Gamma)Z) - 
R^{+}(\Gamma)Z
\vert
\leq C_{\rm univ}R^{+}(\Gamma)^{2}\vert Z\vert^{2},
$$
and
$$
\vert f_{\Gamma_1}^{-1}(R^{+}(\Gamma)Z)\vert
\leq 
R^{+}(\Gamma)\vert Z\vert
+ C_{\rm univ}R^{+}(\Gamma)^{2}\vert Z\vert^{2}.
$$
By Proposition \ref{Prop diff Y} A,
$$
\vert
R^{+}(\Gamma_2)^{-1}
-
R^{+}(\Gamma)^{-1}
\vert
\leq 
R^{+}(\Gamma_2)^{-1}C_{\rm univ}R^{+}(\Gamma)
\leq
\dfrac{C_{\rm univ}}{1-C_{\rm univ}R^{+}(\Gamma)}.
$$
Since we assumed that $C_{\rm univ}R^{+}(\Gamma)\leq 1/4$,
in particular, $1-C_{\rm univ}R^{+}(\Gamma)>0$.
In this way,
\begin{eqnarray*}
\nonumber
\vert Z_2 - Z\vert
&\leq&
\dfrac{C_{\rm univ}}{1-C_{\rm univ}R^{+}(\Gamma)}
(R^{+}(\Gamma)\vert Z\vert
+ C_{\rm univ}R^{+}(\Gamma)^{2}\vert Z\vert^{2})
+
C_{\rm univ}R^{+}(\Gamma)\vert Z\vert^{2}
\\
\nonumber
&\leq &
C_{\rm univ}R^{+}(\Gamma)\vert Z\vert
\Big(
\dfrac{1}{1-C_{\rm univ}R^{+}(\Gamma)}
(1+C_{\rm univ}R^{+}(\Gamma)\vert Z\vert)
+\vert Z\vert
\Big)
\\
\nonumber
&=&
\dfrac{1 + \vert Z\vert}{1-C_{\rm univ}R^{+}(\Gamma)}
C_{\rm univ}R^{+}(\Gamma)\vert Z\vert
\\
\nonumber
&\leq&
\dfrac{8}{3}
C_{\rm univ}R^{+}(\Gamma)\vert Z\vert
\leq 
\dfrac{2}{3} \vert Z\vert.
\end{eqnarray*}
In particular, $\vert Z_2 - Z\vert<\vert Z\vert$.
This means that $\overline{\Int(\widetilde{\Gamma}_2)}$
connects $Z_2$ to the boundary of the disk of radius
$\vert Z\vert$ around $Z$.
In this configuration we can apply the Beurling estimate;
see \cite[Section 3.8]{LawC} and \cite{Oksendal83Beurling}.
There is a universal constant $C_{\rm Beurling}>0$,
such that a.s.,
\begin{multline*}
\P_{e^{i\theta}}
(
B_{T_{\widetilde{\Gamma}}+t}
\text{ exits } \D(Z,\vert Z\vert)
\text{ before hitting } \widetilde{\Gamma}_2
\vert
T_{\widetilde{\Gamma}},
(B_{t})_{0\leq t\leq T_{\widetilde{\Gamma}}}
)
\\
\leq
C_{\rm Beurling}
\dfrac{\vert Z_2 - Z\vert^{1/2}}{\vert Z\vert^{1/2}}
\leq
\dfrac{2\sqrt{2}}{\sqrt{3}}
C_{\rm Beurling}
C_{\rm univ}^{1/2}
R^{+}(\Gamma)^{1/2}
.
\end{multline*}
This shows \eqref{Eq Beurlin bound 1}.
The proof of \eqref{Eq Beurlin bound 2} is similar.
\end{proof}

\subsection{Proof of Proposition \ref{Prop local asymp expect CR 1 8}}
We start by proving a weaker-looking statement with extra conditionings. Let $r_0\in (0,1]$ and $\mathfrak{C}_{r_0}$ be as above.

For $j\in\{1,2\}$, denote by
$
\widehat{D}_{r_0}(z_j)
$
the connected component of $z_j$ in
$
D\setminus
\overline{
\bigcup_{\cC\in \mathfrak{C}_{r_0}}\cC
}
.
$
By construction, $\widehat{D}_{r_0}(z_j)$ is an open connected and simply-connected domain,
and thus conformally equivalent to a disk.
Moreover, $\widehat{D}_{r_0}(z_j)\subset \D(z_j,r_0)$.
By definition, no excursion from $\Xi^{v^2/2}_{D}$
intersects $\widehat{D}_{r_0}(z_j)$,
and no loop in $\cL_{D}$
intersects both $\widehat{D}_{r_0}(z_j)$ and its complement. 
Denote by
$E_{r_0,\varepsilon_1,\varepsilon_2}$ the event
\begin{equation}
\label{Eq def E r0 eps1 eps2}
E_{r_0,\varepsilon_1,\varepsilon_2}
=
\{
\text{No cluster of }
\mathfrak{C}_{r_0}
\text{ disconnects } \D(z_1,\varepsilon_1) 
\text{ from } \D(z_2,\varepsilon_2)
\}.
\end{equation}
This event is of course measurable with respect to
$\mathfrak{C}_{r_0}$. The key intermediate result then is as follows. Below we denote $r(\eps) = \eps^{1/2}$.
\begin{lemma}
\label{Prop local asymp expect CR 1 8 ver 1}
As $\max(\varepsilon_1,\varepsilon_2)\to 0$, we have that
\begin{multline*}
\P(E_{\varepsilon_1,\varepsilon_2}^v)
=
\varepsilon_{1}^{1/8}
\varepsilon_{2}^{1/8}
(\CSLE)^2
\\
\times
\E
\big[
\CR(z_1,\widehat{D}_{r_0}(z_1))^{-1/8}
\CR(z_2,\widehat{D}_{r_0}(z_2))^{-1/8}
\1_{E_{r_0,\varepsilon_1,\varepsilon_2}}
\1_{\{\D(z_j,r(\varepsilon_j))
\subset
\widehat{D}_{r_0}(z_j), j\in\{1,2\}\}}
\big]
(1+o(1))
.
\end{multline*}
\end{lemma}

The Proposition \ref{Prop local asymp expect CR 1 8} results directly from this by obtaining the following limit. 
\begin{lemma}
\label{Prop ident 2pt corr}
We have
\begin{multline*}
\lim_{\max(\varepsilon_1,\varepsilon_2)\to 0}
\E
\big[
\CR(z_1,\widehat{D}_{r_0}(z_1))^{-1/8}
\CR(z_2,\widehat{D}_{r_0}(z_2))^{-1/8}
\1_{E_{r_0,\varepsilon_1,\varepsilon_2}}
\1_{\{\D(z_j,r(\varepsilon_j))
\subset
\widehat{D}_{r_0}(z_j), j\in\{1,2\}\}}
\big]
\\
=\E
\big[
\CR(z_1,\widehat{D}_{r_0}(z_1))^{-1/8}
\CR(z_2,\widehat{D}_{r_0}(z_2))^{-1/8}
\1_{E_{r_0}}
\big]
.
\end{multline*}
\end{lemma}
This lemma in turn follows from the Dominated Convergence Theorem, as soon as sufficient integrability is provided. This integrability comes from Lemma \ref{Lem finite exp CR 1 8}, proved in the next section. We now prove Lemma \ref{Prop local asymp expect CR 1 8 ver 1} first.
\subsubsection{Proof of Lemma \ref{Prop local asymp expect CR 1 8 ver 1}}

Denote by
$\cL_{r_0,j}$
the subset of loops in $\cL_{D}$ that are contained inside the domain 
$\widehat{D}_{r_0}(z_j)$.
The following renewal property is standard.
\begin{lemma}
\label{Lem cond C frak 1}
 Conditionally on  $\mathfrak{C}_{r_0}$,
 the two families $\cL_{r_0,1}$ and $\cL_{r_0,2}$
 are independent from each other,
 and are distributed as Poisson point processes with intensities
 $\frac{1}{2}\mu^{\rm loop}_{\widehat{D}_{r_0}(z_1)}$,
 respectively 
 $\frac{1}{2}\mu^{\rm loop}_{\widehat{D}_{r_0}(z_2)}$.
\end{lemma}
For $j\in\{1,2\}$,
 denote by
 $E_{r_0,\varepsilon_j}^{(j)}$ the event
 $$
E_{r_0,\varepsilon_j}^{(j)} =
\{
\D(z_j,\varepsilon_j)
\subset
\widehat{D}_{r_0}(z_j)
\text{ and no cluster of }
\cL_{r_0,j}
\text{ disconnects } \D(z_j,\varepsilon_j) 
\text{ from } \partial \widehat{D}_{r_0}(z_j)
\}.
 $$
Using this, we can write the following identity of events.
    $$
 E_{\varepsilon_1,\varepsilon_2}^v
\cap
\{\D(z_1,\varepsilon_1)
\subset
\widehat{D}_{r_0}(z_1)\}
\cap
\{\D(z_2,\varepsilon_2)
\subset
\widehat{D}_{r_0}(z_2)\}
=
E_{r_0,\varepsilon_1,\varepsilon_2}
\cap
E_{r_0,\varepsilon_1}^{(1)}
\cap
E_{r_0,\varepsilon_2}^{(2)}.
    $$
Recall also the notation
$$
r(\varepsilon) = \varepsilon^{1/2}.
$$
We will consider the more constraining event
\begin{equation}
\label{Eq loc gasket eps log eps}
 E_{\varepsilon_1,\varepsilon_2}^v
\cap
\{\D(z_1,r(\varepsilon_1))
\subset
\widehat{D}_{r_0}(z_1)\}
\cap
\{\D(z_2,r(\varepsilon_2))
\subset
\widehat{D}_{r_0}(z_2)\},
\end{equation}
and we will argue that
$$
\P( E_{\varepsilon_1,\varepsilon_2}^v
\cap
\{\D(z_1,r(\varepsilon_1))
\subset
\widehat{D}_{r_0}(z_1)\}
\cap
\{\D(z_2,r(\varepsilon_2))
\subset
\widehat{D}_{r_0}(z_2)\})
\sim 
\P(E_{\varepsilon_1,\varepsilon_2}^v).
$$

\begin{lemma}
\label{Lem equiv with FKG}
    Let $j\in\{1,2\}$. The following bound holds:
   $$
       \P( E_{\varepsilon_1,\varepsilon_2}^v
\cap
\{\D(z_j,r(\varepsilon_j))
\not\subset
\widehat{D}_{r_0}(z_j)\}
)
\leq
\P(E_{\varepsilon_1,\varepsilon_2}^v)
\P(\D(z_j,r(\varepsilon_j))
\not\subset
\widehat{D}_{r_0}(z_j)).
$$
In particular, as
$\max(\varepsilon_1,\varepsilon_2)\to 0$,
$$
\P( E_{\varepsilon_1,\varepsilon_2}^v
\cap
\{\D(z_j,r(\varepsilon_j))
\not\subset
\widehat{D}_{r_0}(z_j)\}
)
=
o(\P(E_{\varepsilon_1,\varepsilon_2}^v)).
$$
\end{lemma}
\begin{proof}
We write
\begin{multline*}
\P( E_{\varepsilon_1,\varepsilon_2}^v
\cap
\{\D(z_j,r(\varepsilon_j))
\not\subset
\widehat{D}_{r_0}(z_j)\}
)
=
\\
\P(\D(z_j,r(\varepsilon_j))
\not\subset
\widehat{D}_{r_0}(z_j))
-
\P( (E_{\varepsilon_1,\varepsilon_2}^v)^{c}
\cap
\{\D(z_j,r(\varepsilon_j))
\not\subset
\widehat{D}_{r_0}(z_j)\}
).
\end{multline*}
But the events $(E_{\varepsilon_1,\varepsilon_2}^v)^{c}$
and $\{\D(z_j,r(\varepsilon_j))
\not\subset
\widehat{D}_{r_0}(z_j)\}$
are both increasing for the Poisson point process
$\cL_{D}
\cup
\Xi^{v^2/2}_{D}$.
By the FKG property satisfied by general Poisson point processes \cite[Lemma 2.1]{Janson1984FKGPPP},
$$
\P( (E_{\varepsilon_1,\varepsilon_2}^v)^{c}
\cap
\{\D(z_j,r(\varepsilon_j))
\not\subset
\widehat{D}_{r_0}(z_j)\}
)
\geq
\P( (E_{\varepsilon_1,\varepsilon_2}^v)^{c})
\P(\D(z_j,r(\varepsilon_j))
\not\subset
\widehat{D}_{r_0}(z_j)).
$$
This gives the desired inequality.
Further, we conclude by using the fact that
\begin{displaymath}
\lim_{\varepsilon_j\to 0}
\P(\D(z_j,r(\varepsilon_j))
\not\subset
\widehat{D}_{r_0}(z_j))
=0.
\qedhere
\end{displaymath}
\end{proof}

Next we work with the event \eqref{Eq loc gasket eps log eps},
which is the same as
$$
E_{r_0,\varepsilon_1,\varepsilon_2}
\cap
E_{r_0,\varepsilon_1}^{(1)}
\cap
\{\D(z_1,r(\varepsilon_1))
\subset
\widehat{D}_{r_0}(z_1)\}
\cap
E_{r_0,\varepsilon_2}^{(2)}
\cap
\{\D(z_2,r(\varepsilon_2))
\subset
\widehat{D}_{r_0}(z_2)\}.
$$
Denote by
$f_{j,r_0}$ an uniformization mapping from
$\widehat{D}_{r_0}(z_j)$ to $\D$,
with $f_{j,r_0}(z_j)=0$.
We can for instance additionally impose
$f_{j,r_0}'(z_j)\in (0,+\infty)$ to have it uniquely defined,
but this is of no importance.
Denote by $\cL_{\D}'$ a Brownian loop soup in $\D$ with intensity
$\frac{1}{2}\mu^{\rm loop}_{\D}$,
and independent from $\cL_{D}$.

\begin{lemma}
\label{Lem cond prob E j r 0 eps j}
    Let $j\in\{1,2\}$.
    Then the following equality of conditional probabilities holds:
\begin{multline*}
\P(E_{r_0,\varepsilon_j}^{(j)}
\vert \mathfrak{C}_{r_0}, 
\D(z_j,r(\varepsilon_j))
\subset
\widehat{D}_{r_0}(z_j))
=
\\
\P(
\text{No cluster of }
\cL_{\D}'
\text{ surrounds }
f_{j,r_0}(\D(z_j,\varepsilon_j))
\vert
\mathfrak{C}_{r_0}, 
\D(z_j,r(\varepsilon_j))
\subset
\widehat{D}_{r_0}(z_j)
).
\end{multline*}
\end{lemma}
\begin{proof}
    This follows from Lemma \ref{Lem cond C frak 1} and the conformal invariance of the
    Brownian loop soup.
\end{proof}

Denote now
$$
R^{-}_{j,r_0}(\varepsilon_j)
=
d(0,\partial f_{j,r_0}(\D(z_j,\varepsilon_j))),
\qquad
R^{+}_{j,r_0}(\varepsilon_j)
=
\max\{
\vert w\vert :
w\in\partial f_{j,r_0}(\D(z_j,\varepsilon_j))
\}.
$$
We have the inclusions
$$
\D(0,R^{-}_{j,r_0}(\varepsilon_j))
\subset
f_{j,r_0}(\D(z_j,\varepsilon_j))
\subset
\D(0,R^{+}_{j,r_0}(\varepsilon_j)).
$$
Therefore,
\begin{multline*}
\P(
E_{\D}(R^{-}_{j,r_0}(\varepsilon_j))
\vert
R^{-}_{j,r_0}(\varepsilon_j), 
\D(z_j,r(\varepsilon_j))
\subset
\widehat{D}_{r_0}(z_j)
)
\\
\leq
\P(
\text{No cluster of }
\cL_{\D}'
\text{ surrounds }
f_{j,r_0}(\D(z_j,\varepsilon_j))
\vert
\mathfrak{C}_{r_0}, 
\D(z_j,r(\varepsilon_j))
\subset
\widehat{D}_{r_0}(z_j)
)
\\
\leq
\P(
E_{\D}(R^{+}_{j,r_0}(\varepsilon_j))
\vert
R^{+}_{j,r_0}(\varepsilon_j), 
\D(z_j,r(\varepsilon_j))
\subset
\widehat{D}_{r_0}(z_j)
),
\end{multline*}
where, for $r\in (0,1)$, $E_{\D}(r)$ denotes the event
$$
E_{\D}(r)=
\{
\text{No cluster of }
\cL_{\D}'
\text{ surrounds }
\D(0,r)
\}.
$$
Further, as $\varepsilon_j\to 0$,
$R^{-}_{j,r_0}(\varepsilon_j)$ and $R^{+}_{j,r_0}(\varepsilon_j)$
both behave like
$$
\varepsilon_j
\vert f_{j,r_0}'(z_j)\vert
=
\varepsilon_j
\CR(z_j,\widehat{D}_{r_0}(z_j))^{-1}.
$$
Next we estimate the difference from this leading term.

\begin{lemma}
\label{Lem distortion R pm}
Let $j\in\{1,2\}$.
Let $\varepsilon_j\in (0,1/4)$,
so that $\varepsilon_j<r(\varepsilon_j)/2$.
On the event 
$\{\D(z_j,r(\varepsilon_j))
\subset
\widehat{D}_{r_0}(z_j)\}$,
both  $R^{-}_{j,r_0}(\varepsilon_j)$ and $R^{+}_{j,r_0}(\varepsilon_j)$
satisfy the following deterministic bound:
$$
\dfrac{\varepsilon_j}{\CR(z_j,\widehat{D}_{r_0}(z_j))}
\exp(-2\log(2)\varepsilon_{j}^{1/2})
\leq
R^{\pm}_{j,r_0}(\varepsilon_j)\leq
\dfrac{\varepsilon_j}{\CR(z_j,\widehat{D}_{r_0}(z_j))}
\exp(2\log(2)\varepsilon_{j}^{1/2}).
$$
\end{lemma}
\begin{proof}
Let $z\in \D(z_j,\varepsilon_j)$.
By conformal invariance of the Green's function,
$$
G_{\widehat{D}_{r_0}(z_j)}(z_j,z)
=
G_{\D}(0, f_{j,r_0}(z))
=
\dfrac{1}{2\pi} \vert\log \vert f_{j,r_0}(z)\vert \vert.
$$
Further, we write
$$
G_{\widehat{D}_{r_0}(z_j)}(z_j,z)
=
\dfrac{1}{2\pi} \log \vert z- z_j\vert^{-1}
+
h_{\widehat{D}_{r_0}(z_j)}(z_j,z),
$$
where the function $h_{\widehat{D}_{r_0}(z_j)}(z,w)$
is harmonic on $\widehat{D}_{r_0}(z_j)^{2}$
with respect to both variables.
For $z=z_j$,
$$
h_{\widehat{D}_{r_0}(z_j)}(z_j,z_j)
=
\dfrac{1}{2\pi}
\log \CR(z_j,\widehat{D}_{r_0}(z_j)) .
$$
Thus,
$$
G_{\widehat{D}_{r_0}(z_j)}(z_j,z)
=
\dfrac{1}{2\pi}
\log\Big(
\dfrac{\CR(z_j,\widehat{D}_{r_0}(z_j))}{\vert z- z_j\vert}
\Big)
+
(h_{\widehat{D}_{r_0}(z_j)}(z_j,z)
-
h_{\widehat{D}_{r_0}(z_j)}(z_j,z_j)
).
$$
Next we bound the difference
$$
h_{\widehat{D}_{r_0}(z_j)}(z_j,z)
-
h_{\widehat{D}_{r_0}(z_j)}(z_j,z_j).
$$
Since
$$
\vert z-z_j\vert<\varepsilon_j<\dfrac{1}{2}r(\varepsilon_j)
\leq \dfrac{1}{2} d(z_j,\partial \widehat{D}_{r_0}(z_j)),
$$
we can apply \cite[Proposition 3.2]{lupuWick},
according to which
$$
\vert
h_{\widehat{D}_{r_0}(z_j)}(z_j,z)
-
h_{\widehat{D}_{r_0}(z_j)}(z_j,z_j)
\vert
\leq 
\dfrac{\log(2)}{\pi} \dfrac{\vert z-z_j\vert}{d(z_j,\partial \widehat{D}_{r_0}(z_j))}
\leq
\dfrac{\log(2)}{\pi}\dfrac{\varepsilon_j}{r(\varepsilon_j)}
=
\dfrac{\log(2)}{\pi}\varepsilon_{j}^{1/2}.
$$
So we get that
$$
\dfrac{\vert z-z_j\vert}{\CR(z_j,\widehat{D}_{r_0}(z_j))}
\exp(-2\log(2)\varepsilon_{j}^{1/2})
\leq
\vert f_{j,r_0}(z)\vert 
\leq
\dfrac{\vert z-z_j\vert}{\CR(z_j,\widehat{D}_{r_0}(z_j))}
\exp(2\log(2)\varepsilon_{j}^{1/2}).
$$
This implies the bound for $R^{-}_{j,r_0}(\varepsilon_j)$ and $R^{+}_{j,r_0}(\varepsilon_j)$.  
\end{proof}
We now collect these lemmas and the input of the one-point function to conclude.
\begin{proof}[Proof of Lemma \ref{Prop local asymp expect CR 1 8 ver 1}]
Denote by $p(r)$ the probability
$$
p(r) = \P(E_{\D}(r)).
$$
Define $b(r)$ through the relation
$$
p(r)
=
\CSLE~
r^{1/8}(1+b(r)).
$$
By Proposition \ref{Prop asymp 1 pt ED},
$b(r)\to 0$ as $r\to 0$.
Denote
$$
\widetilde{R}^{-}_{j,r_0}(\varepsilon_j)
=
\dfrac{\varepsilon_j}{\CR(z_j,\widehat{D}_{r_0}(z_j))}
\exp(-2\log(2)\varepsilon_{j}^{1/2}),
\qquad
\widetilde{R}^{+}_{j,r_0}(\varepsilon_j)
=
\dfrac{\varepsilon_j}{\CR(z_j,\widehat{D}_{r_0}(z_j))}
\exp(2\log(2)\varepsilon_{j}^{1/2}).
$$

Next we apply Lemmas \ref{Lem cond prob E j r 0 eps j} and
\ref{Lem distortion R pm},
together with the conditional independence of
$E_{r_0,\varepsilon_1}^{(1)}$ and 
$E_{r_0,\varepsilon_2}^{(2)}$
given $\mathfrak{C}_{r_0}$.
We obtain that the probability
\begin{equation}
\label{Eq prop many cond}
\P( E_{\varepsilon_1,\varepsilon_2}^v
\cap
\{\D(z_1,r(\varepsilon_1))
\subset
\widehat{D}_{r_0}(z_1)\}
\cap
\{\D(z_2,r(\varepsilon_2))
\subset
\widehat{D}_{r_0}(z_2)\})
\end{equation}
is bounded from below by
$$
\E
\big[
p(\widetilde{R}^{-}_{1,r_0}(\varepsilon_1))
p(\widetilde{R}^{-}_{2,r_0}(\varepsilon_2))
\1_{E_{r_0,\varepsilon_1,\varepsilon_2}}
\1_{\{\D(z_j,r(\varepsilon_j))
\subset
\widehat{D}_{r_0}(z_j), j\in\{1,2\}\}}
\big],
$$
and from above by
\begin{equation}
\label{Eq upper tilde R +}
\E
\big[
p(\widetilde{R}^{+}_{1,r_0}(\varepsilon_1))
p(\widetilde{R}^{+}_{2,r_0}(\varepsilon_2))
\1_{E_{r_0,\varepsilon_1,\varepsilon_2}}
\1_{\{\D(z_j,r(\varepsilon_j))
\subset
\widehat{D}_{r_0}(z_j), j\in\{1,2\}\}}
\big].
\end{equation}
Let us focus on the upper bound \eqref{Eq upper tilde R +},
the lower bound being similar.
We get that \eqref{Eq upper tilde R +} equals
\begin{multline*}
\varepsilon_{1}^{1/8}
\varepsilon_{2}^{1/8}
(\CSLE)^2
\exp\big(\log(2)(\varepsilon_{1}^{1/2}+\varepsilon_{2}^{1/2})/4\big)
\\
\times
\E
\Big[
\Big(
\prod_{j\in \{1,2\}}
\CR(z_j,\widehat{D}_{r_0}(z_j))^{-1/8}
(1+b(\widetilde{R}^{+}_{j,r_0}(\varepsilon_j)))
\Big)
\1_{E_{r_0,\varepsilon_1,\varepsilon_2}}
\1_{\{\D(z_j,r(\varepsilon_j))
\subset
\widehat{D}_{r_0}(z_j), j\in\{1,2\}\}}
\Big]
.
\end{multline*}
Of course,
$$
\exp\big(\log(2)(\varepsilon_{1}^{1/2}+\varepsilon_{2}^{1/2})/4\big)
= 1 + o(1).
$$
Further, on the event
$\{\D(z_j,r(\varepsilon_j))
\subset
\widehat{D}_{r_0}(z_j)\}$,
$$
\widetilde{R}^{+}_{j,r_0}(\varepsilon_j)
\leq
\varepsilon_{j}^{1/2}
\exp(2\log(2)\varepsilon_{j}^{1/2}).
$$
This provides a deterministic bound for
$b(\widetilde{R}^{+}_{j,r_0}(\varepsilon_j))$,
which tends to $0$ as $\varepsilon_j\to 0$.
So we get the desired equivalent for 
\eqref{Eq prop many cond}.
We then use Lemma \ref{Lem equiv with FKG},
which ensures that this is also equivalent to
$\P(E_{\varepsilon_1,\varepsilon_2}^v)$.
\end{proof}

\subsection{Proof of Proposition \ref{Prop asymp 2 pt cut disks CR 1 8}}

We will to a large extent mimic the approach of the previous section. 
As mentioned, we consider 
$\cL_{D_{\varepsilon_1,\varepsilon_2}(z_1,z_2)}$
as a subfamily of
$\cL_{D}$ and $\Xi^{v^2/2}_{D, D_{\varepsilon_1,\varepsilon_2}(z_1,z_2)}$ as a subfamily of
$\Xi^{v^2/2}_{D}$. 

Recall
the event
$$E_{\varepsilon_1,\varepsilon_2}^{v,\rm{holes}} 
= \{\text{No cluster of }
\cL_{D_{\varepsilon_1,\varepsilon_2}(z_1,z_2)}
\cup
\Xi^{v^2/2}_{D, D_{\varepsilon_1,\varepsilon_2}(z_1,z_2)}\text{ disconnects } z_1 \text{ from }z_2\},$$
and the notation $\widehat{D}_{r_0}(z_j)$
for the connected component of $z_j$ in
$D\setminus
\overline{\bigcup_{\cC\in \mathfrak{C}_{r_0}}\cC}$.

We stress that here $\mathfrak{C}_{r_0}$ is constructed using $\cL_{D}$ and $\Xi^{v^2/2}_{D}$.

\begin{lemma}
As $\max(\varepsilon_1,\varepsilon_2)\to 0$,
$$
\P(E_{\varepsilon_1,\varepsilon_2}^{v,\rm{holes}}
\cap
\{\D(z_1,r(\varepsilon_1))
\subset
\widehat{D}_{r_0}(z_1)\}
\cap
\{\D(z_2,r(\varepsilon_2))
\subset
\widehat{D}_{r_0}(z_2)\})
\sim
\P(E_{\varepsilon_1,\varepsilon_2}^{v,\rm{holes}}).
$$
\end{lemma}
\begin{proof}
This is similar to Lemma \ref{Lem equiv with FKG},
and follows by applying the FKG inequality.
\end{proof}

For $j\in\{1,2\}$, denote by 
$\cL_{r_0,j,\varepsilon_j}$
the subfamily of $\cL_{D}$
made of loops that stay in
$\widehat{D}_{r_0}(z_j)\setminus 
\overline{\D(z_j,\varepsilon_j)}$.
Conditionally on $\mathfrak{C}_{r_0}$,
$\cL_{r_0,1,\varepsilon_1}$ and 
$\cL_{r_0,2,\varepsilon_2}$
are independent and conditionally distributed as
Brownian loop soups in the respective annular domains.
Let $E_{r_0,\varepsilon_j}^{(j),\rm{holes}}$ be the event
$$
E_{r_0,\varepsilon_j}^{(j),\rm{holes}} =
\{
\D(z_j,\varepsilon_j)
\subset
\widehat{D}_{r_0}(z_j)
\text{ and no cluster of }
\cL_{r_0,j,\varepsilon_j}
\text{ disconnects } \D(z_j,\varepsilon_j) 
\text{ from } \partial \widehat{D}_{r_0}(z_j)
\}.
 $$
Then the following equality between the events holds:
\begin{multline*}
E_{\varepsilon_1,\varepsilon_2}^{v,\rm{holes}}
\cap
\{\D(z_1,r(\varepsilon_1))
\subset
\widehat{D}_{r_0}(z_1)\}
\cap
\{\D(z_2,r(\varepsilon_2))
\subset
\widehat{D}_{r_0}(z_2)\}
\\
=
E_{r_0,\varepsilon_1,\varepsilon_2}
\cap
E_{r_0,\varepsilon_1}^{(1),\rm{holes}}
\cap
\{\D(z_1,r(\varepsilon_1))
\subset
\widehat{D}_{r_0}(z_1)\}
\cap
E_{r_0,\varepsilon_2}^{(2),\rm{holes}}
\cap
\{\D(z_2,r(\varepsilon_2))
\subset
\widehat{D}_{r_0}(z_2)\},
\end{multline*}
where $E_{r_0,\varepsilon_1,\varepsilon_2}$
is given by \eqref{Eq def E r0 eps1 eps2}.
By using the conditional independence of
$E_{r_0,\varepsilon_1}^{(1),\rm{holes}}$
and $E_{r_0,\varepsilon_2}^{(2),\rm{holes}}$
given $\mathfrak{C}_{r_0}$,
and the asymptotics of disconnection probabilities
in annular domains
(here $\widehat{D}_{r_0}(z_j)\setminus 
\overline{\D(z_j,\varepsilon_j)}$,
$j\in\{1,2\}$),
we get the following analogue of 
Lemma \ref{Prop local asymp expect CR 1 8 ver 1}.

\begin{lemma}
\label{Lem local asymp cut disks expect CR 1 8}
As $\max(\varepsilon_1,\varepsilon_2)\to 0$, we have that
\begin{multline*}
\P(E_{\varepsilon_1,\varepsilon_2}^{v,\rm{holes}})
\sim
\P(E_{r_0,\varepsilon_1,\varepsilon_2}
\cap
E_{r_0,\varepsilon_1}^{(1),\rm{holes}}
\cap
\{\D(z_1,r(\varepsilon_1))
\subset
\widehat{D}_{r_0}(z_1)\}
\cap
E_{r_0,\varepsilon_2}^{(2),\rm{holes}}
\cap
\{\D(z_2,r(\varepsilon_2))
\subset
\widehat{D}_{r_0}(z_2)\})
\\=
\dfrac{2}{\pi}
\varepsilon_{1}^{1/8}
\varepsilon_{2}^{1/8}
(1+o(1))
\\
\times
\E
\Big[
\Big(
\prod_{j\in \{1,2\}}
\log\big(
\varepsilon_j^{-1}\CR(z_j,\widehat{D}_{r_0}(z_j))
\big)^{1/2}
\CR(z_j,\widehat{D}_{r_0}(z_j))^{-1/8}
\Big)
\1_{E_{r_0,\varepsilon_1,\varepsilon_2}}
\1_{\{\D(z_j,r(\varepsilon_j))
\subset
\widehat{D}_{r_0}(z_j), j\in\{1,2\}\}}
\Big]
.
\end{multline*}
\end{lemma}

We will further simplify the above asymptotic.
Recall the event $E_{r_0}$ given by \eqref{Eq def E r0}.

\begin{lemma}
\label{Lem finite exp CR 1 8}
We have   
$$
\E
\big[
\CR(z_1,\widehat{D}_{r_0}(z_1))^{-1/8}
\CR(z_2,\widehat{D}_{r_0}(z_2))^{-1/8}
\1_{E_{r_0}}
\big]
<+\infty.
$$
\end{lemma}
\begin{proof}
On the event 
$\{\D(z_j,r(\varepsilon_j))
\subset \widehat{D}_{r_0}(z_j)\}$,
we have 
$\CR(z_j,\widehat{D}_{r_0}(z_j))\geq \varepsilon_j^{1/2}$.
So, by Lemma \ref{Lem local asymp cut disks expect CR 1 8},
we get that
\begin{multline*}
\P(E_{\varepsilon_1,\varepsilon_2}^{v,\rm{holes}})
\geq
\dfrac{1}{\pi}
\varepsilon_{1}^{1/8}
\varepsilon_{2}^{1/8}
\log(\varepsilon_1^{-1})^{1/2}
\log(\varepsilon_2^{-1})^{1/2}
\\
\times
\E
\Big[
\Big(
\prod_{j\in \{1,2\}}
\CR(z_j,\widehat{D}_{r_0}(z_j))^{-1/8}
\Big)
\1_{E_{r_0,\varepsilon_1,\varepsilon_2}}
\1_{\{\D(z_j,r(\varepsilon_j))
\subset
\widehat{D}_{r_0}(z_j), j\in\{1,2\}\}}
\Big]
(1+o(1))
.
\end{multline*}
But by Theorem \ref{thm:twistbc},
we know that
$$
\P(E_{\varepsilon_1,\varepsilon_2}^{v,\rm{holes}})
=
\dfrac{2}{\pi}
\varepsilon_{1}^{1/8}
\varepsilon_{2}^{1/8}
\log(\varepsilon_1^{-1})^{1/2}
\log(\varepsilon_2^{-1})^{1/2}
A(z_1,z_2)
(1+o(1))
$$
for some non-zero $A(z_1, z_2)$.
It follows that
$$
\limsup_{\max(\varepsilon_1,\varepsilon_2)\to 0}
\E
\Big[
\Big(
\prod_{j\in \{1,2\}}
\CR(z_j,\widehat{D}_{r_0}(z_j))^{-1/8}
\Big)
\1_{E_{r_0,\varepsilon_1,\varepsilon_2}}
\1_{\{\D(z_j,r(\varepsilon_j))
\subset
\widehat{D}_{r_0}(z_j), j\in\{1,2\}\}}
\Big]
\leq 
2A(z_1, z_2).
$$
But by monotone convergence this $\limsup$ is also
\begin{displaymath}
\E
\big[
\CR(z_1,\widehat{D}_{r_0}(z_1))^{-1/8}
\CR(z_2,\widehat{D}_{r_0}(z_2))^{-1/8}
\1_{E_{r_0}}
\big].
\qedhere
\end{displaymath}
\end{proof}
We can now conclude.
\begin{proof}[Proof of Proposition \ref{Prop asymp 2 pt cut disks CR 1 8}]
Almost surely, as $\max(\varepsilon_1,\varepsilon_2)\to 0$,
the quantity
\begin{equation}
\label{Eq CR log CR epsj}
\Big(
\prod_{j\in \{1,2\}}
\dfrac{
\log\big(
\varepsilon_j^{-1}\CR(z_j,\widehat{D}_{r_0}(z_j))
\big)^{1/2}}
{\log(\varepsilon_j^{-1})^{1/2}}
\CR(z_j,\widehat{D}_{r_0}(z_j))^{-1/8}
\Big)
\1_{E_{r_0,\varepsilon_1,\varepsilon_2}}
\1_{\{\D(z_j,r(\varepsilon_j))
\subset
\widehat{D}_{r_0}(z_j), j\in\{1,2\}\}}
\end{equation}
converges to
$$
\CR(z_1,\widehat{D}_{r_0}(z_1))^{-1/8}
\CR(z_2,\widehat{D}_{r_0}(z_2))^{-1/8}
\1_{E_{r_0}}.
$$
Moreover,
for $\varepsilon_j\in (0,r_0\wedge(1/2))$,
\eqref{Eq CR log CR epsj} is dominated by
$$
(1+(\log r_0)_{+}/\log(2))
\CR(z_1,\widehat{D}_{r_0}(z_1))^{-1/8}
\CR(z_2,\widehat{D}_{r_0}(z_2))^{-1/8}
\1_{E_{r_0}}.
$$
Since
\begin{equation}
\label{Eq E CR 1 8}
\E
\big[
\CR(z_1,\widehat{D}_{r_0}(z_1))^{-1/8}
\CR(z_2,\widehat{D}_{r_0}(z_2))^{-1/8}
\1_{E_{r_0}}
\big]
<+\infty,
\end{equation}
we get by dominated convergence that the expectation of
\eqref{Eq CR log CR epsj} converges
to \eqref{Eq E CR 1 8}.
Then we conclude by 
Lemma \ref{Lem local asymp cut disks expect CR 1 8}.
\end{proof}

\section{Renormalised two-point correlation functions for CLE$_4$}\label{sec:together}

\subsection{Nested CLE$_4$ gaskets}
\label{Subsec Nested CLE4 gaskets}

In this section we will prove Theorem \ref{thm:main1}. We will cut this proof into two ingredients. The main ingredient is the following resummation theorem for the probability that either an odd, even, or any CLE$_4$ gasket comes near both $z_1$ and $z_2$.

\begin{prop}\label{prop:2ptNestedCLE}
For $v \in \R$ and $\eps_1, \eps_2 > 0$ small, we let $E^v_{\eps_1,\eps_2}$ denote the event that for a GFF in a simply-connected domain $D$ with constant boundary value $v$ there exists a nested CLE$_4$ gasket of height $0$
that intersects both 
$\D(z_1,\eps_1)$ and $\D(z_2, \eps_2)$.

We then have that 
\begin{itemize}
\item $$\P(E^{\rm nest}_{\eps_1,\eps_2}(z_1, z_2)) = \sum_{n \in \Z} \P(E^{2\lambda n}_{\eps_1,\eps_2})  
+O((\eps_1\eps_2)^{1/6}),$$
where $E^{\rm nest}_{\eps_1,\eps_2}(z_1,z_2)$ is the event that there is some gasket of nested CLE$_4$ that intersects both $\D(z_1,\eps_1)$ and 
$\D(z_2,\eps_2)$.

\item $$\P(E^{\rm odd}_{\eps_1,\eps_2}(z_1, z_2))
=\sum_{n \in 2\Z} \P(E^{2\lambda n}_{\eps_1,\eps_2})  
+O((\eps_1\eps_2)^{1/4}),$$
where $E^{\rm odd}_{\eps_1,\eps_2}(z_1, z_2)$ is the event that there is an odd gasket of nested CLE$_4$ that intersects both $\D(z_1,\eps_1)$ and $\D(z_2,\eps_2)$.

\item $$\P(E^{\rm even}_{\eps_1,\eps_2}(z_1, z_2))
=\sum_{n \in 2\Z+1} \P(E^{2\lambda n}_{\eps_1,\eps_2})  
+O((\eps_1\eps_2)^{1/4}),$$
where $E^{\rm even}_{\eps_1,\eps_2}(z_1, z_2)$ 
is the event that there is an even gasket of nested 
CLE$_4$ that intersects both 
$\D(z_1,\eps_1)$ and $\D(z_2,\eps_2)$.
\end{itemize}
\end{prop}

So to conclude Theorem \ref{thm:main1} it remains to just sum over the boundary conditions and link them to calculations with 
$\cos$ and $\sin$. This is the content of the following proposition.

\begin{prop}\label{prop:calc2ptsum}
Given the equivalence from Proposition \ref{prop:2ptNestedCLE}, we have the following representations:
\begin{eqnarray*}
\nonumber
 (\eps_1 \eps_2)^{-1/8}\P(E^{\rm nest}_{\eps_1,\eps_2}(z_1, z_2))+o(1) 
 &=& (\eps_1 \eps_2)^{-1/8}\sum_{n \in \Z} \P(E^{2\lambda n}_{\eps_1,\eps_2}) + o(1),
 \\
 &=& C_{\rm nest}
 \CR(z_1,D)^{-1/8}\CR(z_2,D)^{-1/8}\frac{\theta_3(q^{1/2})}{\theta_2(q^{1/4})},
\end{eqnarray*}
\begin{eqnarray*}
\nonumber
(\eps_1 \eps_2)^{-1/8}\P(E^{\rm odd}_{\eps_1,\eps_2}(z_1, z_2))+o(1)
&=& (\eps_1 \eps_2)^{-1/8}\sum_{n \in 2\Z} \P(E^{2\lambda n}_{\eps_1,\eps_2})+ o(1) 
\\&=& \dfrac{1}{2}C_{\rm nest}
\CR(z_1,D)^{-1/8}\CR(z_2,D)^{-1/8}\frac{\theta_3(q^{1/8})}{\theta_2(q^{1/4})},
\end{eqnarray*}
\begin{eqnarray*}
\nonumber
(\eps_1 \eps_2)^{-1/8}\P(E^{\rm even}_{\eps_1,\eps_2}(z_1, z_2))+o(1)
&=&
(\eps_1 \eps_2)^{-1/8}\sum_{n \in 2\Z+1} \P(E^{2\lambda n}_{\eps_1,\eps_2}) + o(1)
\\
&=&
\dfrac{1}{2}C_{\rm nest}
\CR(z_1,D)^{-1/8}\CR(z_2,D)^{-1/8}\frac{\theta_4(q^{1/8})}{\theta_2(q^{1/4})}.
\end{eqnarray*}
Further, the right hand sides of these formulas are, 
up to explicit constants, equal to 
$$\Big\langle:\exp\Big(i\sqrt\frac{\pi}{2}\Phi_D(z_1)\Big):~
:\exp\Big(-i\sqrt\frac{\pi}{2}\Phi_D(z_2)\Big):
\Big\rangle = \CR(z_1,D)^{-1/8}\CR(z_2,D)^{-1/8}\exp(\frac{\pi}{2}G_D(z_1,z_2)),$$
$$\Big\langle:\cos(\sqrt\frac{\pi}{2}\Phi^D(z_1))::\cos(\sqrt\frac{\pi}{2}\Phi^D(z_2)):\Big\rangle,$$
and
$$\Big\langle:\sin(\sqrt\frac{\pi}{2}\Phi^D(z_1))::\sin(\sqrt\frac{\pi}{2}\Phi^D(z_2)):\Big\rangle,$$
respectively.
\end{prop}

The proof is basically just pasting things together, with the summations controlled by the following lemma.
\begin{lemma}
\label{Lem E eps 1 eps 2 bound v}
Let $z_1\neq z_2\in D$
and $r_0>0$ such that the disks
$\overline{\D(z_1,r_0)}$ and $\overline{\D(z_2,r_0)}$
are disjoint and contained in $D$.
There is a constant
$c=c(D,z_1,z_2,r_0)>0$ 
(depending on $D$, $z_1$, $z_2$ and $r_0$),
such that for every $v\in\R$,
and for every $\varepsilon_1,\varepsilon_2\in (0,r_0]$,
$$
\P(E_{\varepsilon_1,\varepsilon_2}^v)
\leq
\P(E_{\varepsilon_1,\varepsilon_2}^0)
e^{-c v^2}.
$$

In particular, let $\cV$ be a countable subset of $\R$ and 
$a:\cV\rightarrow\R$.
Then assuming that for every $b>0$,
$$
\sum_{v\in\cV}
\vert a(v)\vert e^{-b v^2}
<+\infty,
$$
we have that as $\max(\varepsilon_1,\varepsilon_2)\to 0$,
$$
\sum_{v\in\cV}
a(v)\P(E_{\varepsilon_1,\varepsilon_2}^v)
=
\varepsilon_1^{1/8}
\varepsilon_2^{1/8}
(\CSLE)^2
\,
\sigma_{\rm tw}^D(z_1,z_2)
\,
\Big(
\sum_{v\in\cV}a(v)
e^{-v^2 \Modd_D(z_1,z_2)}
\Big)
(1+o(1)).
$$
\end{lemma}

\begin{proof}
The boundary conditions $0$ and $v$ are naturally coupled on the same probability space by adding to 
the Brownian loop soup $\cL_{D}$
the independent Poisson point process of
boundary excursions $\Xi^{v^2/2}_{D}$.
In this way,
we have the inclusion of events:
\begin{multline*}
E_{\varepsilon_1,\varepsilon_2}^v
\subset
\{
\text{No cluster of }
\cL_{D}
\text{ disconnects } \D(z_1,\varepsilon_1) 
\text{ from } \D(z_2,\varepsilon_2),
\\
\text{and no excursion in }
\Xi^{v^2/2}_{D}
\text{ disconnects }
\D(z_1,r_0) \text{ from } \D(z_2,r_0)
\}.
\end{multline*}
By independence of $\cL_{D}$
and $\Xi^{v^2/2}_{D}$
we get that
$$
\P(E_{\varepsilon_1,\varepsilon_2}^v)
\leq
\P(E_{\varepsilon_1,\varepsilon_2}^0)
\P(\text{No excursion in }
\Xi^{v^2/2}_{D}
\text{ disconnects }
\D(z_1,r_0) \text{ from } \D(z_2,r_0)) .
$$
Since $\Xi^{v^2/2}_{D}$ is a Poisson point process,
$$
\P(\text{No excursion in }
\Xi^{v^2/2}_{D}
\text{ disconnects }
\D(z_1,r_0) \text{ from } \D(z_2,r_0))
=
e^{-c v^2},
$$
for some $c>0$.
\end{proof}

\begin{proof}[Proof of Proposition \ref{prop:calc2ptsum}]
We will just explain the first formula and its equivalence with the imaginary chaos, the other two are analogous. 

Let $C = \Ctw (\CSLE)^2$. By applying Theorem \ref{thm:equivtwist}, Theorem \ref{thm:twistbc} and Theorem \ref{thm:twopointtwist}, we can calculate
$$(\eps_1 \eps_2)^{-1/8}\P(E^{2\lambda n}_{\eps_1,\eps_2}) = C
\CR(z_1,D)^{-1/8}\CR(z_2,D)^{-1/8}\frac{\hat q^{\frac{4}{\pi}(4\lambda^2n^2)}}{\theta_2(q^{1/4})\sqrt{|\log q|}} +o(1),$$
which recalling $2\lambda = \sqrt{\pi/2}$ gives
$$(\eps_1 \eps_2)^{-1/8}\P(E^{2\lambda n}_{\eps_1,\eps_2}) = C
\CR(z_1,D)^{-1/8}\CR(z_2,D)^{-1/8}\frac{\hat q^{2n^2}}{\theta_2(q^{1/4})\sqrt{|\log q|}} +o(1).$$
Using now Lemma \ref{Lem E eps 1 eps 2 bound v} to justify the summation, we obtain
$$(\eps_1 \eps_2)^{-1/8}\sum_{n \in \Z} \P(E^{2\lambda n}_{\eps_1,\eps_2}) + o(1) = C
\CR(z_1,D)^{-1/8}\CR(z_2,D)^{-1/8}\frac{\theta_3(\hat q^{2})}{\theta_2(q^{1/4})\sqrt{|\log q|}}.$$
Applying now Poisson summation formula (also called Jacobi transformation for $z = 0$, \cite[Formula 1.7.12]{lawden2013elliptic}) gives 
$$ C_{\rm nest}
 \CR(z_1,D)^{-1/8}\CR(z_2,D)^{-1/8}\frac{\theta_3(q^{1/2})}{\theta_2(q^{1/4})}$$
with $C_{\rm nest} =  \frac{\pi^{1/4} (\CSLE)^2}{2^{5/4}} $ as claimed.

To link it to 
$$\Big\langle:\exp\Big(i\sqrt\frac{\pi}{2}\Phi_D(z_1)\Big):~
:\exp\Big(-i\sqrt\frac{\pi}{2}\Phi_D(z_2)\Big):
\Big\rangle,$$
observe that this quantity can be explicitly calculated via approximation by circle-averages \cite{junnila2020imaginary}, and gives exactly
$$\CR(z_1,D)^{-1/8}\CR(z_2,D)^{-1/8}\exp(\frac{\pi}{2}G_D(z_1,z_2)).$$
But now recall that by Lemma \ref{lem:parametr}
$$\exp(\pi G_D(z_1,z_2))=\theta_3(q^{1/2})/\theta_2(q^{1/2}),$$
and hence by applying relation \eqref{eq:thetadupl}
to write 
$$\theta_2(q^{1/4})=\sqrt{2\theta_2(q^{1/2})\theta_3(q^{1/2})},$$
we obtain that
$$\frac{\theta_3(q^{1/2})}{\theta_2(q^{1/4})} = \frac{\theta_3(q^{1/2})^{1/2}}{\sqrt{2}\theta_2(q^{1/2})^{1/2}} 
=2^{-1/2}\exp\big(\frac{\pi}{2}G_D(z_1,z_2)\big)$$
as desired.
\end{proof}

\subsubsection{Proof of Proposition \ref{prop:2ptNestedCLE}}
To prove the proposition, we will use three lemmas. First a result that helps to separate 
$E_{\eps_1,\eps_2}^{\rm odd}(z_1,z_2)$ or $E_{\eps_1,\eps_2}^{\rm even}(z_1,z_2)$ into disjoint events. 

\begin{lemma}\label{lem:nointersect}
Let $F_{\eps_1,\eps_2}$ denote the event that at least two different odd CLE$_4$ gaskets of nested CLE$_4$ intersect both 
$\D(z_1, \eps_1)$ and $\D(z_2, \eps_2)$. Then
\begin{itemize}
    \item Let $E_{\eps_1,\eps_2}^{\rm{odd},k}(z_1,z_2)$ denote the event that the $k$-th odd nested CLE$_4$ gasket intersects 
    $\D(z_1,\eps_1)$ and $\D(z_2,\eps_2)$. Then there is a constant $C > 0$
    such that for all $k \geq 1$,
    $$\P(F_{\eps_1,\eps_2} \cap E_{\eps_1,\eps_2}^{\rm{odd},k}(z_1,z_2)) 
    \leq C(\eps_1\eps_2)^{1/4}\left(\P(E_{\eps_1,\eps_2}^{\rm{odd},k}) + 
    \1_{k \geq 2}\P(E_{\eps_1,\eps_2}^{\rm{odd},k-1})\right).$$
    \item Further, for some $\widetilde C > 0$, 
    for all $\eps_1, \eps_2 > 0$ sufficiently small, 
$$\sum_{k \geq 1}\P(E^{\rm{odd},k}_{\eps_1,\eps_2}(z_1,z_2)) < 
\widetilde C\P(E^{\rm odd}_{\eps_1,\eps_2}(z_1,z_2)).$$
    \item In particular, $\P(F_{\eps_1,\eps_2})$ is bounded by 
    $C(\eps_1\eps_2)^{1/4}\P(E^{\rm odd}_{\eps_1,\eps_2}(z_1, z_2))$ for some $C > 0$ as 
    $\max(\eps_1,\eps_2) \to 0$. 
\end{itemize}
Further, the same holds when we replace odd by even everywhere.
\end{lemma}
\noindent Second, to obtain the summation over all gaskets, we also need to estimate the probability that an odd and even gasket simultaneously intersect the disks.
\begin{lemma}\label{lem:oddeven}
Let $F_{\eps_1,\eps_2}^{\rm two}(z_1, z_2)$ denote the event that at least two CLE$_4$ gaskets in the nested CLE$_4$ intersect both 
$\D(z_1, \eps_1)$ and $\D(z_2, \eps_2)$. 
Then $\P(F_{\eps_1,\eps_2}^{\rm two}(z_1, z_2)) \leq C\min(\eps_1,\eps_2)^{1/3}$,
for some constant $C>0$.\footnote{The choice of the exponent $1/3$ is arbitrary; one can see from the proof that any exponent less than $1/2$ would work, and most likely $1/2$ is the optimal one.}
\end{lemma}

\noindent Finally, we will also make use of the following observation about random walks, which the reader will notice is easily proved.

\begin{lemma}\label{lem:walk1}
Let $\P^W_{n}$ be the probability measure on simple random walks $(S_k)_{k\geq 0}$ starting from $n$. 
\begin{itemize}
\item Consider the infinite measure $\mu^A := \sum_{n \in \Z} \P^W_n$. Then $\mu^A(S_{k} = 0) = 1$ for all $k \geq 0$.
\item Consider the infinite measure $\mu^E := \sum_{n \in 2\Z} \P^W_n$. Then $\mu^E(S_{2k} = 0) = 1$ and $\mu^E(S_{2k+1}=0) = 0$ for all $k \geq 0$ .
\item Consider the infinite measure $\mu^O := \sum_{n \in 2\Z+1} \P^W_n$. Then $\mu^O(S_{2k+1} = 0) = 1$ and $\mu^O(S_{2k}=0) = 0$  for all $k \geq 0$. 
\end{itemize}

\end{lemma}
\noindent Before proving Lemmas \ref{lem:nointersect} and \ref{lem:oddeven} let us show how the proposition follows.

\begin{proof}[Proof of Proposition \ref{prop:2ptNestedCLE}:]
Let us start from the case of either odd gaskets or even gaskets, and then treat the case pertaining any CLE$_4$ gasket.

\paragraph{The case of odd / even gaskets.}
Let $\widehat{E}_{\eps_1,\eps_2}^{\rm odd}(z_1, z_2)$ be given by 
$E_{\eps_1,\eps_2}^{\rm odd}(z_1, z_2)$ intersected with the condition that only one odd CLE$_4$ gasket intersects both $\D(z_1,\eps_1)$ and 
$\D(z_2, \eps_2)$. Then the third point of Lemma \ref{lem:nointersect} 
implies that
$$\P(\widehat E_{\eps_1,\eps_2}^{\rm odd}(z_1, z_2)) = 
\P(E_{\eps_1,\eps_2}^{\rm odd}(z_1, z_2)) + O((\eps_1\eps_2)^{1/4}).$$

Now, consider the unique sequence of nested odd CLE$_4$ gaskets that surround at least one of 
$z_1, z_2$ until a gasket that touches both $\D(z_1, \eps_1)$ and $\D(z_2, \eps_2)$ appears or a gasket that does not surround any point of the disks appears. 
By the law of total probability we can decompose 
$$\P(\widehat{E}_{\eps_1,\eps_2}^{\rm odd}(z_1, z_2)) = 
\sum_{k \geq 1}\P(\widehat{E}_{\eps_1,\eps_2}^{\rm{odd}, k}(z_1, z_2)),$$
where $\widehat{E}_{\eps_1,\eps_2}^{\rm{odd},k}(z_1, z_2)$ denotes the event that only one odd gasket of the nested CLE$_4$ intersects both 
$\D(z_1,\eps_1)$ and $\D(z_2,\eps_2)$ and that it is the $k$-th odd one when looking at the sequence of nested gaskets described above. 

But by summing over all boundary conditions $4\lambda n$ we can apply the second point of Lemma \ref{lem:walk1} to see that for each $k \geq 1$
$$\P(\widehat E_{\eps_1,\eps_2}^{\rm{odd},k}(z_1, z_2)) = 
\mu^E(S_{2(k-1)} = 0)\P(\widehat E_{\eps_1,\eps_2}^{\rm{odd},k}(z_1,z_2)) 
= \sum_{n \in 2\Z} \P(\hat I_{\eps_1,\eps_2}^{2\lambda n,k}),$$
where $\hat I_{\eps_1,\eps_2}^{2\lambda n,k}$ is the event that for the GFF with boundary values $2\lambda n$ the $k$-th odd nested CLE$_4$ gasket of height $0$ intersects both 
$\D(z_1, \eps_1)$ and $\D(z_2, \eps_2)$, and is the only odd gasket to intersect both. 
Summing over $k$ now gives 
$$\P(\widehat E_{\eps_1,\eps_2}^{\rm odd}(z_1, z_2)) 
= \sum_{k \geq 1}\sum_{n \in 2\Z}
\P(\hat I_{\eps_1,\eps_2}^{2\lambda n,k}), $$
and by Fubini and the law of total probability
$$\P(\widehat E_{\eps_1,\eps_2}^{\rm odd}(z_1, z_2)) = 
\sum_{n \in 2\Z}\sum_{k \geq 1}
\P(\hat I_{\eps_1,\eps_2}^{2\lambda n,k}) 
= \sum_{n \in 2\Z}\P(\hat I_{\eps_1,\eps_2}^{2\lambda n}),$$
where $\hat I_{\eps_1,\eps_2}^{2\lambda n}$ is the event that for the GFF with boundary values $2\lambda n$ there is exactly one odd nested CLE$_4$ gasket that intersects both $\D(z_1, \eps_1)$ and $\D(z_2, \eps_2)$,
and moreover the height of this gasket is $0$. 
It now remains to show that 
$$ \sum_{n \in 2\Z}\P(\hat I_{\eps_1,\eps_2}^{2\lambda n}) 
=\sum_{n \in 2\Z}\P(E_{\eps_1,\eps_2}^{2\lambda n}) + O((\eps_1\eps_2)^{1/4}).$$
To do this, we first bound
$$\P(E_{\eps_1,\eps_2}^{2\lambda n} 
\cap F_{\eps_1,\eps_2}) \leq 
\sum_{k \geq 1} \P(E_{\eps_1,\eps_2}^{\rm{odd},k} \cap F_{\eps_1, \eps_2})
\P_n^W(S_{2(k-1)} = 0),$$
where we recall that $F_{\eps_1,\eps_2}$ denotes the event that at least two different odd CLE$_4$ gaskets of nested CLE$_4$ intersect both 
$\D(z_1, \eps_1)$ and $\D(z_2, \eps_2)$.
Now applying Lemma \ref{lem:nointersect} point 1, we can bound this by
$$C(\eps_1\eps_2)^{1/4}\sum_{k \geq 1}
\P_n^W(S_{2(k-1)} = 0)
\left(\P(E_{\eps_1,\eps_2}^{\rm{odd},k} ) + 
\1_{k \geq 2}\P(E_{\eps_1,\eps_2}^{\rm{odd},k-1})\right).$$
So we can bound 
\begin{equation}\label{eq:summationbound}
\sum_{n \in 2\Z}\P(E_{\eps_1,\eps_2}^{2\lambda n}) - 
\sum_{n \in 2\Z}\P(\hat I_{\eps_1,\eps_2}^{2\lambda n}) 
\leq 2C(\eps_1\eps_2)^{1/4} \sum_{n \in 2\Z}\sum_{k \geq 1}
\P_n^W(S_{2(k-1)} = 0)\P(E_{\eps_1,\eps_2}^{\rm{odd},k}). 
\end{equation}
Applying Fubini and Lemma \ref{lem:walk1} we obtain a bound of the form 
$2C (\eps_1\eps_2)^{1/4} \sum_{k \geq 1}\P(E_{\eps_1,\eps_2}^{\rm{odd},k}),$
which is finite by point 2 of Lemma \ref{lem:nointersect}.    

\paragraph{Putting together odd and even gaskets.} 

From the above, we have that 
$$\P(E_{\eps_1,\eps_2}^{\rm odd}(z_1, z_2)) = 
\sum_{n \in 2 \Z}\P(E_{\eps_1,\eps_2}^{2\lambda n}) 
+ O((\eps_1\eps_2)^{1/4}),$$
and
$$\P(E_{\eps_1,\eps_2}^{\rm even}(z_1, z_2)) = 
\sum_{n \in 2 \Z+1}\P(E_{\eps_1,\eps_2}^{2\lambda n}) 
+ O((\eps_1\eps_2)^{1/4}).$$
We further have that $E_{\eps_1,\eps_2}^{\rm odd}(z_1, z_2) \cap E_{\eps_1,\eps_2}^{\rm even}(z_1, z_2) \subseteq 
F_{\eps_1,\eps_2}^{\rm two}(z_1, z_2)$, 
where the event $F_{\eps_1,\eps_2}^{\rm two}(z_1, z_2)$ is as in Lemma \ref{lem:oddeven} the event that at least two CLE$_4$ gaskets intersect the $\eps-$disks around $z_1, z_2$. But by this lemma this event is bounded by $C\min(\eps_1,\eps_2)^{1/3}$.
Both bounds are $O((\eps_1\eps_2)^{1/6})$,
and thus we conclude.
\end{proof}

\subsubsection{Proof of Lemma \ref{lem:oddeven}}

Although the bound we give is rather weak in a typical situation, Lemma \ref{lem:oddeven} is surprisingly not straightforward. In this direction we already observed the following above.
\begin{itemize}
    \item In order for at least two CLE$_4$ gaskets to intersect the disks $\D(z_1, \eps_1)$ and 
    $\D(z_2, \eps_2)$, in fact two consecutive ones have to intersect them and thus there has to be some loop in the nested CLE$_4$ intersecting both disks.
\end{itemize}
The idea of the proof is to basically calculate 
$\E\big[:\exp(i\alpha\sqrt{2\pi}\Phi_D):\big](z)$ 
in two ways: first directly, and second by first taking a conditional expectation w.r.t. an exploration 
$\hat{A}_t$ making terms of the type 
$\CR(z,D\setminus \hat{A}_t)^{-\alpha^2/2}$ appear. 
A similar idea has been used in \cite{SSV}, and stems originally from \cite{APS,ALS1}. 
To formalize this idea we will use exponential martingales. It might be surprising that we only use the martingale for one point, but the stopping times we choose depend on both points.

\begin{proof}[Proof of Lemma \ref{lem:oddeven}]
Consider the exploration of CLE$_4$ loops via $A_{-\lambda, \lambda}$ from \cite{ASW}:
\begin{itemize}
\item We start with the TVS $A_1 := A_{-\lambda, \lambda}$.
\item If a loop of $A_1$ intersects both $\D(z_1, \eps_1)$ and $\D(z_2, \eps_2)$, then we stop.
\item If a loop of $A_1$  separates the two disks in the sense that no connected component of 
$\D \setminus A_1$ intersects both $\D(z_1, \eps_1)$ and $\D(z_2, \eps_2)$,
then we also stop.
\item Otherwise we construct $A_2$ by sampling $A_{-\lambda, \lambda}$ only inside the connected component of $\D \setminus A_1$ containing entirely at least one of the disks.
\item We iterate the construction until we stop because of the conditions 1 or 2. We call this discrete stopping time $N$ and we know it is almost surely finite \cite{ASW}.
\end{itemize}
Further, by tracing the relevant SLE$_4(\rho, -2-\rho)$ curves,
this exploration can also be done in continuous time giving rise to a local set process $\hat A_t$ \cite{ASW} such that $A_n = \hat A_{\tau_n}$, and the stopping time $\tau_n$ is defined by finishing the relevant $A_{-\lambda, \lambda}$ in the exploration above. A useful martingale is then given by $h_t(z) := \Phi_{\hat{A}_t}(z)$, where $\Phi_{\hat{A}_t}(z)$ denotes the conditional expectation of $\Phi$ in $z$ given the local set $\hat{A}_t$, and it is pointwise defined until the process hits $z$, which happens with $0$ probability for a fixed $z\in D$. 
We will be interested in the related exponential martingale
$$M_t(z) := \exp\left(i\alpha \sqrt{2\pi}h_t(z)+\alpha^2\pi\langle h_t(z)\rangle\right).$$
Here the quadratic variation can be explicitly calculated and is given by 
$-\frac{1}{2\pi}\log \CR(z,D \setminus \hat{A}_t)$.

Now observe that by choosing $\alpha\in (0,1)$ small enough, 
if we look at this martingale at the times 
$(t_n)_{n\geq 0}$ corresponding to the steps in the exploration described above, then the discrete-time process 
$(M_{t_n}(z))_{n\geq 0}$ is still a martingale. 
Indeed, iterating Lemma 4.1 from \cite{SSV} for 
$n \geq 1$, 
we get, for $p\in(1,\alpha^{-2})$, the following bound:
$$\E(|M_{t_n}(z)|^p) \leq C\E\left[\CR(z,D \setminus A_n)^{-p\alpha^2}\right] < +\infty. $$
We now claim that in fact one can apply the Optional Stopping Theorem to the stopping time $T = t_N$ and that in fact $N$ has super-exponential tails.
\begin{claim}\label{claim:stop}
The number of iterations $N$ corresponding to the stopping time $T$ is bounded by a random variable that has super-exponential tails. In particular $M_T$ is integrable and the Optional Stopping Theorem holds at the time $T$. 
\end{claim}
Before proving this, let us explain how the Lemma follows.
By symmetry, we assume that $\eps_1\leq \eps_2$.
Using the claim we can apply the Optional Stopping Theorem to obtain that
\begin{equation}\label{eq:globalbound}
\E [\Re(M_T(z_1))] = \Re(M_0(z_1)) = 
\CR(z_1, D)^{-\alpha^2/2}< +\infty.
\end{equation}
On the other hand, at the stopping time $T$ we have:
$$\E(\Re(M_T(z_1)))=\E\left[\CR(z_1, D\setminus \hat{A}_T)^{-\alpha^2/2}\Re(\exp(i\alpha\sqrt{2\pi} h_{T}(z_1)))\right].$$
As $h_{T}(z_1)$ is symmetric in law and its sign is independent of the conformal radius, we can further write this as 
$$\E\left[\CR(z_1, D\setminus \hat{A}_T)
^{-\alpha^2/2}\cos(\alpha \sqrt{2\pi} h_{T}(z_1))\right].$$
Now, observe that conditionally on $N$, 
the value of $h_{T}(z_1)$ is given by the end-point of a symmetric simple random walk with $N$ steps and step-size 
$\sqrt{\pi/8}$ (recall we iterate $A_{-\lambda, \lambda}$ in the construction).
Moreover, this value is conditionally independent of 
$\CR(z_1, D \setminus \hat{A}_T)$. 
Hence, we have that  
$$\E(\Re(M_T(z_1))) = 
\E\left[\cos(\alpha \pi/2)^N
\E(\CR(z_1, D\setminus \hat{A}_T)
^{-\alpha^2/2}|N)\right].$$
Notice that because of $\alpha < 1$, the term inside the expectation is positive and thus we can bound the expectation from below by
$$C\eps^{-\alpha^2/2}\E\left[\cos(\alpha \pi/2)^N\P(F_{\eps}|N)\right],$$
where $F_\eps$ is the event that at the stopping time $T$ we have that $d(z_1,\hat{A}_T)\leq \eps$ 
(and thus $\CR(z_1, D \setminus \hat{A}_T) \leq 4\eps)$ by Koebe's distortion estimates). 

Now, for $\delta > 0$ small, we pick for example $\alpha$ such that $\alpha^2/2 = 1/3 + \delta$ and pick $c > 0$ such that $-c \log (\cos(\alpha \pi/2)) = \delta$. 
Then, by bounding
$$\E\left[\cos(\alpha \pi/2)^N\P(F_\eps|N)\right] \geq \eps^\delta\P(F_\eps \text{ and } N \leq c\log \eps^{-1}) 
\geq \eps^\delta(\P(F_\eps) - 
\P(N > c\log \eps^{-1})),$$
we obtain that 
$$\P(F_\eps) \leq C\eps^{1/3} + \P(N > c \log \eps^{-1}).$$ By the Claim \ref{claim:stop} the second term can be bounded also by $O(\eps^{1/3})$. 
But if there is a loop in the nested CLE$_4$ that
intersects both
$\D(z_1,\eps_1)$ and $\D(z_2,\eps_2)$,
then necessarily we are on the event $F_{\eps_1}$,
which has, by the above, a probability 
$O(\eps_1^{1/3})=O(\min(\eps_1,\eps_2)^{1/3})$.
So we conclude the proof of the lemma 
(modulo Claim \ref{claim:stop}). 
Observe that the exponent $1/3$ here was arbitrary and any exponent less than $1/2$ would have worked equally well.

\begin{proof}[Proof of Claim \ref{claim:stop}]
Denote by $p_n$ the probability that at step $n$ the separation event in the construction does not happen. We will prove that there is some $c >  0$ and some i.i.d. almost surely positive random variables $(X_i)_{i \geq 1}$ such that
$$p_n \leq 
\E\Big[\exp\Big(-c\sum_{i = 1}^n X_i\Big)\Big].$$

So let us prove this inequality. To do this we first observe that any simply-connected domain with two marked points $z_1,z_2$ can be mapped uniquely to the horizontal strip $\strip := \{z \in \C: \Im(z) \in (-1,1)\}$ such that 
$z_1 \mapsto 0$ and $z_2 \mapsto x \in(0,+\infty)$. 
We now consider the construction of $A_{-\lambda, \lambda}$ in this domain as a closed union of level lines between any pairs of some fixed dense countable subset of points on the boundaries - see \cite[Lemma 3.6]{AS2}. 
We can draw these level lines in any order we wish.

So, fix some $L\in (0,1)$ small enough, 
and cut the rectangle 
$\{z: \Im(z) \in (-1,1), \Re(z) \in (0,x)\}$ into rectangles of width $L$, i.e. of the form $\{z: \Im(z) \in (-1,1), \Re(z) \in mL, (m+1)L]\}$ with $(m+1)L < x$ and $m \in \N$. 

Now do the following: we start sampling the level lines (which are SLE$_4(-1,-1)$ curves unconditionally) from the bottom midpoints in the even rectangles from left to right, each until they hit one of the other boundaries (top left or right) of the rectangle.

Observe that the conditional probability that a SLE$_4(-1,-1)$ level line from the midpoint of the side of the rectangle living on $\Im(z) = -1$ to the midpoint living on $\Im(z) = 1$ exits from the upper boundary first is uniformly bounded from zero by some $\hat{p}(L) > 0$, conditioned on any configuration of level line bits inside the previous even rectangles. Indeed, this follows from absolute continuity of the GFF and noticing that the Radon-Nikodym derivative remains bounded because of the gap between the even rectangles - see e.g. \cite[Theorem 2.9 and Lemma 3.1]{ALS3}. 

This way we have at least $x/3L$ independent trials to separate the two $\eps_j-$disks, i.e. we obtain that the non-separation probability is bounded from above by 
$(1-\hat{p}(L))^{x/3L}$.

To finish the claim it now remains to argue that this number of trials will go up with the iteration time. To do this notice that the simply-connected domain containing the two will change after constructing $A_{-\lambda, \lambda}$. Each time when $z_1, z_2$ have not been separated, we can uniformize the domain to the strip again as above and denote by $x_n$ the image of $z_2$. The argument before then gives $p_n \leq (1-\hat{p}(L))^{x_n/3L}$.
To conclude we now show that in the process 
$(x_n - x_{n-1})_{n \geq 1}$, conditionally on $(x_1, \dots, x_{n-1})$, the increment $x_n - x_{n-1}$ is stochastically dominated from below by some fixed random variable $Z$ (i.e. uniformly in $n$) 
that is positive almost surely.

But this follows from an analogous argument to the right of the point $x$. Indeed, we can also cut the strip to the right of the point $x$ into rectangles of width $L$ exactly as above and start sampling the level lines from the bottom mid-points from left to right. As above, the probability that there is at least one crossing level line in a rectangle $[x+mL, x+(m+1)L]$ with $m \leq M$, is given by $1-(1-\hat p(L))^M$. 

The claim is then a consequence of the following. Suppose a level line $\gamma$ crosses vertically inside a rectangle $[x+NL, x+(N+1)L]$, then if we uniformize the domain 
$\strip \setminus \gamma$ back to 
$\strip$ with $0 \mapsto 0$ and $x \to \tilde x$, 
then  $\tilde x - x > r_{0} \exp(-\pi L N)$ 
for some $r_0 > 0$ uniformly for all $x > x_0$. 

First, the Dirichlet Green's function in the strip is given for $x \in \R_+$ by 
$$G_\strip(0,x)= \frac{1}{2\pi}\log\frac{1+\exp(-\pi x/2)}{1-\exp(-\pi x /2)}.$$
Second, 
$$G_{\strip\setminus A_{-\lambda, \lambda}}(0,x) 
< G_{\strip \setminus \gamma}(0,x) < 
G_{\strip \setminus \{x+(N+1)L\}\times (-1,1)}(0,x),$$
where $\gamma$ is above a level line crossing inside a rectangle $[x+NL, x+(N+1)L]$.
And finally, setting $C = (N+1)L$, one can also explicitly calculate
$$G_{\strip \setminus \{x+C\}\times (-1,1)}(0,x) 
= \frac{1}{2\pi}
\log\left(
\frac{1-e^{-\pi(x+2C)/2}}{1+e^{-\pi(x+2C)/2}}
\cdot
\frac{1+e^{-\pi x/2}}{1-e^{-\pi x/2}}\right).$$
Using conformal invariance of the Green's function and comparing these expressions we deduce that 
$$\tilde x - x > \frac{2}{\pi}\log \frac{1-\exp(-\pi (x+C))}{1-\exp(-\pi C)} $$
concluding the proof of claim.
\end{proof}
\end{proof}

\subsubsection{Proof of Lemma \ref{lem:nointersect}}
We now move on to prove Lemma \ref{lem:nointersect}. We start with a similar observation.
\begin{itemize}
    \item In order for two consecutive odd CLE$_4$ gaskets to intersect the disks $\D(z_1, \eps_1)$ and 
    $\D(z_2, \eps_2)$, there have to be two consecutive CLE$_4$ loops that intersect both disks.
\end{itemize} 
In this case two further technical observations become handy.

\begin{lemma}\label{lem:tech1}
Let $z_1 \neq z_2 \in \C$. Let $D_1$ be a finite union of 
disjoint open simply-connected domains such that the boundary of each connected component intersects both 
$\D(z_1, \eps_1)$ and 
$\D(z_2, \eps_2)$. Let $D_2 \supseteq D_1$ be an open simply-connected domain such that the two
$\overline{\D(z_j, \eps_j)} \cap \partial D_2$,
$j\in \{1,2\}$, are non-empty.
Let $F_i$ be the event that a CLE$_4(D_i)$ loop\footnote{In the case where $D_1$ is a union of more than one domain, we sample independent CLE-s inside each.} hits both 
$\D(z_1, \eps_1)$ and $\D(z_2, \eps_2)$. 
Then $\P(F_1) \leq \P(F_2)$.
\end{lemma}

\begin{lemma}\label{lem:tech2}
Let $z_1 \neq z_2 \in \C$. Consider any simply-connected domain 
$D$ such that 
$\D(z_j, \eps_j) \subseteq D$ and 
$\partial \D(z_j, \eps_j) \cap \partial D \neq \emptyset$, for $j \in \{1,2\}$. 

Now consider any conformal map $\psi: D \to \D$ mapping a point in $\partial \D(z_1, \eps_1) \cap \partial D$ to $-1$ and a point in $\partial \D(z_2, \eps_2) \cap \partial D$ to $1$. 
Then the images $\psi(\D(z_1, \eps_1))$ and 
$\psi(\D(z_2, \eps_2))$ are included in disks 
$\D(-1, \delta_1)$ and $\D(1, \delta_2)$ respectively,
with crucially $\delta_1\delta_2 \leq C(\eps_1\eps_2)^{1/2}$ for $C > 0$ depending only on $z_1, z_2$,
and not on $D$, $\eps_1$, $\eps_2$ or the choice of $\psi$.

Further, for some $\widetilde C > 0$, 
one can always find a conformal map such that 
$\delta_1 \leq \widetilde C(\eps_1\eps_2)^{1/4}$ and 
$\delta_2 \leq \widetilde C(\eps_1\eps_2)^{1/4}$.
\end{lemma}

We will now show how to conclude Lemma \ref{lem:nointersect} then come back to these observations.

\begin{proof}[Proof of Lemma \ref{lem:nointersect}]
As observed, for the event 
$F_{\eps_1,\eps_2}$ to hold, 
there have to be at least two consecutive CLE$_4$ loops that intersect both $\D(z_1, \eps_1)$ and 
$\D(z_2, \eps_2)$. 
Thus, the event 
$F_{\eps_1,\eps_2} \cap 
E_{\eps_1,\eps_2}^{\rm{odd},k}(z_1, z_2)$ 
can only happen if in the exploration of nested CLE$_4$ loops containing at least some point of 
$\D(z_1, \eps_1) \cup \D(z_2, \eps_2)$ and 
not separating $\D(z_1, \eps_1)$ and $\D(z_2, \eps_2)$, either already the $k-1$-th loop intersects the 
$\eps_j$-disks or the $k$-th loop does. Thus by stopping the exploration at that iteration, we can reduce the lemma to the following claim.

\begin{claim}
\label{Claim CLE 4}
Let $D$ be a finite union of 
disjoint open simply-connected domains such that the boundary of each connected component intersects both 
$\D(z_1, \eps_1)$ and 
$\D(z_2, \eps_2)$.
Then the probability that a loop of the (simple) CLE$_4$ inside this domain intersects both $\D(z_1, \eps_1)$ and 
$\D(z_2, \eps_2)$ is bounded by 
$C(\eps_1\eps_2)^{1/4}$,
uniformly over the domain $D$.
\end{claim}

The point 1 of the Lemma is implied by this claim, because there are at most finitely many nested CLE$_4$ loops hitting both disks, see e.g. \cite[Lemma 2.1]{AruPaponPowell}. Further the points 2 and 3 of the lemma follow rather directly from the point 1. 
\begin{itemize}
\item For the second point we notice that the first point implies that 
$$\P(E_{\eps_1,\eps_2}^{\rm{odd},k}(z_1, z_2)) - \P(\widehat E_{\eps_1,\eps_2}^{\rm{odd},k}(z_1, z_2)) \leq C(\eps_1\eps_2)^{1/4}
\left(\P(E_{\eps_1,\eps_2}^{\rm{odd},k}) + 
\1_{k \geq 2}\P(E_{\eps_1,\eps_2}^{\rm{odd},k-1})\right),$$ 
where, as above, 
$\widehat E_{\eps_1,\eps_2}^{\rm{odd},k}(z_1, z_2) = 
E_{\eps_1,\eps_2}^{\rm{odd},k}(z_1,z_2) \cap 
F_{\eps_1,\eps_2}^c$.
This implies that
$$
\sum_{k \geq 0}
\P(E_{\eps_1,\eps_2}^{\rm{odd},k}(z_1,z_2)) 
\leq 
\sum_{k \geq 0}
\P(\widehat E_{\eps_1,\eps_2}^{\rm{odd},k}(z_1,z_2)) +2C(\eps_1\eps_2)^{1/4}\sum_{k \geq 0}
\P(E_{\eps_1,\eps_2}^{\rm{odd},k}(z_1,z_2)).
$$
Using $\sum_{k \geq 0}
\P(\widehat E_{\eps_1,\eps_2}^{\rm{odd},k}(z_1,z_2))\leq 
\P(E_{\eps_1,\eps_2}^{\rm odd}(z_1,z_2))$ and taking 
$\max(\eps_1,\eps_2)$ small enough,
we conclude.
\item The third point follows directly from writing 
$$F_{\eps_1,\eps_2} = \bigcup_{k \geq 1}F_{\eps_1,\eps_2} \cap E_{\eps_1,\eps_2}^{\rm{odd},k}(z_1,z_2),$$ 
and using the points 1 and 2.
\end{itemize}
Thus it remains to prove the claim.
\begin{proof}[Proof of Claim \ref{Claim CLE 4}]
By Lemma \ref{lem:tech1}, we can replace $D$ by a larger simply-connected domain $D \subseteq D_2$ so that each of $\overline{\D(z_j, \eps_j)}$ for 
$j \in \{1, 2\}$ intersects $\partial D_2$ only at a single point. We can then apply Lemma \ref{lem:tech2}, to find a conformal map such that the images of $\D(z_1, \eps_1)$ and $\D(z_2, \eps_2)$ are contained in disks of radius 
$\widetilde C(\eps_1\eps_2)^{1/4}$ around points $-1, 1$ respectively.
It hence remains to prove the following estimate: the probability that in the unit disk there is a CLE$_4$ loop touching the $C(\eps_1\eps_2)^{1/4}$-neighbourhoods of both $-1, 1$ decays at least like $(\eps_1\eps_2)^{1/4}$. This follows directly from 
\cite[Proposition 3.2]{GaoNolinQian2025CLEarms}. 
\end{proof}
\end{proof}

Finally, we provide the proofs of the two technical lemmas.

\begin{proof}[Proof of Lemma \ref{lem:tech1}]
This follows from Brownian loop soup considerations. Indeed, we can start by sampling the critical Brownian loop soup in $D_1$ and then get the loop soup in $D_2$ by adding all loops intersecting $D_2 \setminus D_1$. Notice that if an outer boundary of an outermost cluster hits both of these $\eps_j$-disks for the loop soup in $D_1$, then by adding new loops an outermost cluster boundary will again hit both balls. Indeed, because the $\eps_j$-disks hit the domain boundary $\partial D_2$, one cannot surround the outermost clusters of the loop soup in $D_1$ without forming another outermost cluster whose outermost boundary would hit the disks. 
\end{proof}

\begin{proof}[Proof of Lemma \ref{lem:tech2}]
Here we make use of the conformal modulus estimates. 

First observe that the extremal length between 
$\D(z_1, \eps_1)$ and $\D(z_2, \eps_2)$ inside any connected domain $D$ containing both disks is bounded from below by that of the extremal length inside $\C$ between these disks \cite[Section 4.4]{Ahlfors2010ConfInv}. This can be further bounded from below by $\frac{1}{2\pi}\log \frac{1}{\eps_1\eps_2} +O(1) $. 
Indeed, by the comparison principle of extremal lengths \cite[Theorem 4.2]{Ahlfors2010ConfInv}, 
the extremal length between the two small disks is larger than the sum of extremal lengths of curves crossing 
$\D(z_1, r) \setminus \D(z_1, \eps_1)$ and those crossing 
$\D(z_2, r) \setminus \D(z_2, \eps_2)$. 
Here $r$ is some macroscopic radius chosen such that the disks are disjoint and inside $\D$, e.g. $r = \frac{1}{3}\min(\vert z_1-z_2\vert, 1-\vert z_1\vert, 1-\vert z_2\vert)$. 
As the extremal length in these annuli is given by 
$\frac{\log \vert r/\eps_j\vert}{2\pi}$, the claimed lower bound follows.

On the other hand one can also calculate an upper bound on the extremal length inside $\D$ between any two curves 
$\gamma_1, \gamma_2$ that start from $-1,1$ respectively, stay inside $\D$ and go to distance at least $\delta_1, \delta_2$ of the points $-1, 1$ respectively. Indeed, an upper bound is given by $\frac{1}{\pi} \log \frac{1}{\delta_1 \delta_2} +O(1)$. 
This can be seen by mapping the disk via the right-half plane to the strip of height $1$ by using 
$z\mapsto\frac{1}{\pi}\log \frac{1+z}{1-z}$, mapping $\pm 1$ to $\pm \infty$. Then the segments $\gamma_1$, $\gamma_2$, map to curves from $-\infty$ and $\infty$ respectively, and attain real parts at least equal to $-\pi^{-1}\log \frac{2}{\delta_1}$ and $\pi^{-1}\log \frac{1}{2\delta_2}$ from the left and right respectively. But now we can use  
\cite[Theorem 4.8]{Ahlfors2010ConfInv} to upper bound this extremal length by 
$2\Lambda(\exp(\log(\frac{1}{\delta_1 \delta_2}))=2\Lambda((\delta_1 \delta_2)^{-1})$, 
where $\Lambda$ is the so-called modulus of the Teichmüller annulus. 
Using the bound 
$\exp(2\pi \Lambda(R)) \leq 16(R+1)$
\cite[Eq. 4.21]{Ahlfors2010ConfInv}, 
we conclude our upper bound.

But now extremal length is a conformal invariant and thus we obtain that 
$$\frac{1}{2\pi}\log \frac{1}{\eps_1\eps_2} +O(1) \leq \frac{1}{\pi} \log \frac{1}{\delta_1 \delta_2} +O(1),$$
from where we conclude that 
$\delta_1 \delta_2 \leq C(\eps_1\eps_2)^{1/2}$ for any such two curves. But then the claimed inclusion in the disks follows as we can always take curves $\gamma_1, \gamma_2$ going to the maximal possible distance inside $\psi(\D(z_1, \eps_1))$ and $\psi(\D(z_2, \eps_2))$ .

The final part of the lemma is obtained by using the family of conformal automorphisms of the unit disk $\D$ fixing $-1$ and $1$, to make the (reduced) extremal length from the origin to both images to be equal. 
Indeed, such maps are given by $\psi_a(z) = \frac{z+a}{1+az}$ for $a \in (-1,1)$. They expand near points $-1,1$ by a factor of $b, b^{-1}$ respectively with $b = \frac{1+a}{1-a}$. Hence by choosing $a$ appropriately we can find the appropriate map.
\end{proof}

\subsection{Outermost CLE$_4$ gasket}
\label{Subsec Outermost CLE4 gaskets}

Here we explain how to prove Theorem \ref{thm:main-simple}. It follows exactly the same strategy so we will just explain the very minor differences.

First, the key Proposition \ref{prop:2ptNestedCLE} is replaced by the following.

\begin{prop}\label{prop:2ptCLEoutermost}
As previously, denote by $E_{\eps_1,\eps_2}^{2\lambda n}$ the event that for a GFF in a simply connected domain $D$ with boundary values $2\lambda n$,
a nested height-0 CLE$_4$ gasket intersects both 
$\D(z_1,\eps_1)$ from $\D(z_2, \eps_2)$.

We then have that 
$$\sum_{n \in 2\Z} (-1)^{n/2}
\P(E_{\eps_1,\eps_2}^{2\lambda n}) 
= \P(E_{\eps_1,\eps_2}^{\rm simp}(z_1, z_2))+
O((\eps_1\eps_2)^{1/4}),$$
where $E_{\eps_1,\eps_2}^{\rm simp}(z_1, z_2)$ is the event that the outermost CLE$_4$ gasket intersects both 
$\D(z_1,\eps_1)$ and $\D(z_2,\eps_2)$.
\end{prop}

The proof of this proposition follows exactly in the same way as Proposition \ref{prop:2ptNestedCLE} above. The only difference is that instead of Lemma \ref{lem:walk1} we will use a different simple fact on random walks.

\begin{lemma}\label{lem:walk2}
Let $\P^W_{n}$ be the probability measure on simple random walks $(S_k)_{k\geq 0}$ starting from $n$. Now consider the infinite signed measure $\nu^W := \sum_{n \in 2\Z} (-1)^{n/2}\P^W_n$. Then $\nu^W(S_0 = 0) = 1$ and $\nu^W(S_k = 0) = 0$ for all $k > 0$.
\end{lemma}

The calculation of Proposition \ref{prop:calc2ptsum} is replaced by the following formula, whose proof is again the same. Indeed, we use Theorem \ref{thm:equivtwist}, Theorem \ref{thm:twistbc} and Theorem \ref{thm:twopointtwist} to get a first family formula, where Lemma \ref{Lem E eps 1 eps 2 bound v} justifies the summation. After that we apply Poisson resummation \cite[Formula 1.7.15]{lawden2013elliptic} and the duplication formula \eqref{eq:thetadupl}.

\begin{lemma}\label{lemma:calc2ptsimple}
We have the following summation for $(\eps_1 \eps_2)^{-1/8}\P(E_{\eps_1,\eps_2}^{\rm simp}(z_1, z_2))+
o(1)$:
$$(\eps_1 \eps_2)^{-1/8}\sum_{n \in 2\Z} (-1)^{n/2}\P(E^{2\lambda n}_{\eps_1,\eps_2}) + o(1) = C_{\rm simp}
\CR(z_1,D)^{-1/8}\CR(z_2,D)^{-1/8}\sqrt{\frac{\theta_3(q^{1/4})}{\theta_2(q^{1/4})}},$$
where $C_{\rm simp} = \frac{\pi^{1/4} (\CSLE)^2}{2^{3/4}}$.
\end{lemma}

Combining Proposition \ref{prop:2ptCLEoutermost} and Lemma \ref{lemma:calc2ptsimple} we conclude Theorem \ref{thm:main-simple}.
\subsection{The $m-$th CLE$_4$ gasket}
\label{Subsec kth gasket}

In fact by appropriately weighing the boundary conditions one can also find explicit series solutions to the renormalised probability that $z_1, z_2$ are on one of the $m-$th outermost CLE$_4$ gaskets surrounding the points, when counted from the outside. The case $m=1$ is the last subsection and is the only case where there is exactly one such gasket.

The proof is again very similar to the proofs in Sections \ref{Subsec Nested CLE4 gaskets} and 
\ref{Subsec Outermost CLE4 gaskets} above, 
so we omit the details and only bring out again the key differences. First, the random walk lemma is a slightly more refined calculation.

\begin{lemma}\label{lem:walkl}
Let $\P^W_{n}$ be the probability measure on simple random walks $(S_k)_{k\geq 0}$ starting from $n$. 
Then for any $m \geq 1$ there are real coefficients $(a_{m,n})_{n \in \Z}$ such that if we consider the infinite signed measure 
$\nu^{W}_m := \sum_{n \in \Z} a_{m,n}\P^W_n$. 
Then $\nu^{W}_m(S_m = 0) = 1$ and 
$\nu^{W}_m(S_k = 0) = 0$ for all $k \neq m$.

Moreover, these coefficients are explicit and given for $m \geq 1$ as follows:
\begin{itemize}
\item $a_{m,n} = 0$ for all $0 \leq n < m$ and all $n$ of different parity than $m$;
\item the coefficients are symmetric: $a_{m,n} = a_{m,-n}$;
\item we have that $a_{m,m} := 2^{m-1}$ and for $n > m$
$$a_{m,n} = 2^{m-1}(-1)^{\frac{n-m}{2}}\left( \binom{\frac{n+m}{2}}{\frac{n-m}{2}}+\binom{\frac{n+m}{2}-1}{\frac{n-m}{2}-1}\right).$$
\end{itemize}
In particular the growth of the coefficients $a_{m,n}$ satisfy $\vert a_{m,n}\vert \leq c_m (1+\vert n\vert^m)$.
\end{lemma}

\begin{proof}
Let $p_k(n) := \P^W_n(S_k = 0)$ and consider the generating function $g_n(z) := \sum_{k \geq 0}p_k(n)z^k$ for $n \geq 0$, and let $g_{-n} = g_n$. 
From the relation 
$p_k(n) = \frac{1}{2}(p_{k-1}(n+1) + p_{k-1}(n-1))$, 
we obtain that for all $n > 0$,
$g_n(z) = \frac{z}{2}(g_{n+1}(z) + g_{n-1}(z))$.
This relation can be solved explicitly to give
$$g_n(z) = \frac{1}{\sqrt{1-z^2}}\left(\frac{1-\sqrt{1-z^2}}{z}\right)^n.$$
Thus our condition is 
$\sum_{n \in \Z} a_{m,n}g_n(z) = z^m$. 
It is more convenient now to use the variable 
$r = \frac{1-\sqrt{1-z^2}}{z}$, 
which provides us with the expression:
$$\frac{1+r^2}{1-r^2}(a_{m,0} 
+ 2\sum_{n \geq 1}a_{m,n} r^n) = (\frac{2r}{1+r^2})^m,$$
that is
$$a_{m,0} + 2\sum_{n \geq 1}a_{m,n} r^n 
= 2^mr^m\frac{1-r^2}{(1+r^2)^{m+1}} 
= 2^mr^m(1-r^2)\sum_{k \geq 0}(-1)^kr^{2k}\binom{k+m}{k}.$$
By comparing the coefficients, we obtain the lemma.
\end{proof}

Second, a bit more care is needed in obtaining the conclusion of the proof because the coefficients in the summations will be now given by a sequence $a_n$ growing in absolute value. Suppose for concreteness that $m$ is odd. Then in particular, the equivalent of Equation \eqref{eq:summationbound} becomes:
$$
\vert\sum_{n \in \Z}a_{m,n}
\P(E_{\eps_1,\eps_2}^{2\lambda n}) - 
\sum_{n \in \Z}a_{m,n}
\P(\hat I_{\eps_1,\eps_2}^{2\lambda n})\vert
\leq 2C(\eps_1\eps_2)^{1/4} 
\sum_{n \in \Z}\vert a_{m,n}\vert
\sum_{k \geq 1}\P_n^W(S_{2(k-1)} = 0)
\P(E_{\eps_1,\eps_2}^{\rm{odd},k}). 
$$
Now using the fact that $\P_n^W(S_{2(k-1)} = 0) = 0$ for all $|n| > 2(k-1)$ and the bound on $a_{m,n}$ in the previous lemma, we can bound
$$\sum_{n \in \Z}\vert a_{m,n}\vert
\P_n^W(S_{2(k-1)} = 0) 
\leq Ck^{m+1}.$$
But for the event $E^{\rm{odd},k}_{\eps_1, \eps_2}$ to hold, there have to be at least $k$ CLE$_4$ loops surrounding both $z_1$ and $z_2$. But by  
\cite[Lemma 21]{Dubedat19DoubleDimers}, we know that the probability of this event decays superexponentially in $k$. This implies that
$$\sum_{k \geq 1} k^m\P(E^{\rm{odd},k}_{\eps_1, \eps_2}) < \infty $$
and we conclude that 
$$\vert\sum_{n \in \Z}a_{m,n}
\P(E_{\eps_1,\eps_2}^{2\lambda n}) - 
\sum_{n \in \Z}a_{m,n}
\P(\hat I_{\eps_1,\eps_2}^{2\lambda n})\vert 
\leq \widetilde C (\eps_1\eps_2)^{1/4}, $$
as desired.

So finally, we get the following asymptotic.

\begin{prop}
Fix $m\geq 2$.
Then, as $\max(\eps_1,\eps_2)\to 0$,
The probability that an $m$-th nested CLE$_4$ gasket in $D$ intersects both disks
$\D(z_1,\eps_1)$ and $\D(z_2, \eps_2)$
behaves like
$$
\sum_{n\in \Z} a_{m,n}
\P(E_{\eps_1,\eps_2}^{2\lambda n})
+ O((\eps_1\eps_2)^{1/4}).
$$
\end{prop}

We hence conclude the following formula for the probability:

\begin{theorem}\label{thm:kthgasket}
Let $D$ be a simply-connected domain, $z_1 \neq z_2$ two distinct points in $D$.
Fix $m \geq 2$ and let $E_{\eps_1,\eps_2}^{\rm m^{th}}(z_1, z_2)$ denote the event that an $m$-th nested CLE$_4$ gasket in $D$ intersects both disks
$\D(z_1,\eps_1)$ and $\D(z_2, \eps_2)$.
We then have that 
$$(\eps_1 \eps_2)^{-1/8}\P(E_{\eps_1,\eps_2}^{\rm m^{th}}(z_1, z_2))+
o(1)$$equals 
$$\Ctw (\CSLE)^2\,
\dfrac{\CR(z_1,D)^{-1/8}\CR(z_2,D)^{-1/8}}
{\theta_2(q^{1/4})\sqrt{|\log q|}}\sum\limits_{\substack{|n| \geq m-1 \\ n = m \mod 2}}2^{m-1}(-1)^{\frac{n-m}{2}}\left( \binom{\frac{n+m}{2}}{\frac{n-m}{2}}+\binom{\frac{n+m}{2}-1}{\frac{n-m}{2}-1}\right)\hat q^{2n^2},
$$
where $\hat q$ is the complementary nome to nome $q$,
which in turn, by Lemma \ref{lem:parametr},
is uniquely defined via 
$$\exp\Big(\pi G_D(z_1,z_2)\Big) 
= \dfrac{\theta_3(q^{1/2})}{\theta_2(q^{1/2})}.$$
\end{theorem}

\subsection{Two-point correlations on the AT critical line}

In this section we give CLE$_4$-based calculations for a collection of iterated CLE$_4$ gaskets, obtained from sampling in turn some two valued sets $A_{- 2\lambda, 2\sqrt{g}\lambda - 2\lambda}$ and then again the CLE$_4$, 
for $g > 1$. 
The relevant constructions might look a bit mysterious at first sight, 
but are actually natural from the following
two points of view.
\begin{itemize}
\item First, the CLE$_4$ gaskets in these constructions correspond to the conjectural FK-type representation of the single spin of the Ashkin-Teller model on the whole critical line of 2nd order phase transitions with wired boundary conditions. This can be motivated by following \cite{saleur1988correlation}, or the mathematical literature giving explicit random cluster representation of the AT model and its relation to the 6V height function starting from \cite{pfistervelenik, glazmanpeled, lis2022AT}. In particular, this implies the surprising conjecture that there should be a CLE$_4$-based FK representation of the continuum limit of the Ising spins, which we support by a 2-point correlation calculation. In fact, in a concurrent and independent work the authors in \cite{AlcadeHeeneyLis2026} managed to go further and prove such a continuum FK representation for the Ising model!
\item Second, these sets are also very natural from the perspective of the level lines of the compactified GFF - heuristically we are looking at all CLE$_4$ gaskets that form the dual of the sign excursions for the compactified GFF modulo $2\sqrt{g}\lambda \Z$. We hope to return to this interesting geometric observation in a subsequent work.
\end{itemize}
Due to some technical issues with the current strategy we are able to prove it only in the region $\{2-1/\sqrt{2} < \sqrt{g} < 1+ 1/\sqrt{2}\}\cup\{g \geq 4\}$, but we believe it should hold for all $g > 1$.

\begin{theorem}\label{thm:twopointAT}
Let $D$ be a simply-connected domain. Let $g\in\left((2-1/\sqrt{2})^2,(1+ 1/\sqrt{2})^2\right)\cup[4,+\infty)$ and let $E_{\eps_1,\eps_2}^{AT,g}$ denote the event that one of the CLE$_4$ gaskets obtained from the following iteration intersects both $\D(z_1,\eps_1)$ and 
$\D(z_2, \eps_2)$.
\begin{enumerate}
    \item We sample a CLE$_4$ gasket in $D$ to obtain $A_1$; 
\item Inside each connected component of $D \setminus A_1$, we sample a random gasket with the law of the two-valued set $A_{-2\sqrt{g}\lambda + 2\lambda, 2\lambda}$ to obtain $A_2$;
\item we then again sample CLE$_4$ inside each connected component of $A_2$ and iterate this way: on odd steps we sample CLE$_4$ gaskets, on even $A_{-2\sqrt{g}\lambda + 2\lambda, 2\lambda}$ gaskets.
\end{enumerate}
We then have that 
\begin{equation}\label{eq:twopointAT}
\lim_{\max(\eps_1, \eps_2) \to 0}(\eps_1\eps_2)^{-1/8}\P(E_{\eps_1,\eps_2}^{AT,g}) = 
C_{AT,g}
\CR(z_1,D)^{-1/8}\CR(z_2,D)^{-1/8}
\frac{\theta_3(q^{1/(2g)})}{\theta_2(q^{1/4})},
\end{equation}
where the nome $q$ is fixed by $\exp(\pi G_D(z_1,z_2)) = \frac{\theta_3(q^{1/2})}{\theta_2(q^{1/2})}$ and 
$C_{AT, g} = \frac{\pi^{1/4} (\CSLE)^2}{2^{5/4}\sqrt{g}}$

\end{theorem}

Plugging in $g = 4$, we obtain Theorem \ref{thm:main1} from above. Plugging in $g = 2$, we obtain the correlation of the Ising two-point function, Theorem \ref{Thm TVS Ising}. The iteration in the theorem only gives laws of the random gaskets, but seeing them in the coupling with a GFF is even more natural and is actually crucial to observe the rewriting as sums over different boundary conditions. In that context the iteration can be described as follows:
\begin{itemize}
    \item We explore the outermost CLE$_4$ in the domain; 
\item next, inside each loop we explore a two-valued set $A_{- 2\lambda, 2\sqrt{g}\lambda - 2\lambda}$ or $A_{-2\sqrt{g}\lambda+  2\lambda, + 2\lambda}$ depending on whether the label of the relevant CLE$_4$ loop was $2\lambda$ or $-2\lambda$ respectively;
\item we then again explore CLE$_4$ gaskets inside each of the resulting loops, and continue this way;
\item at each odd step we sample a CLE$_4$ gasket, and after each even step we make sure to arrive at height $2\sqrt{g}\lambda\Z$.
\end{itemize}

The above calculation motivates the following conjecture for the spin to spin correlations of the Ashkin-Teller model on the critical line. 

\begin{defn}[AT model]
The Ashkin-Teller model on a box $\Lambda_n$ is a probability measure on configurations $(\sigma_1, \sigma_2): \Lambda_n \to \{-1,+1\}^2$,
with density proportional to 
$$\exp\left(\sum_{x \sim y}J\left(\sigma_1(x)\sigma_1(y) + \sigma_2(x)\sigma_2(y)\right)+\sum_{x\sim y}U\sigma_1(x)\sigma_1(y)\sigma_2(x)\sigma_2(y)\right).$$
\end{defn}

At $U = 0$ this gives two decoupled Ising models, at $J = U$ the 4-Potts model, etc. 
There is a self-dual line of second order phase transitions given by \cite{saleur1988correlation}
\begin{equation}\label{eq:ATline}
\sinh(2J) = \exp(-2U)\text{ for $J \geq 0$ and $J \geq U$}.
\end{equation}
The parameter on the critical line can be equivalently encoded by a parameter $g \in (4/3,4]$ given by \cite{saleur1988correlation, Yang}
\begin{equation}\label{eq:g}
g = \frac{8}{\pi}\sin^{-1}\left(\frac{\coth 2J}{2}\right).
\end{equation}
Here $g = 2$ corresponds to the Ising model, $g = 4$ to 4-Potts that corresponds to the endpoint $U = J$. In fact, from the point of view of the Gaussian continuum description, one may continue the parameter further down to $g=1$, corresponding to the KT point of the XY model; however, this no longer corresponds to real lattice parameters $(J,U)$ in \eqref{eq:ATline}. \footnote{Notice that in fact the geometric description given by Theorem \ref{thm:twopointAT} also moves up to $g = \infty$.}

In this context we have the following conjectures for the FK representations of the scaling limit of the single spins $\sigma_1$. This conjecture has also been put forward simultaneously via a different line of thought by \cite{AlcadeHeeneyLis2026}, where it is in fact extended to both AT spins, to both wired and free boundary conditions and is  importantly also proved in the case of the XOR-Ising model.

\begin{conj}[FK representation of the continuum limit of AT spins]\label{conj:ATFK}
Consider the AT model on the self-dual line with parameters $(J,U)$ satisfying \eqref{eq:ATline}.
Let $g \in (4/3, 4]$ be given by \eqref{eq:g}. 

Then an FK representation of the continuum limit of 
$\sigma_1$ with wired boundary conditions is given by 
$\sum_{C \in CLE_4^g}\nu(C)s_C,$
where we sum in the decreasing order of gasket diameter over all the CLE$_4$ gaskets appearing in the iteration described in Theorem \ref{thm:twopointAT}, the $\nu(C)$ are Minkowski content measures \footnote{In fact, Minkowski content measures have not been directly shown to exist for CLE$_4$, but several constructions of natural measures on CLE$_4$ gasket exist and they should all be equal to this Minkowski measure (up to absolute constants that can be determined) \cite{SSV, miller2024existence, AlcadeHeeneyLis2026}.} and $s_C$ are i.i.d. Rademacher signs.
\end{conj}

Before explaining the modifications necessary to obtain the proof, let us show the calculation demonstrating that our formula for $g = 2$ is up to a universal constant equal to the formula obtained in the seminal paper \cite{CHIcorr}.

Indeed, rewriting the Formula (1.2) in \cite{CHIcorr} in a conformally invariant way we have that
$$\langle \sigma(z_1)\sigma(z_2)\rangle^+ = \frac{\sqrt{2}}{\CR(z_1, D)^{1/8}\CR(z_2, D)^{1/8}} \sqrt{\exp(-\pi G_D(z_1,z_2)) + \exp(\pi G_D(z_1,z_2))}.$$
Using Lemma \ref{lem:parametr} and identities \eqref{eq:thetasquares}, \eqref{eq:thetadupl} we can rewrite this as
$$\langle \sigma(z_1)\sigma(z_2)\rangle^+ = \frac{2^{3/4}}{\CR(z_1, D)^{1/8}\CR(z_2, D)^{1/8}} \frac{\theta_3(q^{1/4})}{\theta_2(q^{1/4})},$$
matching the formula \eqref{eq:twopointAT} for $g = 2$ up to an absolute constant.

\subsubsection{Proof of Theorem \ref{thm:twopointAT}}

The proof of the theorem runs again along the same lines as the proof of Theorem \ref{thm:main1}, so let us just see how the key ingredients need to be modified. 

First, the computation equivalent to Proposition \ref{prop:calc2ptsum} is clear - it is the same calculation over just different boundary conditions - we now sum over boundary conditions $2\sqrt{g}\lambda \Z$.

The relevant boundary conditions are fixed by Proposition \ref{prop:CLEcoupling2} and the following random walk lemma. Here the random walk again pertains just to the height of the CLE$_4$ gaskets that arrive at each odd step. The difference is that if at step $n$ we are at height $h$, then the next CLE$_4$ gasket can arrive both at the same height or also at the height $h \pm 2\sqrt{g}\lambda$. As we saw in Proposition \ref{prop:CLEcoupling2}, if we denote the probability of staying at the same height by $p_g$, then it is easy to see that $p_g = 1-1/\sqrt{g}$, 
and further the walk is always symmetric, i.i.d. 

\begin{lemma}\label{lem:walkG}
Let $\P^W_{n}$ be the probability measure on lazy random walks $(S_k)_{k \geq 0}$, with independent steps, 
that start from $n$ and at each step go up or down with probability $p/2$, 
and stay at their current position with probability $1-p$. 
Now consider the infinite measure $\nu^W := \sum_{n \in \Z} \P^W_n$. Then $\nu^W(S_m = 0) = 1$ for all $m \geq 0$.
\end{lemma}

The proof is immediate again: we just notice that after one step the mass at each point $x$ is still exactly $1$. Observe that the case $p = 1$ gives the Lemma \ref{lem:walk1}.

Second, the equivalent of Proposition \ref{prop:2ptNestedCLE} follows exactly in the same lines, given the generalizations of the key lemmas. 
Let us start with the equivalent of Lemma \ref{lem:nointersect}.

\begin{lemma}\label{lem:nointersectAT}
Let $\widetilde F_{\eps_1,\eps_2}$ denote the event that at least two different odd CLE$_4$ gaskets of the family of CLE$_4$ gaskets constructed in Theorem \ref{thm:twopointAT} intersect both 
$\D(z_1, \eps_1)$ and $\D(z_2, \eps_2)$. Then
\begin{itemize}
    \item Let $E_{\eps_1,\eps_2}^{\rm{odd},k}(z_1,z_2)$ denote the event that the $k$-th odd nested CLE$_4$ gasket intersects 
    $\D(z_1,\eps_1)$ and $\D(z_2,\eps_2)$. Then there is a constant $C > 0$
    such that for all $k \geq 1$,
    $$\P(\widetilde F_{\eps_1,\eps_2} \cap E_{\eps_1,\eps_2}^{\rm{odd},k}(z_1,z_2)) \leq C(\eps_1\eps_2)^{1/4}\left(\P(E_{\eps_1,\eps_2}^{\rm{odd},k}) + \1_{k \geq 2}\P(E_{\eps_1,\eps_2}^{\rm{odd},k-1})\right).$$
    \item Further, for some $\widetilde C > 0$, 
    for all $\eps_1,\eps_2 > 0$ sufficiently small, 
$$\sum_{k \geq 1}\P(E^{\rm{odd},k}_{\eps_1,\eps_2}(z_1,z_2)) < \widetilde C\P(E_{\eps_1,\eps_2}^{\rm odd}(z_1,z_2)).$$
    \item In particular, $\P(\widetilde F_{\eps_1,\eps_2})$ is bounded by $C(\eps_1\eps_2)^{1/4}\P(E_{\eps_1,\eps_2}^{\rm odd}(z_1, z_2))$ for some $C > 0$ as $\max(\eps_1,\eps_2)\to 0$. 
\end{itemize}
Further, the same holds when we replace odd by even everywhere.
\end{lemma}
We will explain how to reduce this result to Lemma \ref{lem:nointersect}.
\begin{proof}
Observe that for two consecutive odd (or even CLE$_4$) gaskets to touch both $\D(z_1,\eps_1)$ and 
$\D(z_2, \eps_2)$, the following has to happen: 
\begin{enumerate}[i.]
    \item a loop in one of the iterated CLE$_4$ gaskets has to hit both disks, 
    \item then inside this loop, a loop of the $A_{-2\lambda, 2\sqrt{g}\lambda - 2\lambda}$ (or of the symmetric one) has to hit both disks, 
    \item then again one CLE$_4$ loop has to hit both, 
    \item and yet another $A_{-2\lambda, 2\sqrt{g}\lambda - 2\lambda}$ loop has to hit both balls. 
\end{enumerate}
Now we can apply Lemma \ref{lem:tech1} to conclude that the  event with these bullet points, conditionally on the first loop that touches both disks, is less probable than the event that a CLE$_4$ loop inside the first loop touches again these disks, i.e. the event of Lemma \ref{lem:nointersect}. 

Indeed, let us consider loops $A_{-2\lambda, 2\sqrt{g}\lambda - 2\lambda}$ discovered in the second step. We can now apply Lemma \ref{lem:tech1} with $D_1$ being the union of the interiors of such loops and $D_2$ being the original domain to bound the probability of the described event by the probability that just the next CLE$_4$ loop hits both balls. Hence we conclude that all the probabilities in the lemma can be bounded by those of Lemma \ref{lem:nointersect}. 
\end{proof}

Finally, we also need to modify Lemma \ref{lem:oddeven} to deal with consecutive CLE$_4$ loops in this iteration of loops. Notice that for $g = 4$ this is clear because the construction becomes exactly the construction of the nested CLE$_4$. For $g > 4$ it follows from \ref{lem:oddeven} by monotonicity arguments. However, the case $1 < g < 4$ needs an actual modification of the previous argument and it is here where we will need to restrict the parameter $g$.

\begin{lemma}\label{lem:oddevenAT}
Let $g\in \left((2-1/\sqrt{2})^2,(1+ 1/\sqrt{2})^2\right)\cup [4,+\infty)$, 
and let $\widetilde F^{\rm{two}}_{\eps_1, \eps_2}$ denote the event that in the iteration of Theorem \ref{thm:twopointAT} there are two CLE$_4$ gaskets that intersect both $\D(z_1, \eps_1)$ and $\D(z_2, \eps_2)$. 
Then, for every $g$ in the above-mentioned regime there is some $\beta_g > 0$ such that 
$\P(\widetilde F^{\rm{two}}_{\eps_1, \eps_2}) \leq C\min(\eps_1, \eps_2)^{1/4+\beta_g}$.
\end{lemma}

\begin{proof}
As mentioned, the case $g \geq 4$ comes from the case $g = 4$ (i.e. Lemma \ref{lem:oddeven} and monotonicity of TVS - $A_{-a,b}$ for $a, b \geq 2\lambda$ contains 
$A_{-2\lambda, 2\lambda}$, that is the CLE$_4$ gasket \cite{ASW}).

In the other regime, we will proceed like before and the result comes down to bounding the probability of the event that there is at least one CLE$_4$ gasket in this new iteration that hits both disks. 

We can use the same argument as in Lemma \ref{lem:oddeven}, but only after a small modification: because of the asymmetry of the random walk we will need to bound our walk by a symmetric walk. 

As the iteration of the loops itself was different, we also have to use a slightly different construction: 
\begin{enumerate}
\item We again start with $A_1 := A_{-\lambda, \lambda}$.
\item If a loop of $A_1$ touches both $\D(z_1, \eps_1)$ and 
$\D(z_2, \eps_2)$ we stop.
\item If a loop of $A_1$  separates them in the sense that no connected component of $\D \setminus A_1$ intersects both $\D(z_1, \eps_1)$ and $\D(z_2, \eps_2)$ we also stop.
\item Otherwise, we construct $A_2$ by sampling 
$A_{-\lambda, \lambda}$ only inside the connected component of $\D \setminus A_1$ containing entirely at least one of the disks.
\item We iterate the construction in the following way: if the loop of interest has a label that takes values either in $2\sqrt{g}\lambda \Z$ or $2\sqrt{g}\lambda \Z\pm \lambda$, we sample $A_{-\lambda, \lambda}$. Further, 
\begin{itemize} 
\item if the label takes values in $2\sqrt{g}\lambda \Z + 2\lambda$, we sample $A_{2\sqrt{g}\lambda-4\lambda, 2\sqrt{g}\lambda-2\lambda}$,
\item if the label takes values in $2\sqrt{g}\lambda \Z - 2\lambda$, we sample $A_{-2\sqrt{g}\lambda+2\lambda, -2\sqrt{g}\lambda+4\lambda}$.
\end{itemize}
The stopping conditions remain the same at each iteration. We call the resulting almost sure stopping time $T$. 
\end{enumerate}
As before one of the main ingredients is to show the following control on the stopping time.
\begin{claim}\label{claim:stopAT}
The number of iterations $N$ corresponding to the stopping time $T$ is bounded by a random variable that has super-exponential tails. In particular $M_T$ is integrable and the Optional Stopping Theorem holds at the time $T$. 
\end{claim}
The proof of the claim is exactly the same - the fact that half of the time we do not use SLE$_4(-1,-1)$ but rather SLE$_4(\rho, -2-\rho)$ for some $\rho \in (-2,0)$ does change the non-separation probabilities $p_n$ by changing the law of the random variables $X_i$, yet they still remain i.i.d. and a.s. positive.

However, to finalize the argument one also needs to revisit the martingale calculation, i.e. bounding
$$\E(\Re(M_T(z_1)))=\E\left[\CR(z_1, D\setminus \hat{A}_T)^{-\alpha^2/2}\Re(\exp(i\alpha\sqrt{2\pi} h_{T}(z_1)))\right]$$
from below. A key step above was to justify the following equality
$$\E(\Re(M_T(z_1))) = 
\E\left[\cos(\alpha \pi/2)^N
\E(\CR(z_1, D\setminus \hat{A}_T)
^{-\alpha^2/2}|N)\right],$$
after which we could lower the bound the second expectation inside using the event that the set $\hat A_T$ comes at least $\eps-$close to $z_1$. However, to obtain this expression we used the fact that the walk is symmetric and thus the final conformal radius is independent of the walk. This is no longer the case for the construction above!

A way around is to notice that we can instead run the argument also on a symmetric exploration $\check{A}_t$, constructed similarly, but where every time we take a step of the form $E$ in the original construction, 
we take $A_{-a\lambda,a\lambda}$ in the auxiliary one, where $a = \max(4-2\sqrt{g},2\sqrt{g}-2)$. This auxiliary exploration can be coupled to the original one by monotonicity of TVS such that the conformal radius $\CR(z_1, D\setminus \check{A}_T)$ in the auxiliary one at the step $T$ of the original construction satisfies $\CR(z_1, D\setminus \check{A}_T) \leq \CR(z_1, D\setminus \hat A_T)$. For this auxiliary construction, the event $F_\eps$ that the original conformal radius is less than $\eps$ is independent of the walk and we can use any $\alpha$ with 
$\alpha a < 1$ to obtain the bound $\P(F_\eps) \leq C\eps^{\alpha^2/2}$. The condition of the lemma now becomes clear: we want $\alpha^2/2 > 1/4$
which gives us $\alpha^2 > 1/2$, giving an upper bound on the range of $a$.
\end{proof}

\bibliographystyle{alpha}	
\bibliography{biblio}

\end{document}